\documentclass[12pt]{article}

\usepackage{tikz,amsthm,amsmath,amstext,amssymb,epsfig,euscript, mathrsfs, dsfont,pspicture,multicol,graphpap,graphics,graphicx,times,enumerate,subfig,sidecap,
wrapfig,color}
\usepackage{enumerate}

\newcommand\res{\hbox{ {\vrule height .22cm}{\leaders\hrule\hskip.2cm} } }

\def\R{{\mathbb{R}}}
\def\H{{\mathscr{H}}}
\def\sm{\setminus}
\def\wt{\widetilde}

\def\ms{\medskip}

\def\W{{\Bbb W}}
\def\CD{C_{\delta}}

\def\1{{\mathds 1}}
\usepackage{hyperref} 
\hypersetup{backref,  pdfpagemode=FullScreen,  colorlinks=true} 
\usepackage{esint}

\def\bB{{\mathbb{B}}}

\def\bR{{\mathbb{R}}}

\def\bZ{{\mathbb{Z}}}

\def\bN{{\mathbb{N}}}

\def\cB{{\mathscr{B}}}

\def\cD{{\mathscr{D}}}

\def\cF{{\mathscr{F}}}
\def\cG{{\mathscr{G}}}

\def\cR{{\mathscr{R}}}
\def\cS{{\mathscr{S}}}

\def\cV{{\mathscr{V}}}

\renewcommand{\d}{{\partial}}

\def\Tan{\text{Tan}}
\def\BMO{\text{BMO}}
\def\Lip{\text{Lip}}

\def\diam{\mathop\mathrm{diam}}

\def\dist{\mathop\mathrm{dist}}

\newcommand{\ps}[1]{\left( #1 \right)}

\newcommand{\av}[1]{\left| #1 \right|}

\def\barintgerm_#1{\mathchoice
{\mathop{\vrule width 6pt height 3 pt depth -2.5pt
\kern -8.8pt \intop}\nolimits_{#1}}%
{\mathop{\vrule width 5pt height 3 pt depth -2.6pt
\kern -6.5pt \intop}\nolimits_{#1}}%
{\mathop{\vrule width 5pt height 3 pt depth -2.6pt
\kern -6pt \intop}\nolimits_{#1}}%
{\mathop{\vrule width 5pt height 3 pt depth -2.6pt \kern -6pt
\intop}\nolimits_{#1}}}

\theoremstyle{plain}
\newtheorem{theorem}{Theorem}

\newtheorem{corollary}[theorem]{Corollary}

\newtheorem{lemma}[theorem]{Lemma}
\newtheorem{proposition}[theorem]{Proposition}

\theoremstyle{definition}
\newtheorem{example}[theorem]{Example}
\newtheorem{definition}[theorem]{Definition}
\newtheorem{remark}[theorem]{Remark}
\newtheorem{remarks}[theorem]{Remarks}

\numberwithin{equation}{section}
\numberwithin{theorem}{section}

\begin{document}

\def\atwid{\tilde{\alpha}}
\def\wklim{\mathop\mathrm{wk--}\lim}
\def\dsh{\mathop\mathrm{--}}

\title{Wasserstein Distance and the Rectifiability of Doubling Measures: Part I}
\author{Jonas Azzam
Guy David 
and Tatiana Toro\footnote{The first author was partially supported by NSF RTG grant 0838212.
The second author acknowledges the generous support of
the Institut Universitaire de France, and of the ANR
(programme blanc GEOMETRYA, ANR-12-BS01-0014).
The third author was partially supported by an NSF grants DMS-0856687 and DMS-1361823, a grant from the  
Simons Foundation (\# 228118) and the Robert R. \& Elaine F. Phelps Professorship in 
Mathematics.}}

\maketitle

\ms\noindent{\bf Abstract.}
Let $\mu$ be a doubling measure in $\R^n$. We investigate quantitative relations
between the rectifiability of $\mu$ and its distance to flat measures.
More precisely, for $x$ in the support $\Sigma$ of $\mu$ and $r > 0$, 
we introduce a number $\alpha(x,r)\in (0,1]$ that measures,
in terms of a variant of the $L^1$-Wasserstein distance,
the minimal distance between the restriction of $\mu$ to $B(x,r)$ and 
a multiple of the Lebesgue measure on an affine subspace that meets $B(x,r/2)$.
We show that the set of points of $\Sigma$ where 
$\int_0^1 \alpha(x,r) {dr \over r} < \infty$ can be decomposed into rectifiable pieces of various dimensions. We obtain additional control on the pieces and the size of $\mu$ when we assume that 
some Carleson measure estimates hold.

\ms\noindent{\bf R\'esum\'e en Fran\c cais.}
Soit $\mu$ une mesure doublante dans $\R^n$. On \'etudie des relations quantifi\'ees
entre la rectifiabilit\'e de $\mu$ et la distance entre $\mu$ et les mesures plates. 
Plus pr\'ecis\'ement, on utilise une variante de la $L^1$-distance de Wasserstein
pour d\'efinir, pour $x$ dans le support $\Sigma$ de $\mu$ et $r>0$,
un nombre $\alpha(x,r)$ qui mesure la distance minimale entre 
la restriction de $\mu$ \`a $B(x,r)$ et une mesure de Lebesgue sur un sous-espace affine 
passant par $B(x,r/2)$. On d\'ecompose l'ensemble des points $x\in \Sigma$ tels que
$\int_0^1 \alpha(x,r) {dr \over r} < \infty$ en parties rectifiables de dimensions diverses,
et on obtient un meilleur contr\^ole de ces parties et de la taille de $\mu$ quand les 
$\alpha(x,r)$ v\'erifient certaines conditions de Carleson. 

\ms\noindent{\bf Key words/Mots cl\'es.}
Rectifiability, tangent measures, doubling measures, Wasserstein distance.

\tableofcontents

\section{Introduction}
\subsection{Statement of Results}

In this paper we are concerned with the question of rectifiability of doubling measures. More precisely we
explore quantitative conditions which imply that a doubling measure in Euclidean space is rectifiable.
Recently this question has been addressed by several authors in the context of Ahlfors regular measures \eqref{1.4}
(see \cite{CGLTolsa}, \cite{Tolsa-uniform-rectifiability} and \cite{Tolsa-mass-transport}). Roughly speaking Ahlfors regular measures 
have a prescribed polynomial growth rate while doubling measure do not. While our work can be seen as an extension of 
their results it is important to note that our proofs are very different as we lack the technology available in the 
Ahlfors regular setting. 

The question of uniform rectifiability of doubling measures can be understood as a geometric version
of whether a doubling measure supported in Euclidean space is an $A_\infty$ weight with respect to 
Lebesgue measure. The later question has been the subject of intensive research, see for example
\cite{BR}, \cite{B}, \cite{FKP}, \cite{GN1}, \cite{GN2}, \cite{NRTV} and \cite{TT}. The relationship between these two questions
is apparent in Section \ref{examples}.

In this paper we study doubling measures which are well approximated by flat measures. This notion of approximation
is expressed in terms of a minor variant of the $L^1$-Wasserstein distance between
the scaled restriction of $\mu$ to balls and the scaled restriction of $d$-dimensional Hausdorff
measure on disks of dimension $d$. In a subsequent paper (see \cite{Part2}) we
consider the extent to which the self-similarity properties
of $\mu$, still measured in terms of the $L^1$-Wasserstein distance above yield rectifiability of the measure.

In this paper, $\mu$ denotes a Radon measure on $\R^n$
(i.e., a locally finite positive Borel measure), 
and 
\begin{equation}
\Sigma_\mu = \Sigma=\left\{x\in\R^n:\, \mu(B(x,r)) > 0\hbox{  for all  }r > 0\right\}\label{1.0}
\end{equation}
denotes its support. Here and below, $B(x,r)$ denotes the open ball centered at $x$ and with radius $r$.

We say $\mu$ is \textit{doubling} when there is a constant $\CD >0$ for which
\begin{equation}\label{1.1}  
\mu (B(x,2r)) \leq \CD \mu(B(x,r)) \mbox{ for all $x\in \Sigma$ and } r>0.
\end{equation}

\begin{definition} \label{t1.1}
Let $d\in [0,n]$ be an integer.
We say that $\mu$ is \textit{$d$-rectifiable} if 
it is absolutely continuous with respect to $\H^d$
and its support $\Sigma$ may be covered, up to a set of $\mu$-measure zero, 
by countably many $d$-dimensional Lipschitz graphs.
\end{definition}

Here and below, $\H^d$ denotes the $d$-dimensional Hausdorff measure
(see \cite{Federer} or \cite{Mattila}) which we renormalize so that
\begin{equation}\label{1.2}
\H^d(\R^d \cap B(0,1)) = 1.
\end{equation}

The fact that $\mu$ is rectifiable does no imply that $\Sigma$ is rectifiable. In fact $\Sigma$ may contain a purely unrectifiable subset $Z$, with $\mu(Z) = 0$ but $\H^d(Z) > 0$.

\ms
The rectifiability of a set or measure is a coveted property from several different facets of analysis and geometric measure theory. In particular, many basic analytic properties and 
tools known to hold for smooth manifolds carry over to these sets (e.g., Rademacher's theorem, area and co-area formulas). Other properties, such as the boundedness of 
certain singular integral operators, or the absolute continuity of harmonic measure 
with respect to $\H^{n-1}$, are clearly linked to the notion of rectifiability,
but are more quantitative in nature. These led to measuring rectifiability with 
certain quantities, and in particular  the so-called P. Jones $\beta$-numbers
\begin{equation} \label{1.3}
\beta_{d,q}(x,r)=\inf_{V\in A(d,n)}\ps{\fint_{B(x,r)} 
\ps{\frac{\dist(z,V)}{r}}^{q} d\mu(z)}^{\frac{1}{q}},
\end{equation}
where $A(d,n)$ denotes the set of affine $d$-planes in $\bR^{n}$. 
The original Jones $\beta$-numbers, which were introduced 
in \cite{Jones-TSP}, correspond to $\beta_{1,\infty}$ 
(and in fact used a supremum, as there was no measure); the 
generalized form above was introduced in \cite{DS} and \cite{of-and-on}, 
where the authors studied relations between uniform rectifiability properties 
of Ahlfors $d$-regular measures, and the boundedness of singular integral
operators. Recall that $\mu$ is {\it Ahlfors $d$-regular} when there is a 
constant $C_{ar}$ such that
\begin{equation} \label{1.4}
C_{ar}^{-1} r^{d}\leq \mu(B(x,r))\leq C_{ar} r^{d}
\ \text{ for $x\in \Sigma$ and $r>0$.}
\end{equation}
In this context, the $\beta$-numbers are quite powerful, in particular 
because we already know that $\mu$ is also strongly $d$-dimensional.

In this paper, we consider doubling measures $\mu$
that are not necessarily Ahlfors regular. In this case it is more convenient
to use the following variant of the Wasserstein $L^1$ distance on measures.

\begin{definition} \label{t1.2}
Let $\mu$ and $\nu$ be measures on $\R^n$, whose restrictions to
$\bB:= B(0,1)$ are probability measures. We set
\begin{equation} \label{1.5}
\W_{1}(\mu,\nu):=\sup_{\psi}\av{\int\psi d\mu-\int\psi d\nu},
\end{equation}
where the supremum taken is over all the functions 
$\psi : \R^n \to \R$ that are $1$-Lipschitz (i.e., such that $|\psi(x)-\psi(y)| \leq |x-y|$
for $x,y \in \R^n$), and supported on $\bB$.
\end{definition}

Thus $\W_{1}(\mu,\nu)$ only measures some distance between the restrictions 
to $\bB$ of $\mu$ and $\nu$.
This quantity is similar to the usual $L^1$-Wasserstein distance, which by the Kantorovich duality theorem, has the same definition as $\W_{1}$ except that the infimum ranges over all $1$-Lipschitz nonnegative functions in $\bB$.
This creates a minor difference of control near the unit sphere, 
but using the $L^1$-Wasserstein distance instead of $\W_1$ would yield the same results. 
We could also have used a slightly smoother version of $\W_1$; 
see Section 5 of \cite{Part2}.
See \cite{Villani} for additional information about the Wasserstein distance, 
and \cite{Tolsa-mass-transport} for its relationship to $\W_{1}$ and uniform rectifiability.

The idea of using $\W_1$ for doubling measures is not new; it was used
in \cite{Tolsa-uniform-rectifiability}, in connection with singular integrals. 
See also \cite{Tolsa-mass-transport} and \cite{CGLTolsa}.

Here we mostly compare a scaled version of $\mu$ to the restriction of Hausdorff
measures to affine subspaces. For $x\in \Sigma$ and $r > 0$,
define a measure $\mu_{x,r}$ on $\R^n$ by
\begin{equation} \label{1.6}
\mu_{x,r}(A) = {\mu(x + r A)\over \mu(B(x,r))}
\ \text{ for every Borel set } A \subset \R.
\end{equation}
That is, we push forward $\mu$ by a translation and a dilation then normalize 
to make $\mu_{x,r}$ a probability measure on $\bB$.

\begin{definition} \label{t1.3}
Denote by $A(d,n)$ the set of $d$-dimensional affine planes
and by $A'(d,n)$ the set of spaces $V \in A(d,n)$ that meet $B(0,1/2)$.
For $V \in A'(d,n)$, set 
\begin{equation} \label{1.7}
\nu_V = c_V \H^d|_{V} = c_V \1_{V} \H^d,
\ \text{ with }  c_V = \H^d(V\cap \bB)^{-1}.
\end{equation}
That is, we restrict $\H^d$ to $V \cap \bB$ and normalize. 
Notice that 
\begin{equation} \label{1.8}
1 \leq c_V \leq 2^d
\end{equation}
by definition of $A'(d,n)$ and our normalization \eqref{1.2}.
Then set
\begin{equation} \label{1.9}
\alpha_d(x,r) = \inf\big\{ \W_1(\mu_{x,r},\nu_V) \, ; \, V \in A'(d,n) \big\}.
\end{equation}
If $d=0$, then $\nu_V$ is just a Dirac mass somewhere in $B(0,1/2)$.
\end{definition}

\ms
These numbers play a key role in this paper. The restriction to $V \in A'(d,n)$
is not an issue, because the best $V$ should pass near $0$ (since $0 \in \Sigma$
and $\mu$ is doubling). The bound on $c_V$ which appears in \eqref{1.8} guarantees that all
constants are under control.
The advantage of using the $\alpha_{d}(x,r)$ is that they contain more information than 
the $\beta$-numbers: while the latter measure the flatness of the support,
the former also measure the degree to which $\mu$ resembles Hausdorff measure on that plane. Hence, in determining conditions to guarantee 
that a {\it doubling} measure is rectifiable, the $\alpha_d$'s are a natural object to consider.

\begin{remark}\label{tt-r2}
Note that $\alpha_d(x,r)$ is a Borel function of $(x,r)\in\Sigma\times(0,\infty)$.
In the definition \eqref{1.5} of $\W_1(\mu,\nu)$, we can restrict ourselves 
to a countable family $\cD$ of functions $\psi^\ast$. That is, if 
$L$ is the space of $1$-Lipschitz functions $\psi$ that vanish on $\d \bB$,
it is easy to find a countable set $\cD \subset L$, which is dense
in the sup norm, and then \eqref{1.5} stays the same if we restrict 
the supremum to $\psi \in \cD$. Similarly, the infimum in the definition
\eqref{1.9} of $\alpha_d(x,r)$ can be taken in a countable dense
class $\cV$ of $A'(d,n)$. Then 
\begin{equation} \label{7.16}
\alpha_d(x,r) = \inf_{V \in \cV} \sup_{\psi \in \cD} \delta_{x,r}(V,\psi),
\end{equation}
where
\begin{eqnarray} \label{7.17}
\delta_{x,r}(V,\psi) &=& \av{\int_{\bB}\psi d\mu_{x,r}-\int_{\bB}\psi d\nu_V}
\\
&=& \av{\mu(B(x,r))^{-1}\int_{B(x,r)}\psi(x+ry) d\mu(y)-\int_{\bB}\psi d\nu_V}.
\nonumber
\end{eqnarray}
Notice that for each $\psi$, $\int_{B(x,r)}\psi(x+ry) d\mu(y)$ is a continuous
function of $(x,r) \in \Sigma$, while $\mu(B(x,r))$ is a Borel function that does not vanish. 
Thus $\alpha_d(x,r)$ is a Borel function.
\end{remark}

In \cite{Tolsa-uniform-rectifiability}, Tolsa proves that if $\mu$ is Ahlfors $d$-regular, then $\Sigma$ is uniformly rectifiable (of dimension $d$) 
if and only if $\mu$ satisfies the following Carleson condition: 
there is a constant $C \geq 0$ such that
\begin{equation} \label{1.10}
\int_{ B(x,r)}\int_{ 0}^{r}\alpha_{d}(y,t)^{2}\, \frac{dtd\mu(y)}{t} 
\leq C \mu(B(x,r))
\end{equation}
for $x \in \Sigma$ and $r > 0$. 
Let us not define uniform rectifiability for the moment 
(see explanation above Theorem \ref{t1.8}). We would like to emphasize 
that \eqref{1.10} is a characterization (and in particular the exponent $2$
is right). This should be compared with the now more standard fact
that if $\mu$ is  Ahlfors $d$-regular, $\Sigma$ is uniformly rectifiable if and only if
\begin{equation} \label{1.11}
\int_{B(x,r)}\int_{0}^{r}\beta_{d,q}(y,t)^{2}\, \frac{dtd\mu(y)}{t} 
\leq C \mu(B(x,r))
\end{equation}
for $x \in \Sigma$ and $r > 0$, and $q \in [1, {2d \over d-2})$.
See \cite{DS}. 

Since we do not assume that $\mu$ is Ahlfors regular, or
even absolutely continuous with respect to $\H^d$ (we want to get it as a conclusion), 
the $\beta$-numbers (or their bilateral counterparts where you also make
sure that $\Sigma$ has no hole) cannot be enough. In fact there are 
doubling measures supported on $\R$ that are not absolutely continuous
with respect to Lebesgue measure (see Section \ref{examples}).

The first result does not require a priori knowledge of the dimension $d$.
Thus we use the numbers
\begin{equation} \label{1.12}
\alpha(x,r)= \min_{d=0,1,...,n} \alpha_{d}(x,r).
\end{equation}
We assume that they are often small, and get a decomposition of 
$\mu$-almost all of $\Sigma$ into rectifiable sets $\cS_{d}$ of 
various dimensions $d$.

\begin{theorem} \label{t1.5}    
Let $\mu$ be a doubling measure in $\R^n$, denote by  $\Sigma$ its support,
and set
\begin{equation} \label{1.13}
\Sigma_0 = \big\{ x\in \Sigma \, ; \, 
\int_{0}^{1}\alpha(x,r) \frac{dr}{r} < + \infty \big\}.
\end{equation} 
Then there are disjoint Borel sets
$\Sigma_0(d) \subset \Sigma$, $0 \leq d \leq n$, such that 
\begin{equation} \label{1.14}    
\Sigma_0 = \bigcup_{d=0}^n  \Sigma_0(d),
\end{equation}
with the following properties. 
\begin{enumerate}
\item
First, $\Sigma_0(0)$ is the set
of points of $\Sigma$ where $\mu$ has an atom; it is at most countable and
each of its point is an isolated point of $\Sigma$.
\item
For $1 \leq d \leq n$ and $x\in \Sigma_0(d)$, the limit
\begin{equation} \label{1.15}
\theta_d(x) := \lim_{r \to 0} r^{-d} \mu(B(x,r))
\end{equation}
exists, and $0 < \theta_d(x) < \infty$.
\item
For $1 \leq d \leq n$ and $x\in \Sigma_0(d)$, $\Sigma$ has a tangent $d$-plane
at $x$ (with the standard definition \eqref{1.18}). 
Call this plane $W$, and set $W^\ast = W-x$ (the corresponding vector space).
Then $\Tan(x,\mu)= \{c\H^{d}|_{W^\ast} \, ; \, c \geq 0\}$. In addition, the 
measures $\mu_{x,r}$ of \eqref{1.6} converge weakly to $\H^{d}|_{W^\ast}$.
\item Further decompose $\Sigma_0(d)$, $1 \leq d \leq n$,
into the sets
\begin{equation} \label{1.16}
\Sigma_0(d,k) = \big\{ x\in \Sigma_0(d) \, ; \, 
2^k \leq \theta_d(x) < 2^{k+1} \big\}, \ k\in \bZ;
\end{equation}
then each $\Sigma_0(d,k)$ is a rectifiable set of dimension $d$, with
$\H^d(\Sigma_0(d,k) \cap B(0,R)) < \infty$ for every $R > 0$, 
$\mu$ and $\H^d$ are mutually absolutely continuous on $\Sigma_0(d,k)$, 
and $\mu = \theta_d \H^d$ there.
\end{enumerate}   
\end{theorem}

Here $\Tan(x,\mu)$ denotes the space of tangent measures of $\mu$ at $x$ (see Definition \ref{t1.4}).

\begin{remarks}\label{tt-r1}

1. The condition $\int_0^1\alpha(x,r) \frac{dr}{r} < \infty$ ensures that a series converges. We would like to be able to obtain information about $\mu$ and its 
support  by imposing a similar condition on the quantity $\alpha^2(x,r)$ rather than $\alpha(x,r)$. The rationale for this is addressed in the 
comments that follow Theorem \ref{t1.6}. Unfortunately at this point we are not able to do this.

\ms

2. On the set $\Sigma_0(d)$, the measures $\mu$ and $\H^d$ are mutually
absolutely continuous, in the sense that for $A \subset \Sigma_0(d)$,
\[
\begin{aligned}
\H^d(A) = 0 
&\iff \H^d(A \cap \Sigma_0(d,k)) = 0 \text{ for all } k \in \bZ
\\
&\iff \mu(A \cap \Sigma_0(d,k)) = 0 \text{ for all } k \in \bZ
\\
&\iff \mu(A) = 0.
\end{aligned}
\]

3. The density $\theta_d$ allows to compute $\mu$ from $\H^d$. In fact
let $A \subset \Sigma_0(d)$ be bounded, then
\[
\mu(A) = \sum_k \mu(A \cap \Sigma_0(d,k))
= \sum_k \int_{A \cap \Sigma_0(d,k))} \theta_d(x) d\H^d(x)= \int_{A} \theta_d(x) d\H^d(x)
\]
and similarly
\[
\H^d(A) = \sum_k \H^d(A \cap \Sigma_0(d,k))
= \sum_k \int_{A \cap \Sigma_0(d,k))} \theta_d(x)^{-1} d\mu(x)
= \int_A \theta_d(x)^{-1} d\mu(x).
\]
Note that $\H^d(\Sigma_0(d))$ is  not necessarely locally finite,
because it could happen that 
\[
\H^d(\Sigma_0(d) \cap B(0,1)) \geq {1 \over 2}
\sum_k 2^{-k} \mu(\Sigma_0(d,k)\cap B(0,1)) = \infty.
\]
See Examples \ref{t4.3}  and \ref{t4.4}. 

\ms
4. We are mostly interested in the case when
\begin{equation} \label{1.17}
\int_{0}^{1}\alpha(x,r) \frac{dr}{r} < \infty
\ \text{ for $\mu$-almost every } x\in \R^n,
\end{equation}
i.e., when $\mu(\Sigma \sm \Sigma_0) = 0$.
In this case we get a decomposition of the $\Sigma_0(d)$,
$d \geq 1$, into countably many rectifiable pieces $\Sigma_0(d,k)$,
which implies that
\[
\text{ the restriction of $\mu$ to $\Sigma_0(d)$ is $d$-rectifiable,}
\]
as in Definition \ref{t1.1}. 
We do not have much information  $\Sigma \sm \Sigma_0$.
Even when $d=1$, it may happen that $\mu$-almost every point of $\Sigma$
lies on the countably rectifiable set $\Sigma_0(1)$ of dimension $1$,
but $\Sigma$ contains a snowflake of dimension $\delta > 1$.
It could also happen that $\Sigma_0 = \Sigma_0(1)$ and $\H^1(\Sigma_0(1)) < \infty$, but 
$\Sigma$ is a snowflake of dimension $\delta > 1$
(but this second example seems less interesting because some of 
the measure $\mu$ lives on the snowflake). See Example~\ref{t4.4}. 

Our statement is not quantitative. Even if $\Sigma = \Sigma_0 = \Sigma_0(d)$
for some $d$ and $\int_{0}^{1}\alpha(x,r) \frac{dr}{r}$ is bounded on $\Sigma$,
we get no integral bound on the density $\theta_d(x)$, for instance.
This is because $\mu$ may look $d_1$-dimensional at many scales, for a different
$d_1$. See Examples \ref{t4.2} and \ref{t4.3}. 
This problem will be fixed in the next result.

\ms

5. The set $\Sigma_0(0)$ is composed of isolated points of $\Sigma$. While the atoms 
might have an accumulation point in $\Sigma$,
such point would not be an atom.
See Example \ref{t4.1}. 

\ms

6. When we say that $W$ is a tangent plane for $\Sigma$ at $x$,
we mean that
\begin{equation} \label{1.18}
\lim_{r \to 0} \, {1 \over r}\sup\big\{ \dist(y,W) \, ; \, y\in \Sigma \cap B(x,r) \big\}
= 0.
\end{equation}
This is stronger than the fact that $\Sigma_0(d)$ has a tangent plane at $x$, in fact it asserts that all of $\Sigma$ is approaching $W$ not just $\Sigma_0(d)$.
The fact that $\Tan(x,\mu)= \{c\H^{d}|_{W-x} \, ; \, c \geq 0\}$
yields additional information concerning how 
$\mu$ is distributed near $x$.

\ms

7. The constants in the proof of  Theorem \ref{t1.5} do not depend on $n$, 
but just on the dimensions $d$ and the doubling constant $\CD$. 
Thus  a modified version of Theorem \ref{t1.5} should also be valid in a Hilbert space $H$.
That is, if $\mu$ is a doubling measure on (a subset of) $H$, we can define the numbers
$\alpha_d(x,r)$ as above, and $\alpha_d(x,r)$ can only be small when
$2^d \leq 2\CD$ (otherwise, test \eqref{1.5} on a function $\psi$ which is a Lipschitz approximation of 
$\1_{B(0,1/2)}$). Then an analogue version Theorem \ref{t1.5} holds in this setting when we replace $\alpha(x,r)$
by $\alpha_d(x,r)$.

\end{remarks}

\ms
For the next statement we fix an integer dimension $d \in [1,n]$,
and give an integral version of Theorem \ref{t1.5} where
we assume that in some ball $B$ we have a good integral control on the 
$\alpha$-numbers, and then get a large piece $A$ of $\Sigma \cap B$
where we have a good control for $\mu$.
The control is good both because $A$ is bi-Lipschitz-equivalent to a subset of $\R^d$ 
and because $\mu|_A$ is equivalent in size to $\H^d|_A$.

\begin{theorem} \label{t1.6}
For each $d\in\{1,\cdots, n\}$, $\CD \geq 1$, 
$C_1 >0$, and $\gamma > 0$, we can find $L = 
L(n,d,\CD,C_1, \gamma) \geq 0$ such that 
if $\mu$ is a Radon measure satisfying the doubling property \eqref{1.1} (with the constant $\CD$), $B = B(x,r)$ is a ball centered on $\Sigma$ (the support of $\mu$),
and if
\begin{equation} \label{1.19}
\int_{B(x,2r) }\int_{0}^{2r} \alpha_{d}(y,t)\frac{dtd\mu(y)}{t} 
\leq C_1\mu (B).
\end{equation}
Then there is a Borel set $A \subset B$, such that
\begin{equation} \label{1.20}
\mu(B \sm A) \leq \gamma \mu(B),
\end{equation}
there is a $L$-bi-Lipschitz map $f: A \to f(A)\subset \R^d$,
and $\mu\res A$ and  $\H^{d}\res A$ are mutually absolutely continuous.
Moreover
\begin{equation} \label{1.21}
L^{-1} \frac{\H^d(Z)}{r^d} \leq   \frac{\mu(Z)} {\mu(B)}\leq L\frac{ \H^d(Z)}{r^d}
\ \text{ for every Borel set } Z \subset A.
\end{equation}
\end{theorem}

\ms
There is nothing special about $2B$ in \eqref{1.19}; $2B$ could be replaced with
$\lambda B$ for any $\lambda > 1$. In this case $L$ would depend on 
$\lambda$ also. By $L$-bi-Lipschitz on $A$, we just mean that $L^{-1}|x-y| \leq |f(x)-f(y)|
\leq L |x-y|$ for $x, y \in A$.

Note that while for a $d$-Ahlfors regular measure condition \eqref{1.10} only requires 
that $\alpha_d(y,t)^2\frac{d\mu(y)\, dt}{t}$ be a Carleson measure for a general doubling measure 
condition \eqref{1.19} requires that $\alpha_d(y,t)\frac{d\mu(y)\, dt}{t}$ be a Carleson measure.
Although evidence suggests that the correct power is 2, even for doubling measures
currently we are not able to prove this. See \cite{B}, \cite{FKP}, \cite{BR}, \cite{NRTV} and Example \ref{t4.6}.
The main issue is that we do not have enough control on the density 
of $\mu$. We should mention that even when $n=d=1$, the gist of the proof of Theorem \ref{t1.6}
is to control the density as in \eqref{1.21}.

If we assume that $C_1$ in \eqref{1.19} is small enough, the set
$A$ constructed in Theorem \ref{t1.6} can be taken to be contained in a small Lipschitz graph, and on which $\mu$
is almost proportional to $\H^d$. For us, a $\gamma$-Lipschitz graph of 
dimension $d$ is a set of the form
\begin{equation} \label{1.22}
\Gamma = \big\{ x + f(x) \, ; \, x\in V \big\},
\end{equation}
where $V$ is a vector subspace of dimension $d$ of $\R^n$, and
$f : V \to V^\perp$ is $\gamma$-Lipschitz, i.e., $|f(x)-f(y)| \leq \gamma |x-y|$
for $x, y \in V$.

\ms
\begin{theorem} \label{t1.7} 
For each $d\in\{1,\cdots, n\}$, $\CD \geq 1$, 
and $\gamma \in (0,1)$, we can find $C_1 = C_1(d,\CD, \gamma) > 0$, 
such that if $\mu$ is a Radon measure that
satisfies the doubling property \eqref{1.1} and  
$B = B(x,r)$ is a ball centered on $\Sigma$ that satisfies \eqref{1.19} with $C_1$ small enough,
then there is a Borel set $A \subset B$, such that \eqref{1.20} holds,
$A$ is contained in a $\gamma$-Lipschitz graph of dimension $d$,
and 
\begin{equation} \label{1.23}
(1-\gamma) \frac{\H^d(Z)}{r^d} \leq  \frac{\mu(Z)}{ \mu(B)}
\leq (1+\gamma)\frac{ \H^d(Z)}{r^d}
\ \text{ for every Borel set } Z \subset A.
\end{equation}
\end{theorem}

\ms
The following theorem will be deduced from Theorem \ref{t1.7} by a standard
localization argument. Let us first define \textit{uniformly rectifiable sets}.
Let $E \subset \R^n$ be an Ahlfors regular set of dimension $d$; this means that
the restriction of $\H^d$ to $E$ satisfies \eqref{1.4}, or equivalently that there
is an Ahlfors regular measure $\mu$ (as in \eqref{1.4}), such that the support of
$\mu$ is equal to $E$ (if this is the case $\mu$ and $\H^d\res E$ are such that
$C^{-1} \mu \leq \H^d\res E \leq C \mu$). We say that $E$ is
uniformly rectifiable if there are constants $L \geq 1$ and $\eta > 0$
such that, for $x\in E$ and $r > 0$, we can find a Borel set $A \subset E \cap B(x,r)$
such that $\H^d(A) \geq \eta r^d$ and an $L$-bi-Lipschitz
mapping $f : A \to f(A) \subset \R^d$. In the language of \cite{of-and-on}, 
$E$ contains big pieces of bi-Lipschitz images of $\R^d$ (BPBI). 
Many other characterizations of uniform rectifiability exist, see \cite{of-and-on}.
The one above works well in this context.
In fact $\Sigma$, the support of $\mu$, satisfies that following slightly stronger 
property, namely  $\Sigma$ 
contains big pieces of Lipschitz graphs. 
This means that there are constants $\eta > 0$
and $\gamma \geq 0$ such that, for $x\in \Sigma$ and $r > 0$, we can find a Borel set 
$A \subset \Sigma \cap B(x,r)$ such that $\H^d(A) \geq \eta r^d$ and
$A$ is contained in a $\gamma$-Lipschitz graph of dimension $d$ (see \eqref{1.22}).

\ms
\begin{theorem} \label{t1.8} 
For each $d\in\{1,\cdots, n\}$ and $\CD \geq 1$, there is a (small)
constant $C_1 > 0$ such that  if $\mu$ is a doubling measure with support
$\Sigma$ and  \eqref{1.19} holds for all $x\in \Sigma$
and $r > 0$. 
Then $\Sigma$ is a uniformly rectifiable set of dimension $d$, 
which contains big pieces of Lipschitz graphs, $\mu$ and  $\H^d\res\Sigma$ 
are mutually absolutely continuous. Moreover $\mu \in A_{\infty}(\H^d\res\Sigma)$.
\end{theorem}

\ms
Note that the definition of uniform rectifiability includes the Ahlfors regularity
of the restriction of $\H^d$ to $\Sigma$, i.e. $\H^d\res\Sigma$. 
Here $A_{\infty}$ is the Muckenhoupt class. Recall that 
$\mu \in A_{\infty}(\H^d\res\Sigma)$ when both measures are locally finite, one of them is assumed to be doubling 
and there exist constants
$\varepsilon \in (0,1)$ and $\delta \in (0,1)$ such that if
$\mu(A ) \leq \varepsilon \mu(B)$  then  $\H^d(A \cap \Sigma)
\leq \delta \H^d(B \cap \Sigma)$ whenever $B$ is a ball centered
on $\Sigma$ and $A \subset B$ is a Borel set. 
Also, $\H^d\res\Sigma \in A_{\infty}(\mu)$ if and only if
$\mu \in A_{\infty}(\H^d\res\Sigma)$.
See for instance \cite{GarciaCuerva} or \cite{Journe}.

We show that $\Sigma$ satisfies a slightly stronger property than the one stated above.
More precisely we prove that for each $x\in \Sigma$ and each $r>0$ there exits a $\gamma$-Lipschitz function whose graph
$\Gamma$ locally covers $\Sigma\cap B(x,r)$, and $\gamma$ and $r^{-d}\H^d(\Sigma \sm \Gamma)$ can be made small
depending on $C_1$.
In Theorem \ref{t1.8} the Ahlfors regularity and uniform rectifiability constants for $\Sigma$,
as well as the $A_{\infty}$ constants $\varepsilon \in (0,1)$ and $\delta \in (0,1)$,
can be chosen to depend only on $d$, and $\CD$.

While by assuming that  $C_1$ is small in Theorem \ref{t1.8} to get the uniform rectifiability
of $\Sigma$, if we assume that the Carleson condition of Theorem \ref{t1.8} is satisfied with
some large $C_1$ we can only show that $\mu$ is uniformly rectifiable, that is
the conclusion of Theorem \ref{t1.6} holds for every ball $B(x,r)$ centered on $\Sigma$.
We note that $\Sigma$ may fail to be uniformly rectifiable, because 
$\H^d(B(0,1)) = \infty$. See  Example \ref{t4.5}.
 
Theorem \ref{t1.8} follows from Theorem \ref{t1.7}; the difficulty comes from
the fact that we need to control the Hausdorff measure of $\Sigma$.
We can achieve this by taking
$C_1$ small.
 
\subsection{Outline}

In Section \ref{prelim}, we introduce notation and various estimates that will help
for the later proofs. In particular, we show that the numbers $\alpha(x,r)$
control the variations of density (see Subsections \ref{measure} and \ref{goodclients}), 
and the Jones numbers $\beta(x,r)$ (see Subsection \ref{beta}).
 We prove Theorem \ref{t1.5} in Section \ref{proof1}; the main ingredient is
the control of density that was obtained in Section \ref{prelim}.
In Section \ref{examples} we present some examples, 
that mainly illustrate Theorems \ref{t1.5} and \ref{t1.8}.
In Section \ref{proof2} we prove Theorem \ref{t1.7}. In the case when $C_1$ is small
assumption \eqref{1.19} ensures that on a large subset of $\Sigma\cap B(x,r)$ the 
series $\sum\alpha_d(\cdot, 2^{-j})$ converges which yields valuable information when trying to control
the measure $\mu$ and its support $\Sigma$.
Section \ref{proof3} is devoted to the proof of Theorem \ref{t1.6}.
We need a stopping time argument, and so we follow a standard route in
the theory of uniformly rectifiable sets: we present adapted dyadic cubes in
Subsection \ref{cubes}, construct a Lipschitz graph associated
to a stopping time region in Subsection \ref{stop}, control the
number of stopping time regions in Subsection \ref{corona}, and 
construct the desired big piece of bi-Lipschitz image 
in Subsection \ref{s:big-piece}. We conclude in Subsection \ref{density}
with the required absolute continuity and density estimates.
Section \ref{UR} contains the proof of Theorem \ref{t1.8}.

 \subsection{Acknowledgements}

The authors are grateful to Alessio Figalli and Xavier Tolsa for helpful discussions. 
The first author would like to thank IPAM for its hospitality, 
part of this manuscript was written while he was in residence there.

\section{Preliminary Results} \label{prelim}

\subsection{Notation, tangent measures, and atoms}
\label{tgm}

We denote by $B(x,r)$ the open ball of center $x$  and radius $r>0$ 
and $\bB:=B(0,1)$. 
For nonempty sets $E,F\subseteq \bR^{n}$ and $x\in \R^n$, we write
\[\diam E=\sup\{|x-y| \, ; \, x,y\in E\},
\quad \dist(x,F)= \inf\{|x-y| \, ; \, y\in F\}, \]
and
\[
\dist(E,F)=\inf\{|x-y| \, ; \, x\in E,y\in F\}.
\]

We denote by $G(d,n)$ the set of $d$-dimensional subspaces of $\bR^{n}$,  
by $A(d,n)$ the set of $d$-dimensional affine planes in $\bR^{n}$, and by
$A'(d,n)$ the set of $d$-planes $V \in A'(d,n)$ that intersect $B(0,1/2)$.

A function $f: A \subset \bR^{n}\rightarrow U$, with $U \subset \R^m$
for some $m$ is \textit{$L$-Lipschitz} on $A$ if there exists $L>0$ such that
\begin{equation} \label{2.1}
|f(x)-f(y)|\leq L|x-y|  \text{ for } x, y \in A
\end{equation}
and it is {\it $L$-bi-Lipschitz} on $A$ if
\begin{equation} \label{2.2}
\frac{1}{L}|x-y|\leq |f(x)-f(y)|\leq L|x-y|
\text{ for } x, y \in A.
\end{equation}
The smallest $L$ such that \eqref{2.1} holds will be denoted by
$||f||_{\Lip(A)}$, or $||f||_{\Lip}$ if $A$ is clear from the context.

Unless otherwise specified, $\mu$ will denote a Borel regular doubling measure, 
with support $\Sigma$ (see \eqref{1.0}) and doubling constant $CD$ (as in \eqref{1.1}).
In general, $C$ denotes a constant whose value may change from line to line. We
try to be explicit about the dependence of $C$ on various parameters.

\begin{definition}  \cite{Preiss87}\label{t1.4}
For $x\in \Sigma$ (the support of $\mu$), we denote by 
$\Tan(\mu,x)$ the space of {\it tangent measures to $\mu$ at  $x$}.
$\Tan(\mu,x)$ is the set of Radon measures $\nu$
that can be obtained as the weak limit of a sequences
$\{ \mu_k \}$, where for each $k$, $\mu_k(A) = c_k \mu(x+r_k A)$
for $A \subset \R^n$ (a Borel set), with constants $c_k \geq 0$ and radii
$r_k > 0$ that tend to $0$.
\end{definition}

By weak convergence, we mean that 
$\int f d\nu = \lim_{k \to \infty}\int f d\mu_k$
for every continuous function $f$ with compact support. Here $\mu$ is doubling,
so it is not hard to show that $\Tan(\mu,x)$ contains some nontrivial measures.
For $x\in \Sigma$ and $r>0$
Define a measure $\mu_{x,r}$ on $\R^n$ as in \eqref{1.6}
The normalization and the doubling condition yield
\begin{equation} \label{2.5}
\mu_{x,r}(B(0,2^j)) \leq \CD^j
\ \text{ for } j \in \bN.
\end{equation}
Thus given a sequence $\{ r_k \}$
that tends to $0$, \eqref{2.5}  ensures that there is a subsequence for which the $\mu_{x,r_k}$s converge
weakly (as in Definition \ref{t1.4}) to a measure $\mu_\infty$.
Since
$1 \leq \mu_\infty(\overline \bB) \leq \CD$, then $\mu_\infty \neq 0$ 
and $\mu_\infty \in \Tan(\mu,x)$.

\begin{remark} \label{t2.1}
If $\mu$ is a doubling measure all its tangent measures are obtained by the process above. In fact, if  $\sigma \in \Tan(\mu,x)$ then by Definition \ref{t1.4}, $\sigma$ is a weak limit of measures $\sigma_k$,
where for each $k$, $\sigma_k(A) = c_k \mu(r_k A+x)$, with 
$c_k \geq 0$ and a sequence $\{ r_k \}$ that tends to $0$. Replacing
$\{ r_k \}$ with
 a subsequence for which the $\mu_{x,r_k}$ converge weakly to
 a limit $\mu_{\infty}$, as above, we note that since $\sigma_k = c_k\mu(B(x_k,r_k))\mu_{x,r_k}$ then
 $\sigma=c\mu_\infty$ where 
 $c=\lim_{k\rightarrow\infty} c_k\mu(B(x_k,r_k))$.
 \end{remark}

Recall that a measure $\mu$ is said to have an atom at $x\in \R^n$ if $\mu(\{x\}) > 0$.

\ms
\begin{lemma} \label{t2.2}
If the doubling measure $\mu$ has an atom at $x_0$, then
$x_0$ is an isolated point of $\Sigma$.
\end{lemma}

\begin{proof}
Suppose $\mu$ has an atom at $x_{0}\in\Sigma$, and that there is a sequence 
$\{ x_k \}$ in $\Sigma \sm \{ x_0 \}$ that tends to $x_0$.  Assume that $x_k\in B(x_0,1/2)$
By passing to a subsequence, we may assume that 
$|x_{k+1}-x_0| \leq {1 \over 4} |x_{k}-x_0|$ for $k \geq 0$,
and then the balls $B_k = B(x_k,|x_{k}-x_0|/2)$ are disjoint.
Since $x_0 \in 4B_k$, the doubling property yields
\[
\mu (\{x_{0}\}) \leq \mu(4B_k) \leq \CD^2 \mu(B_k)
\]
for each $k$. Since $B_k\subset B(x_0,1)$ are disjoint then
\[
\mu(B(x_0,1))\geq \sum_k \mu(B_k) \geq \CD^{-2} \sum_k \mu (\{x_{0}\}) =\infty,
\]
which contradicts the fact that $\mu$ is a Radon measure.
\end{proof}

\subsection{ The $\alpha$ numbers control the measure of some balls}
\label{measure}

The main result of this subsection is the next lemma, which allows us to keep 
under control difficulties that come from potential rapid variations of the normalizing
factor $\mu(B(x,r))$.

\begin{lemma} \label{t2.3} 
Let $\mu$ and $\nu$ be two probability measures on $\bB=B(0,1)$
(not necessarily doubling)
and let $\phi : \overline \bB \to [0,1]$ be a $L$-Lipschitz function
such that
\begin{equation}  \label{2.6}
\phi(x) = 1 \ \text{ for } x\in \d \bB.
\end{equation}
Set 
\begin{equation}  \label{2.7}
F(t) = \mu(\phi^{-1}([0,t))) \ \text{ and } \ 
G(t) = \nu(\phi^{-1}([0,t))) \ \text{ for } t\in [0,1].
\end{equation}
Then 
\begin{equation} \label{2.8}
\int_0^1 |F(t)-G(t)| dt \leq L \W_1(\mu,\nu).
\end{equation}
As a consequence, if we set
\begin{equation} \label{2.9}
A_{\varepsilon} = \Big\{ t\in (0,1) \, ;  
\big|F(t) - G(t)\big| 
\leq {L \W_1(\mu,\nu) \over \varepsilon} \Big\},
\end{equation}
we get that $|(0,1) \sm A_\varepsilon| \leq \varepsilon$.
\end{lemma}

\ms
Let us comment on the lemma before we prove it. 
We typically use it with $\mu_{x,r}$ in the role of $\mu$, 
and some $\nu_V$ in the role of $\nu$; here $\nu_V$ is a multiple of the Hausdorff measure on the subspace $V$. A typical function $\phi$ 
will be of the form $\phi(x) = \min(1,L|x-x_0|)$, and the functions 
$F$ and $G$ will give some information on the measure of balls.
We prefer to use the definition of $A_{\varepsilon}$ above, because
we want to make sure that $|(0,1) \sm A_\varepsilon| \leq \varepsilon$,
so that various sets $A_\varepsilon$ intersect.
The influence of the large constant $\varepsilon^{-1}$, 
is often compensated by $\W_1(\mu,\nu)$ which is assumed to be very small.

\begin{proof}
Let $\psi:[0,1]\rightarrow\bR$ be $1$-Lipschitz and such that $\psi(1)=0$. 
Then $\psi \circ \phi$ is $L$-Lipschitz and vanishes on $\d \bB$,
so the definition \eqref{1.5} of $\W_1(\mu,\nu)$ 
(applied to $L^{-1} \psi \circ \phi$ if $L > 0$) yields
\begin{equation} \label{2.10}
\left|\int \psi\circ\phi \, d\mu-\int \psi\circ\phi \, d\nu \right| \leq L\W_1(\mu,\nu).
\end{equation}

Denote by $\sigma$ the image of $\mu$, pushed by $\phi$.
Thus $\sigma$ is supported by $[0,1]$, and 
$\sigma([0,a)) = \mu(\phi^{-1}([0,a)) = F(a)$ for $a > 0$.
We claim that 
\begin{equation} \label{2.11}
\int_{\bB}\widetilde \psi\circ\phi \, d\mu = \int_{[0,1]} \widetilde\psi d\sigma
\end{equation}
for all bounded measurable functions $\widetilde\psi$ (in particular for $\psi$ as in \eqref{2.10}).
When $\widetilde\psi$ is the characteristic function of some set $Z$,
the left-hand side is $\mu(\phi^{-1}(Z))$, which is equal to
$\sigma(Z)$ by definition. The general case follows by a standard 
measure-theoretic argument.

Using integration by parts, the fact that $\psi(1)=0$ and Fubini, we have 
\begin{eqnarray}\label{2.11A}
\int_{[0,1]} \psi(t) d\sigma(t) 
&= &- \int_{t \in [0,1]} \int_{u= t}^1 \psi'(u) du d\sigma(t) 
= - \int\int_{[0,1]^2} \1_{\{ t < u \}} \psi'(u) du d\sigma(t)\nonumber
\\ 
&= &- \int_{u \in (0,1]} \psi'(u) \int_{t \in [0,u)} d\sigma(t) du\nonumber\\
&=& - \int_{u \in (0,1)} \psi'(u) F(u) du
\end{eqnarray}
We have similar formulas for
$\nu$ and $G$, so by \eqref{2.10}, \eqref{2.11} and \eqref{2.11A}
\[
\int_{u \in (0,1)} \psi'(u) [G(u)-F(u)] du
= \int \psi\circ\phi \, d\mu-\int \psi\circ\phi \, d\nu 
\leq L\W_1(\mu,\nu)
\]
This holds for all 1-Lipschitz 
functions $\psi : [0,1] \to \R$ such that $\psi(1)=0$, 
hence 
\begin{equation} \label{2.12}
\int_{(0,1)} h(u) [G(u)-F(u)] du \leq L\W_1(\mu,\nu)
\end{equation}
for every bounded function $h$, with $||h||_\infty \leq 1$
(just integrate $h$ to get a $\psi$).
This proves \eqref{2.8}; the fact that $|(0,1) \sm A_\varepsilon| \leq \varepsilon$,
and then the lemma, follow easily.
\end{proof}

\ms
We return to our doubling measure $\mu$, and first explain notation
that will be used systematically. 
For $d \in [0,n]$, $x\in \Sigma$, and $r > 0$, we choose an
affine space $V_d(x,r) \in A'(d,n)$ such that
\begin{equation} \label{2.13}
\W_{1}(\mu_{x,r},\nu_{V_d(x,r)})\leq 2 \alpha_{d}(x,r)
\end{equation}
(see Definition \ref{t1.3}). We also set
\begin{equation} \label{2.14}
c_d(x,r) = \H^d(V_{d}(x,r) \cap \bB)^{-1} \in [1,2^d]
\end{equation}
(by \eqref{1.8}) so that (by \eqref{1.7})
\begin{equation} \label{2.15}
\nu_{V_d(x,r)} = c_d(x,r) \H^d|_{V_{d}(x,r)}.
\end{equation}
Often $d$ will be fixed, and we may drop
it from the notation. But for Theorem~\ref{t1.5}, we work with all integers
$d$ between 1 and $n$ at the same time, so we choose an integer $d = d(x,r)$ such that
\begin{equation} \label{2.16}
\alpha(x,r)=\alpha_{d(x,r)}(x,r).
\end{equation}
We also use the other affine spaces
\begin{equation} \label{2.17}
W_d(x,r) = x + r V_d(x,r) 
\end{equation}
that pass near $x$ and are parallel to the previous ones.

\ms
Let us apply Lemma \ref{t2.3} in this context.
\begin{lemma} \label{t2.4}
For each choice of $d \in [0,n]$,
$x\in \Sigma$, $r > 0$, $y \in B(0,1/2)$, 
and $\varepsilon > 0$, there is a set $A \subset (0,1/2)$ such that
\begin{equation} \label{2.18}
|(0,1/2) \sm A| \leq \varepsilon,
\end{equation}
and 
\begin{equation} \label{2.19}
\big| \mu_{x,r}(B(y,s)) - \nu_{V_d(x,r)}(B(y,s)) \big|
\leq 4 \varepsilon^{-1} \alpha_d(x,r)
\ \text{ for } s \in A.
\end{equation}
\end{lemma}

\ms
\begin{proof}  
We want to apply Lemma \ref{t2.3} to the (restriction to $\bB$ of the)
measures $\mu_{x,r}$ and $\nu_{V_d(x,r)}$, and with 
\[
\phi(z) = \min(1,2|z-y|).
\]
Since $\phi$ is $2$-Lipschitz and $\phi(z) = 1$ on $\d \bB$,
Lemma \ref{t2.3} applies, with $L=2$.
Notice that $\phi^{-1}([0,t)) = B(y,t/2)$ for $0 < t < 1$. 
That is, with the notation of Lemma \ref{t2.3},
\begin{equation} \label{2.20}
F(t) = \mu_{x,r}(B(y,t/2)) \ \text{ and } \ 
G(t) = \nu_{V_d(x,r)}(B(y,t/2)).
\end{equation}
Let us take $A = A_\varepsilon/2 \subset (0,r/2)$,
where $A_\varepsilon$ is as in Lemma \ref{t2.3}.
Note that \eqref{2.18} holds because $|(0,1) \sm A| \leq \varepsilon$,
and \eqref{2.19} holds because, for $s\in A$, by \eqref{2.20}, \eqref{2.9}, and \eqref{2.13}
\begin{eqnarray}  \label{2.21}
\Big| \mu_{x,r}(B(y,s)) - \nu_{V_d(x,r)}(B(y,s)) \Big|
&=& |F(2s)-G(2s)| \leq  2 \varepsilon^{-1} \W_1(\mu_{x,r},\nu_{V_d(x,r)})
\nonumber \\
&\leq&  4 \varepsilon^{-1} \alpha_d(x,r).
\end{eqnarray}
 
\end{proof}

Let us comment again on the role of Lemma \ref{t2.3} and $\varepsilon$.
If we could apply the definition of $\W_1(\mu_{x,r},\nu_{V_d(x,r)})$
with $\psi = \1_{B(y,s)}$, we would directly get a better bound than 
\eqref{2.19}, valid for all $s \in (0,1/2)$. The standard way to deal with that
would be to use cut-off functions $\psi$ that look like $\1_{B(y,s)}$,
and then we would replace $\mu_{x,r}(B(y,s))$ by a slightly fuzzy quantity
that lies between $\mu_{x,r}(B(y,s))$ and $\mu_{x,r}(B(y,s+\varepsilon))$,
say. Instead we decide to keep the same numbers $\mu_{x,r}(B(y,s))$, and
pay for this with a larger bound (as above) and the obligation to restrict
to $s\in A$.

Note that $y$ in the statement does not need to lie in $\Sigma$;
this will be convenient because for some estimates we prefer to take it
in $V_d(x,r)$, for instance. The estimate \eqref{2.19} will be central,
but not always so easy to use because of the various normalizations
that it contains.

\subsection{The number $\alpha_d(x,2r)$ controls $\beta_d(x,r)$}
\label{beta}

Most of our estimates on the geometry of $\Sigma$ will use estimates
on the $\beta$-numbers. But also, if we want to use \eqref{2.19} efficiently,
it will be good to find points of $V_d(x,r)$ near $\Sigma$, because the 
homogeneity of $\nu_{V_d(x,r)}$ is better at those points.

It will be easier to control the numbers
\begin{equation} \label{2.22}
\beta_{d,1}(x,r)=\inf_{W\in A'(d,n)}\fint_{B(x,r)} \frac{\dist(z,W)}{r} d\mu(z),
\end{equation}
where $A'(d,n)$ is as in Definition \ref{t1.3}
and we restrict to $x\in \Sigma$ to make sure that the average is well defined.

The following lemma, which comes from \cite{Tolsa-uniform-rectifiability},
is the main result of this subsection. 
We include the proof because we use similar ideas later on.

\begin{lemma}(\cite[Lemma 3.2]{Tolsa-uniform-rectifiability}) \label{t2.5}
For $x\in \Sigma$ and $r > 0$,
\begin{equation}\label{2.23}
\beta_{d,1}(x,r) 
\leq \fint_{B(x,r)}\frac{\dist(z,W_d(x,2r))}{r} d\mu(z) 
\leq 16 \CD \alpha_{d}(x,2r).
\end{equation}
\end{lemma}

\ms
\begin{proof}
The first inequality is obvious, so we prove the second one.
Let $\phi$ be the $2$-Lipschitz function such that 
$\1_{B(0,1/2)}\leq \phi \leq \1_{\bB}$, and define a Lipschitz
function $\psi$ by
\[
\psi(z)= 2\dist(z,V_{x,2r})\phi(z) \ \text{ for } z\in \overline \bB.
\]
Notice that $\psi$ is $8$-Lipschitz, and it vanishes on 
$\d\bB$ because $\phi$ does.
Also, $\int \psi d \nu_{V_d(x,2r)} = 0$, because $\psi = 0$
on $V_d(x,2r)$. Thus the definition \eqref{1.5} of $\W_1$ yields
\[
\int \psi d \mu_{x,2r} = 
\Big|\int \psi d \mu_{x,2r} - \int \psi d \nu_{V_d(x,2r)} \Big| 
\leq 8\W_1(\mu_{x,2r},\nu_{V_d(x,2r)})
\leq 16 \alpha_{d}(x,2r).
\]
But the change of variable $z = x+2ry$ yields
\begin{align*}
\fint_{B(x,r)}\frac{\dist(z,W_d(x,2r))}{r} &d\mu(z)
= \mu(B(x,r))^{-1} \int_{B(x,r)}\frac{\dist(z,W_d(x,2r))}{r} d\mu(z)
\nonumber \\
&={\mu(B(x,2r)) \over \mu(B(x,r))}
\int_{B(0,1/2)} 2\dist(y,V_d(x,2r)) d\mu_{x,2r}(y)
\nonumber \\
&= {\mu(B(x,2r)) \over \mu(B(x,r))}
\int_{B(0,1/2)} \psi d\mu_{x,2r}
\leq 16 \CD \alpha_{d}(x,2r)
\end{align*} 
by \eqref{1.6}, \eqref{2.17}, the fact that $\phi(y) = 1$ on $B(0,1/2)$,
and \eqref{1.1}.
\end{proof}

It is classical that for Ahlfors regular measures, the number
$\beta_{d,1}(x,r)$ gives some control on $\beta_{d,\infty}(x,r/2)$,
where we
\begin{equation} \label{2.24}
\beta_{d,\infty}(x,r)=\inf_{W\in A(d,n)} \ \sup_{y\in \Sigma \cap B(x,r)}
\frac{\dist(y,W)}{r}.
\end{equation}
See for instance \cite[p.27]{DS}. The same proof also yields that
\begin{equation} \label{2.25}
\beta_{d,\infty}(x,r/2) \leq C \beta_{d,1}(x,r)^\eta,
\end{equation}
where $C$ and $\eta > 0$ depend on the doubling constant $\CD$.
When $\mu$ is Ahlfors regular of dimension $d$, we get $\eta = 1/(d+1)$. 
Although \eqref{2.25} is not precise enough to be used systematically it yields useful information

For a general doubling measure we derive the analogue of \eqref{2.25} from a direct application
of the definitions.

\begin{lemma} \label{t2.6}
There are constants $C > 0$ and $\eta >0$, that depend only on $\CD$,
such that for $0 \leq d \leq n$,  $x\in \Sigma$, and $r > 0$,
\begin{equation} \label{2.26}
\beta_{d,\infty}(x,2r/3) \leq {3 \over 2} \sup_{y\in \Sigma \cap B(x,2r/3)}
\frac{\dist(y,W_d(x,r))}{r}
\leq C \alpha_d(x,r)^\eta.
\end{equation}
\end{lemma} 

\ms
\begin{proof}  
We just need to prove that for $y \in \Sigma \cap B(x,2r/3)$
\begin{equation} \label{2.27}
r^{-1}\dist(y,W_d(x,r)) \leq C \alpha_d(x,r)^\eta. 
\end{equation}
Let such a $y$ be given, 
set $\xi = r^{-1}(y-x) \in B(0,2/3)$ and
\begin{equation} \label{2.28}
\tau  = \min\{1/3, r^{-1}\dist(y,W_d(x,r)\} = \min\{1/3, \dist(\xi,V_d(x,r)\}
\end{equation}
(because $W_d(x,r) = x + rV_d(x,r)$).
If $y \in W_d(x,r)$, \eqref{2.27} holds trivially. Otherwise, we use the fact that
\begin{equation} \label{2.29}
B(\xi,\tau) \cap V_d(x,r) = \emptyset.
\end{equation}
The function $\psi$ defined on $\R^n$ by
\[
\psi(z) = [\tau- |\xi-z|]_+ = \max\{0, \tau- |\xi-z|\}
\]
is $1$-Lipschitz and vanishes on $\d \bB$ (because $\tau \leq 1/3$), so we can 
apply the definition \eqref{1.5} of $\W_1(\mu_{x,r},\nu_{V_d(x,r)})$ to it.
Because of \eqref{2.29}, there is no contribution of $\nu_{V_d(x,r)}$, and
we get that by \eqref{1.5} and \eqref{2.13}
\begin{eqnarray}  \label{2.30}
{\tau \over 2}\, \mu_{x,r}(B(\xi,\tau/2)) 
&\leq& \int \psi d\mu_{x,r}  = \int \psi d\mu_{x,r}-\int \psi d\mu_{x,r}  \nu_{V_d(x,r)}
\nonumber\\
&\leq& \W_1(\mu_{x,r},\nu_{V_d(x,r)}) \leq 2 \alpha_d(x,r).
\end{eqnarray}
 Let $k \geq 0$ denote the integer such that
$2^{-k} \leq \tau < 2^{-k+1}$. Notice that by \eqref{1.6}
\[
\mu_{x,r}(B(\xi,\tau/2)) = {\mu(x+rB(\xi,\tau/2)) \over \mu(B(x,r))}
= {\mu(B(y, \tau r/2)) \over \mu(B(x,r))}.
\]
 
But $B(x,r) \subset B(y,2r) \subset B(y, 2^{k+2}(\tau r/2))$, so \eqref{1.1} yields
\[
{\mu(B(y, \tau r/2)) \over \mu(B(x,r))} \geq  \CD^{-k-2},
\]
and now \eqref{2.30} says that
\begin{equation} \label{2.31}
2^{-k+1} \CD^{-k-2} \leq \tau \mu_{x,r}(B(\xi,\tau/2)) 
\leq 4\alpha_d(x,r).
\end{equation}
Choose $\eta$ such that
\begin{equation} \label{2.32}
(2\CD)^\eta = 2;
\end{equation}
Then \eqref{2.31} also says that
\[
\tau^{1/\eta} 
\leq 2^{(1-k)/\eta} = (2\CD)^{1-k}  \leq C \alpha_d(x,r),
\]
where $C$ depends on $\CD$.
If $\tau = r^{-1}\dist(y,W_d(x,r))$, we deduce \eqref{2.27} from this.
Otherwise, $\tau = 1/3$, $k=1$, \eqref{2.31} says that $\alpha_d(x,r) \geq C^{-1}$,
and \eqref{2.27} holds just because $\dist(y,W_d(x,r))\leq 2r$.
\end{proof}

\ms
We can also control the distance from points of $W_d(x,r)$ to $\Sigma$
by the same sort of argument.

\begin{lemma} \label{t2.7}
For $0 \leq d \leq n$,  $x\in \Sigma$, and $r > 0$,
\begin{eqnarray} \label{2.33}
\sup_{y\in W_d(x,r) \cap B(x,r/2)} {\dist(y,\Sigma) \over r}
&\leq& 8 \W_1(\mu_{x,r},\nu_{V_d(x,r)})^{1\over d+1} 
\nonumber\\
&\leq& 16 \alpha_d(x,r)^{1\over d+1} .
\end{eqnarray}
\end{lemma} 

\ms
\begin{proof}  
Let $\tau \in (0,1/2]$ and $y\in W_d(x,r) \cap B(x,r/2)$
be given. Set $\xi = r^{-1}(y-x)$ and define a Lipschitz function $\psi$ by
\[
\psi(z) = \max\{0, \tau- |z-\xi|\}\ \text{ for } z \in \R^n;
\]
notice that $\psi$ is $1$-Lipschitz and vanishes on $\d\bB$, 
because $\xi \in B(0,1/2)$ and $\tau \leq 1/2$, so \eqref{1.5} yields
\[
\int \psi d\nu_{V_d(x,r)} - \int \psi d \mu_{x,r} \leq A
\hbox{ with }
A = \W_1(\mu_{x,r},\nu_{V_d(x,r)}) \leq 2 \alpha_d(x,r)
\]
by \eqref{2.13}. If $\dist(y,\Sigma) \geq \tau r$,
$B(\xi,\tau)$ does not meet the support of $\mu_{x,r}$, hence
$\int \psi d \mu_{x,r} = 0$, while on the other side
\[
\int \psi d\nu_{V_d(x,r)} \geq {\tau \over 2}\, \nu_{V_d(x,r)}(B(\xi,\tau/2))
\geq {\tau \over 2} \H^d(B(\xi,\tau/2) \cap V_d(x,r))
\geq (\tau/2)^{d+1}
\]
by \eqref{1.7}, \eqref{1.8}, our convention \eqref{1.2}, and because 
$\xi \in V_d(x,r)$.
Thus we get that 
\[
\tau \leq  2A^{1 \over d+1} \leq 4 \alpha_d(x,r)^{1 \over d+1}.
\] 
Now let $\delta$ denote the supremum
in \eqref{2.33}. If $\delta \leq 1/2$, let $\tau = \delta$ in the argument above. If $\dist(y,\Sigma) \geq \delta r$ then
$\delta \leq 2A^{1 \over d+1}$, which implies \eqref{2.33}.
Otherwise, let $\tau = 1/2$ then $\dist(y,\Sigma) \geq \frac{r}{2}$ which yields $A^{1 \over d+1} \geq 1/4$. Thus
$\delta \leq 2 \leq 8 A^{1\over d+1}$; and Lemma \ref{t2.7} follows.
\end{proof}

\begin{lemma} \label{t2.8}
For $a \in (0,1)$, there exists $C_a \geq 0$ depending 
 on $a$, $\CD \,$, and $d$, with the following property. 
Let $0 \leq d \leq n$,  $x,y\in \Sigma$, and $r,t > 0$
be such that
\begin{equation} \label{2.34}
ar \leq t \leq r  \ \text{ and } \  \ |x-y|+{t \over 2} < r
\end{equation}
Then 
\begin{equation} \label{2.35}
\dist(z,W_d(x,r)) \leq C_a (\alpha_d(x,r) + \alpha_d(y,t)) r
\ \text{ for } z \in W_d(y,t) \cap B(x,2r)
\end{equation}
and
\begin{equation} \label{2.36}
\dist(z,W_d(y,t)) \leq C_a (\alpha_d(x,r) + \alpha_d(y,t)) r
\ \text{ for } z \in W_d(x,r) \cap B(x,2r).
\end{equation}
\end{lemma} 

\ms
\begin{proof}
At the same time we test $W_d(x,r)$ in $B(x,r)$ and 
$W_d(y,t)$ in $B(y,t)$, with corresponding functions $\psi_x$
and $\psi_y$. First define a cut-off function $\varphi$ by
\begin{equation} \label{2.37}
\phi(z) = [t - 2|z-y|]_+ = \max\{0,t - 2|z-y|\};
\end{equation}
notice that $\phi$ is $2$-Lipschitz and bounded by $t \leq r$. Moreover
if
$\phi(z) \neq 0$, then $z \in B(y,t/2) \subset B(x,r)$
(by \eqref{2.34}). That is, $\phi = 0$ vanishes outside of $ B(y,t/2)$.
It is $2$-Lipschitz and bounded by $t \leq r$.
We define $\psi$ by
\begin{equation} \label{2.38}
\psi(z) = \phi(z) \dist(z,W_d(x,r)),
\end{equation}
which is now $5r$-Lipschitz, and its equivalents $\psi_x$ and $\psi_z$ defined by
\[
\psi_x(w) = \psi(x+rw)  \ \text{ and }\ \psi_y(w) = \psi(y+tw);
\]
Since $\psi_x$ is $5r^2$-Lipschitz and vanishes on $\R^n \sm \bB$, 
\eqref{1.5} yields
\begin{equation} \label{2.39}
 \av{\int\psi_x d\mu_{x,r}-\int\psi_x d\nu_{V_d(x,r)}}
\leq 5 r^2 \W_{1}(\mu_{x,r},\nu_{V_d(x,r)}) \leq 10 r^2 \alpha_d(x,r)
\end{equation}
by \eqref{2.13}.
Since $\psi_y$ is $5r^2$-Lipschitz and vanishes on $\R^n \sm B(0,1/2)$, then
\begin{equation} \label{2.40}
 \av{\int\psi_y d\mu_{y,t}-\int\psi_y d\nu_{V_d(y,t)}}
\leq 5 r^2 \W_{1}(\mu_{y,t},\nu_{V_d(y,t)}) \leq 10 r^2 \alpha_d(y,t).
\end{equation}
We now use the special form of $\psi$, which makes it vanish on
$W_d(x,r)$; thus $\psi_x$ vanishes on $V_d(x,r)$, the corresponding
term of \eqref{2.39} disappears, and we are left with
\[
\int\psi_x d\mu_{x,r} \leq 10 r^2 \alpha_d(x,r).
\]
Using \eqref{1.6} and a change of variable we have
\[  
\begin{aligned}
\int\psi_x(z) d\mu_{x,r}(z) &= \int_{\bB} \psi(x+rw) d\mu_{x,r}(z)
= \fint_{B(x,r)} \psi(u) d\mu(u)
\\
&= \mu(B(x,r))^{-1} \int_{B(x,r)} \psi(u) d\mu(u)
\\
& = \mu(B(x,r))^{-1} \int \psi(u) d\mu(u)
 \end{aligned}
\]
A similar computation yields
\[
\int\psi_y(z) d\mu_{y,t}(z) =\mu(B(y,t))^{-1} \int \psi(u) d\mu(u).
\]
Therefore
\begin{equation}\label{2.40A}
\int\psi_x(z) d\mu_{x,r}(z)= {\mu(B(y,t))\over \mu(B(x,r))} \int\psi_y d\mu_{y,t}
\end{equation}
By \eqref{2.34}, the fact that $B(x,r) \subset B(y,2r) \subset B(y,2a^{-1}t)$ and 
multiple applications of \eqref{1.1} yield
\begin{equation} \label{2.41}
\mu(B(x,r)) \leq C(a) \mu(B(y,t)),
\end{equation}
where $C(a)$ depends on $a$ and $\CD$. Thus \eqref{2.40A} and \eqref{2.41} ensure that 
\[
\int\psi_y d\mu_{y,t} \leq C(a) \int\psi_x d\mu_{x,r}
\leq 10 C(a) r^2 \alpha_d(x,r),
\]
and \eqref{2.40} yields
\begin{eqnarray} \label{2.42}
\int\psi_y d\nu_{V_d(y,t)} 
&\leq&  \int\psi_y d\mu_{y,t}+ 10 r^2 \alpha_d(y,t)
\nonumber\\
&\leq& 10 C(a) r^2 \alpha_d(x,r) + 10 r^2 \alpha_d(y,t).
\end{eqnarray}
On the other hand by \eqref{2.37} and \eqref{2.38}
\[
\begin{aligned}
\int\psi_y d\nu_{V_d(y,t)} 
&= c_d(y,t) \int_{\bB \cap V_d(y,t)} \psi(y+tz) d\H^d(z)
\\
&= c_d(y,t) t^{-d} \int_{B(y,t) \cap W_d(y,t)} \psi(u) d\H^d(u)
\\
&\geq c_d(y,t) t^{-d} \int_{B(y,t/4) \cap W_d(y,t)} 
{t \over 2} \dist(u, W_d(x,r)) d\H^d(u).
\end{aligned}
\]
We compare with \eqref{2.42}, use the fact that $c_d(y,t) \geq 1$, and \eqref{2.34} to obtain that 
\begin{equation} \label{2.43}
t^{-d} \int_{B(y,t/4) \cap W_d(y,t)} \dist(u, W_d(x,r)) d\H^d(u)
\leq C r (\alpha_d(x,r) + \alpha_d(y,t)),
\end{equation}
where here and for the rest of the lemma, $C$ denotes a constant that
may depend on $d$, $\CD$, and $a$.
Now we can also run the same estimate with the function $\psi$ defined by
\begin{equation} \label{2.44}
\psi(z) = \phi(z) \dist(z,W_d(y,t))
\end{equation}
(instead of \eqref{2.38}). We notice that \eqref{2.39} and \eqref{2.40} are
still valid, that the second term in \eqref{2.40} vanishes, get an estimate
on $\int\psi_y d\mu_{y,t}$, transform it into an estimate on
$\int\psi_x d\mu_{x,r}$, and plug it back in \eqref{2.39}.
Note that \eqref{2.41} holds because $B(y,t) \subset B(x,2r)$ 
and so \eqref{1.1} yields $\mu(B(y,t)) \leq \CD \mu(B(x,r))$.
As above we get an estimate for
$\int\psi_x d\nu_{V_d(x,r)}$, which yields, as in \eqref{2.43} that
\begin{equation} \label{2.45}
r^{-d} \int_{B(y,t/4) \cap W_d(x,r)} \dist(u, W_d(y,t)) d\H^d(u)
\leq C r (\alpha_d(x,r) + \alpha_d(y,t)).
\end{equation}

We need to check that the two integrals, in \eqref{2.43}
and \eqref{2.45} really concern significant pieces of $W_d(y,t)$
and $W_d(x,r)$ respectively. We check this only when
\begin{equation} \label{2.46}
\alpha_d(x,r) + \alpha_d(y,t) < c
\end{equation}
where the small constant $c$ depends on $d$ and $\CD$
and will be chosen soon. Otherwise, \eqref{2.35} and \eqref{2.36} are trivial, because all the distances written there there are less than $4r$.

If $c$ is small enough, Lemma \ref{t2.6} (applied to $B(y,t)$) implies that 
\[
\dist(y, W_d(y,t)) \leq C \alpha_d(y,t)^\eta t \leq t/8,
\]
and then $B(y,t/4)$ contains a ball $B_1$ of radius $t/8$ centered on $W_d(y,t)$. 
Similarly, we may apply Lemma \ref{t2.6} to $B(x,r)$ and get that
\[
\dist(y, W_d(x,r)) \leq C \alpha_d(x,r)^\eta r \leq t/8,
\]
because $y \in \Sigma \cap B(x,r/2)$ and if $c$ is small enough, depending on
$a$ as well. Then again $B(y,t/4)$ contains a ball $B_2$ of radius $t/8$ 
centered on $W_d(x,r)$. To conclude the proof of Lemma \ref{t2.8} we need the following simple lemma.

\begin{lemma} \label{t2.9}
For each integer $d$, there is a constant $C(d)  > 0$ such that
if $W$ is an affine space of dimension $d$, $A$ is an affine function
on $\R^n$, $B$ is a ball centered on $W$, and
\begin{equation} \label{2.47}
\fint_{B\cap W} |A(u)| d\H^d(u) \leq 1, 
\end{equation}
then for $\lambda > 1$,
\begin{equation} \label{2.48}
\sup\big\{ |A(u)| \, ; \, u \in \lambda B \cap W\big\} \leq C(d) \lambda.
\end{equation}
\end{lemma}

\ms
\begin{proof} 
We may assume that $W = \R^d$, $B= B(0,1)$, and that on $W$,
$A(u) = \alpha u_1 + \beta$ for some choice of $\alpha$, $\beta \in \R$,
and where $u_1$ is the first coordinate of $u$.
Then $|\beta| \leq 2$ (integrate on the half space where $\alpha u_1$ and $\beta$
have the same sign, so that $|A(u)| \geq \beta$). Thus
$\fint_{B\cap W} |\alpha u_1| \leq 3$, $\alpha \leq C(d)$, and the
lemma follows.
\end{proof}

We apply the lemma to $B_1$, the affine space $W_d(y,t)$,
a multiple of the function $\dist(u, W_d(x,r))$
(which is indeed of the form $|A(u)|$), and $\lambda = 16 a^{-1}$
(to make sure that $B(x,2r) \subset \lambda B_1$), and we deduce from
\eqref{2.43} that \eqref{2.35} holds.
Similarly, \eqref{2.36} follows from \eqref{2.45} and Lemma \ref{t2.9},
applied to $B_2$ and a multiple of $\dist(u, W_d(y,t))$.
Lemma \ref{t2.8} follows.
\end{proof}

\subsection{Evaluation of the density at the good points}
\label{goodclients}

It is important to understand how the density ratios for $y\in \Sigma$ and $t > 0$
\begin{equation} \label{2.49}
\theta_d(y,t) = t^{-d} \mu(B(y,t)),
\end{equation}
vary. In this section we 
fix $d \geq 0$, $x\in \Sigma$, and $r > 0$, and find a large good set
$\cG = \cG_d(x,r)$ of pairs $(y,t)$ (the good points), on which 
$\theta_d(y,t)$ is nearly constant.

We work with three small constants here:
a small  $\kappa \in (0,1/2)$, to be chosen later, a smaller
$\varepsilon \in (0,\kappa)$, that will depend on $\kappa$, and 
$\alpha_0 > 0$, that will depend on $\kappa$ and $\varepsilon$.
We assume that
\begin{equation} \label{2.50}
\alpha_d(x,r) \leq \alpha_0.
\end{equation}
The constant $\alpha_0$ works as a threshold; that is, \eqref{2.50}
allows us to estimate things easily. In general  we expect $\alpha_d(x,r)$
to be often much smaller than $\alpha_0$, and the more precise estimates
require $\alpha_d(x,r)$.

We first define a good set of points $y$, by
\begin{equation} \label{2.51}
G_\Sigma = \big\{ y\in \Sigma \cap B(x,r/2) \, ; \, 
r^{-1} \dist(y,W_d(x,r)) \leq \varepsilon^{-1} \alpha_d(x,r) \big\}.
\end{equation}
Then, by Chebyshev and \eqref{2.23},
\begin{eqnarray} \label{2.52}
\mu(B(x,r/2) \sm G_\Sigma) 
&\leq& \varepsilon \alpha_d(x,r)^{-1} 
\int_{\Sigma \cap B(x,r/2)}  {\dist(y,W_d(x,r)) \over r}
\nonumber\\
&\leq& 32 \CD \varepsilon \mu(B(x,r/2)).
\end{eqnarray}
Then let $\kappa > 0$ 
be another small constant, and for each
$y\in G_\Sigma$, we set $\xi_y = r^{-1}(y-x)$ and
\begin{equation} \label{2.53}
G_\kappa(y) = \big\{ s \in (\kappa, 1/2) \, ; \, 
\big| \mu_{x,r}(B(\xi_y,s)) - \nu_{V_d(x,r)}(B(\xi_y,s))\big|
\leq 4 \varepsilon^{-1}\alpha_d(x,r) \big\}.
\end{equation}

By \eqref{2.19}, $G_\kappa(y)$ contains $A\cap (\kappa, 1/2)$, where
$A$ is the set of Lemma \ref{t2.5}, and so \eqref{2.18} says that
\begin{equation} \label{2.54}
|(\kappa, 1/2) \sm G_\kappa(y)| \leq \varepsilon.
\end{equation}
Finally, our set of good points is
\begin{equation} \label{2.55}
\cG = \cG_d(x,r) = \big\{ (y,t) \in \Sigma \times (0,r/2) \, ; \,
y\in G_\Sigma \text{ and } r^{-1}t \in G_\kappa(y) \big\}.
\end{equation}

\begin{lemma} \label{t2.10}
There is a constant $C \geq 0$, that depends only on $d$, such that
if \eqref{2.50} holds and $\alpha_0$ is chosen small enough, 
depending on $n$, $d$, $\varepsilon$, and $\kappa$, then for $(y,t) \in \cG_d(x,r)$
\begin{equation} \label{2.56}
\Big| {\theta_d(y,t) r^d \over \mu(B(x,r))} - c_d(x,r)  \Big|
\leq C \kappa^{-d} \varepsilon^{-1} \alpha_d(x,r).
\end{equation}
\end{lemma}

\ms
It should be noted that the quantities $c_d(x,r)$ and
$\mu(B(x,r))$ may not be known as precisely as \eqref{2.56}
would suggest, typically because the center $x$ may be a little bit too
far from $W_d(x,r)$. Nonetheless \eqref{2.56} says that their product is rather stable.

\ms
\begin{proof} 
Let $(y,t) \in \cG$ be given. By \eqref{2.51}, we can find 
$w \in W_d(x,r)$ such that $|y-w|\leq \varepsilon^{-1} \alpha_d(x,r) r$.
Set $\xi_y = r^{-1}(y-x)$ (as above) and $\xi_w = r^{-1}(w-x)$; then 
\begin{equation} \label{2.57}
|\xi_y - \xi_w| = r^{-1}|y-w| \leq \varepsilon^{-1} \alpha_d(x,r).
\end{equation}
Also set $s = r^{-1}t$, and notice that $s \in G_\kappa(y)$
because $(y,t) \in \cG$.
Set $\tau = \varepsilon^{-1} \alpha_d(x,r)$; if $\alpha_0$ is small enough 
(depending on $\kappa$ and $\varepsilon$),
then $\tau = \varepsilon^{-1} \alpha_d(x,r) \leq \varepsilon^{-1} \varepsilon_0 
\leq 10^{-1} s$ because $s \geq \kappa$. Then by \eqref{2.57},
\[
B(\xi_w, s-\tau) \subset B(\xi_y,s) \subset B(\xi_w,s+\tau).
\]
By \eqref{2.15} and our normalization \eqref{1.2},
\[
\nu_{V_d(x,r)}(B(\xi_w,s-\tau)) = c_d(x,r) \H^d(V_d(x,r)) B(\xi_y,s-\tau))
= c_d(x,r) (s-\tau)^{d}
\]
and similarly for $B(\xi_w,s+\tau)$; hence
\begin{equation} \label{2.58}
c_d(x,r) (s-\tau)^{d} \leq \nu_{V_d(x,r)}(B(\xi_y,s))
\leq c_d(x,r) (s+\tau)^{d}.
\end{equation}
Recall from \eqref{2.14} that
\begin{equation} \label{2.59}
1 \leq c_d(x,r) \leq 2^d.
\end{equation}
Then \eqref{2.58} implies that using the fact that $\tau\le 10^{-1}s$ and $\tau=\varepsilon^{-1}\alpha_d(x,r)$
\begin{eqnarray} \label{2.60}
| \nu_{V_d(x,r)}(B(\xi_y,s)) - c_d(x,r) s^d|
&\leq& c_d(x,r) s^d \Big(\big({s+\tau \over s}\big)^d - 1\Big)
\leq C s^d \, {\tau \over s} 
\nonumber \\
&=& C s^{d-1} \varepsilon^{-1} \alpha_d(x,r).
\end{eqnarray}
Adding this to the defining inequality in \eqref{2.53}, 
which holds because $s \in G_\kappa(y)$, and get that
\begin{equation} \label{2.61}
\big| \mu_{x,r}(B(\xi_y,s)) - c_d(x,r) s^d \big|
\leq 4 \varepsilon^{-1}\alpha(x,r) + C s^{d-1} \varepsilon^{-1} \alpha_d(x,r).
\end{equation}
By \eqref{2.49} and \eqref{1.6}, and because $y = x+r \xi_y$ and $t=rs$,
\[
\theta_d(y,t) = t^{-d} \mu(B(y,t)) = t^{-d} \mu(B(x,r)) \mu_{x,r}(B(\xi_y,s)).
\]
Then \eqref{2.61} yields 
\begin{eqnarray} \label{2.62}
\Big| {\theta_d(y,t) r^d \over \mu(B(x,r))} - c_d(x,r)  \Big|
&=& \Big| r^d t^{-d}\mu_{x,r}(B(\xi_y,s)) - c_d(x,r)  \Big|
\nonumber \\
&=& \Big| s^{-d}\mu_{x,r}(B(\xi_y,s)) - c_d(x,r)  \Big|\nonumber\\
&\leq& 4 \varepsilon^{-1} s^{-d} \alpha_d(x,r) + C s^{-1} \varepsilon^{-1} \alpha_d(x,r)
\nonumber \\
&\leq& C \kappa^{-d} \varepsilon^{-1} \alpha_d(x,r)
\end{eqnarray}
This proves \eqref{2.56}, and Lemma \ref{t2.10} follows.
\end{proof}

\ms
Lemma \ref{t2.10} suggests to use the stabilized density ratios
\begin{equation} \label{2.63}
\theta_d^\ast(x,r) = r^{-d} c_d(x,r) \mu(B(x,r))
\end{equation}
for $x\in \Sigma$ and $r > 0$, where we expect that the 
slightly wilder variations of $c_d(x,r)$ and $\mu(B(x,r))$ will compensate each
other. The next results uses this.

\begin{corollary} \label{t2.11}
There is a constant $C \geq 0$, which depends only on $d$ and $\CD$,
such that
\begin{equation} \label{2.64}
\Big| \log\Big({\theta_d^\ast(x,r)\over \theta_d^\ast(x,\rho)}\Big)\Big|
\leq C (\alpha_d(x,r) + \alpha_d(x,\rho))
\end{equation}
for $x\in \Sigma$ and $0 < \rho \leq r \leq 4 \rho$.
\end{corollary}

\ms
\begin{proof}  
Choose
\begin{equation} \label{2.65}
\kappa = 10^{-1} \ \text{ and } \ \varepsilon = (100 \CD^3)^{-1};
\end{equation}
this gives a constant $\alpha_0$ such that the conclusions of  
Lemma \ref{t2.10} hold when we have \eqref{2.50}. We first
prove that \eqref{2.64} holds when
\begin{equation} \label{2.66}
\alpha_d(x,r) + \alpha_d(x,\rho) \leq \alpha_0,
\end{equation}
so that we can apply Lemma \ref{t2.10} to both pairs $(x,\rho)$ and $(x,r)$.

Let $G_\Sigma$ be the set of \eqref{2.51}, and denote by $G'_\Sigma$
its analogue for the radius $\rho$.  By \eqref{2.52},
\[
\mu(B(x,r/2) \sm G_\Sigma)
\leq 32 \CD \varepsilon \mu(B(x,r/2))
\leq 32 \CD^3 \varepsilon \mu(B(x,\rho/2))
\]
and similarly
\[
\mu(B(x,\rho/2) \sm G'_\Sigma)
\leq 32 \CD \varepsilon \mu(B(x,\rho/2)).
\]
Then by the definition of $\varepsilon$ (see\eqref{2.65})
\[
\mu(B(x,\rho/2) \sm (G_\Sigma \cap G'_\Sigma))
\leq 32 [\CD^3+\CD] \varepsilon \mu(B(x,\rho/2))
< \mu(B(x,\rho/2).
\]
Thus there is 
$y\in \Sigma \cap B(x,\rho/2) \cap G_\Sigma \cap G'_\Sigma$.

Then we have two large sets $G_\kappa(y)$ and $G'_\kappa(y)$, associated to
$y$ as in \eqref{2.53} (but with the radii $r$ and $\rho$), 
and \eqref{2.54} allows us to choose $t \in (\kappa r, r/2)$
such that $s = r^{-1}t \in G_\kappa(y)$, but also $s' = \rho^{-1} t \in G'_\kappa(y)$
(this last forces $t\in (\kappa\rho,\rho/2)$, but there is a lot of room
left since $\rho \leq r \leq 4 \rho$ and $\kappa = 10^{-1}$).
Then $(y,t)$ lies in both good sets $\cG(x,r)$ and $\cG(x,\rho)$ 
(see \eqref{2.55}). By \eqref{2.56} for $r$,
\begin{equation} \label{2.67}
\Big| {\theta_d(y,t) r^d \over \mu(B(x,r))} - c_d(x,r)  \Big|
\leq C \kappa^{-d} \varepsilon^{-1} \alpha_d(x,r) \leq C \alpha_d(x,r),
\end{equation}
where from now on $C$ also depends on $\CD$ (because of $\varepsilon$ see \eqref{2.65}).
Similarly,
\[
\Big| {\theta_d(y,t) \rho^d \over \mu(B(x,\rho))} - c_d(x,\rho)  \Big|
\leq C \kappa^{-d} \varepsilon^{-1} \alpha_d(x,\rho).
\]
In addition, by \eqref{2.14} $1 \leq c_d(x,\rho) \leq 2^d$,
so \eqref{2.67} yields
\[
\Big| {\theta_d(y,t) r^d \over c_d(x,r) \mu(B(x,r))} - 1 \Big|
\leq C \alpha_d(x,r),
\]
and the definition \eqref{2.63} yields
\[
\big| \log(\theta_d(y,t)) - \log(\theta_d^\ast(x,r)) \big|
= \Big| \log\Big({\theta_d(y,t) r^d \over c_d(x,r) \mu(B(x,r))}\Big) \Big|
\leq C \alpha_d(x,r)
\]
We have a similar estimate for $\rho$, 
and \eqref{2.64} follows by adding the two.

We are left with the case when \eqref{2.66} fails.
In this case, we just observe that 
$\mu(B(x,\rho)) \leq \mu(B(x,r)) \leq \CD^2\mu(B(x,\rho))$
by the doubling property \eqref{1.1}, and since $c_d(x,r)$ and 
$c_d(x,\rho)$ both lie in $[1,2^d]$ by \eqref{2.14}, the definition \eqref{2.63}
yields $C^{-1} \leq {\theta_d^\ast(x,r)\over \theta_d^\ast(x,\rho)} \leq C$, 
which is enough for \eqref{2.64} because \eqref{2.66} fails.
Corollary~\ref{t2.11} follows.
\end{proof}

\ms
We end this section with another estimate on densities. This one is less precise 
than Lemma \ref{t2.10} or Corollary~\ref{t2.11}, but it is still rather useful.

\begin{lemma} \label{t2.12}
There is a constant $C \geq 0$, which depends only on $d$,
such that if $d \geq 1$,  $0 < \delta \leq 10^{-2}$, and $y \in B(x,r/3)$ are such that
$\dist(y,W_d(x,r)) \leq \delta r$, then
\begin{equation} \label{2.68}
\Big|{\mu(B(y, a r)) \over \mu(B(x,r))} - c_d(x,r) a^d \Big| 
\leq C \delta + 2 \delta^{-1} \alpha_d(x,r)
\ \text{ for } 0 < a < 1/3.
\end{equation}
\end{lemma}

\begin{proof}  
Let $\delta$, $y$, and $a$ be as in the statement,
pick $z\in W_d(x,r)$ such that $|z-y| = \dist(y,W_d(x,r)) \leq \delta r$, 
and set $\xi_y = (y-x)/r$ and $\xi_z = (z-x)/r$. Thus 
$\xi_y \in B(0,1/3)$ and $|\xi_x-\xi_z| \leq \delta$.
Let $\psi$ be a Lipschitz function such that $\1_{B(\xi_z, a + \delta)}\le \psi\le \1_{B(\xi_z, a +2 \delta)}$ and
$\psi$ is $\delta^{-1}$-Lipschitz.  Here $\1_B$ denotes the characteristic function of $B$. Notice also
that $\psi = 0$ on $\R^n \sm \bB$, because $|\xi_z| \leq 1/3 + \delta$
and $a \leq 1/3$.
We test \eqref{1.5} with $\delta^{-1}\psi$ and get that
\begin{equation} \label{2.70}
\Big| \int \psi d\mu_{x,r} - \int \psi d\nu_{V_d(x,r)} \Big|
\leq \delta^{-1} \W_1(\mu_{x,r}, \nu_{V_d(x,r)}) \leq 2 \delta^{-1} \alpha_1(x,r)
\end{equation}
by \eqref{2.13}. Notice that \eqref{2.15}, combined with the fact that $\xi_z \in V_d(x,r)$, the normalization \eqref{1.2},
and since $c_d(x,r) \leq 2^d$ by \eqref{2.14} and 
$(a + 2\delta)^d \leq a^d + C \delta$ for $a \leq 1$ and $\delta \leq 10^{-2}$ yields
\begin{eqnarray}\label{2.70A}
\int \psi d\nu_{V_d(x,r)} &\leq \nu_{V_d(x,r)} (B(\xi_z, a + 2\delta))
= c_d(x,r) \H^d(V_d(x,r) \cap B(\xi_z, a + 2\delta))\nonumber
\\
& = c_d(x,r) (a + 2\delta)^d  \leq c_d(x,r) a^d + C \delta 
\end{eqnarray}
On the other hand since $|\xi_y-\xi_z| \leq \delta$, and by the definition \eqref{1.6}
of $\mu_{x,r}$ we have
\begin{equation}\label{2.70B}
\int \psi d\mu_{x,r} \geq \mu_{x,r}(B(\xi_z, a + \delta))
\geq \mu_{x,r}(B(\xi_y,a))
= {\mu(B(y,ar)) \over \mu(B(x,r))}.
\end{equation}
Thus combining \eqref{2.70}, \eqref{2.70A} and \eqref{2.70B} we have
\[
{\mu(B(y,ar)) \over \mu(B(x,r))} \leq c_d(x,r) a^d + C \delta 
+2 \delta^{-1} \alpha_1(x,r)
\]
This gives an upper bound which is compatible with \eqref{2.68}.

For the lower bound, first observe that if $a \leq 2\delta$, the lower
bound coming from the fact that $\mu(B(y,ar)) \geq 0$ is enough,
because $a^d c_d(x,r) \leq C \delta$. 
This is where we use our extra assumption that $d \geq 1$.
Otherwise, use a different function $\psi$ defined such that $\1_{B(\xi_z, a -2 \delta)}\le \psi \le \1_{B(\xi_z, a - \delta)}$
 We still have \eqref{2.70}, and now
\begin{eqnarray}\label{2.71A}
\int \psi d\nu_{V_d(x,r)} &\geq \nu_{V_d(x,r)} (B(\xi_z, a - 2\delta))
= c_d(x,r) \H^d(V_d(x,r) \cap B(\xi_z, a - 2\delta))\nonumber
\\
&
= c_d(x,r) (a - 2\delta)^d  \geq c_d(x,r) a^d - C \delta,
\end{eqnarray}
while
\begin{eqnarray}\label{2.71B}
\int \psi d\mu_{x,r} \leq \mu_{x,r}(B(\xi_z, a - \delta))
\leq \mu_{x,r}(B(\xi_y,a))
= {\mu(B(y,ar)) \over \mu(B(x,r))}
\end{eqnarray}
as before. Thus we deduce from \eqref{2.70},  \eqref{2.71A} and \eqref{2.71B} that
\[
{\mu(B(y,ar)) \over \mu(B(x,r))} \geq c_d(x,r) a^d -  C \delta 
- 2 \delta^{-1} \alpha_1(x,r)
\]
which gives the lower bound needed for \eqref{2.68}.
Lemma \ref{t2.12} follows.
\end{proof}

\section{Proof of the decomposition result - Theorem \ref{t1.5}}
\label{proof1}

Let $\mu$ be a doubling measure, denote by $\Sigma$ its support, and
set 
\begin{equation} \label{3.1}
\Sigma_0 = \big\{ x\in \Sigma \, ; \, \int_0^1 \alpha(x,r) {dr \over r} < \infty \big\}.
\end{equation}
We want to cut $\Sigma_0$ into $d$-dimensional pieces
$\Sigma_0(d)$, $0 \leq d \leq n$. In order to do this
we first fix a point $x \in \Sigma$ and do some estimates which indicate
which $\Sigma_0(d)$ $x$ belongs to.

We do not need all the numbers $\alpha(x,r)$, one per dyadic interval is be enough.
For each integer $k \geq 0$, we set
\begin{equation} \label{3.2}
\alpha(k) = \inf\big\{ \alpha(x,r) \, ; \, r \in [2^{-k-1},2^{-k}] \big\}
\end{equation}
and then choose $r_k \in [2^{-k-1},2^{-k}]$ and $d = d(k)$ 
such that
\begin{equation} \label{3.3}
\alpha_d(x,r_k) = \alpha(x,r_k) \leq 2 \alpha(k).
\end{equation}
Notice that
\begin{equation} \label{3.4}
\sum_{k} \alpha(k) \leq \sum_k \fint_{[2^{-k-1},2^{-k}]} \alpha(x,r)
\leq 2\int_0^1 \alpha(x,r) {dr \over r} < \infty.
\end{equation}
Set
\begin{equation} \label{3.5}
V(k) = V_{d(k)}(x,r_k) \ \text{ and }Ê\ W(k) = x+ r_k V_{d(k)}(x,r_k)
\end{equation}
as in \eqref{2.17}. We restate some of the estimates from the previous section in this context.

\begin{lemma} \label{t3.1}
For each small $\beta > 0$, there exists $\alpha_1 > 0$, that depends
on $d$ and the doubling constant $\CD$, such that
\begin{equation} \label{3.6}
\dist(y,W(k)) \leq \beta 2^{-k} \ \text{ for } y\in \Sigma \cap B(x,2^{-k-2})
\end{equation}
and 
\begin{equation} \label{3.7}
\dist(y,\Sigma) \leq \beta 2^{-k} \ \text{ for } y\in W(k) \cap B(x,2^{-k-2})
\end{equation}
whenever $\alpha(k) \leq \alpha_1$, and 
\begin{equation} \label{3.8}
d(k+1) = d(k) \hbox{  provided  } \alpha(k) +\alpha(k+1) \leq \alpha_1.
\end{equation}
 \end{lemma}

\ms
\begin{proof}  
Recall from Lemma \ref{t2.6} that
\begin{equation} \label{3.9}
\dist(y,W(k)) \leq C \alpha(k)^\eta r_k 
\ \text{ for } y\in \Sigma \cap B(x,r_k/2),
\end{equation}
for some constants $C$ and $\eta$ that depend on the doubling
constant; \eqref{3.6} follows. Similarly,
Lemma \ref{t2.7} says that
\begin{equation} \label{3.10}
\dist(y,\Sigma) \leq 16 \alpha(k)^{1\over d(k+1)} r_k 
\ \text{ for } y\in W(k) \cap B(x,r_k/2),
\end{equation}
which implies \eqref{3.7}.
Finally assume that $\alpha(k)$ and $\alpha(k+1)$ are both small;
it follows from \eqref{3.6} and \eqref{3.7} (applied with $\beta = 10^{-3}$,
say) that 
\begin{equation} \label{3.11}
\dist(y,W(k+1)) \leq 2\, \beta 2^{-k} \ \text{ for } y\in W(k) \cap B(x,2^{-k-4}),
\end{equation}
and similarly with $W(k+1)$ and $W(k)$ exchanged. 
For $\beta$ small enough this forces $W(k+1)$ and $W(k)$ to have 
the same dimension.
\end{proof}

\ms
Although the proof above gives some estimate for the distance between
$W(k+1)$ and $W(k)$ we use the more precise ones given in Lemma \ref{t2.8}. 

In the mean time, observe that $\lim_{k \to \infty} \alpha(k) =0$,
by \eqref{3.4}, so the assumptions of Lemma \ref{t3.1} are satisfied for $k$
large, the sequence $\{ d(k) \}$ is stationary, and we can set
\begin{equation} \label{3.12}
d_{x} = \lim_{k \to \infty} d(k) \in [0,n].
\end{equation}
Naturally we take
\begin{equation} \label{3.13}
\Sigma_0(d) = \big\{ x\in \Sigma_0 \, ; \,  d_{x} = d \big\},
\end{equation}
and this gives the desired partition of $\Sigma_0$.

\ms
Next we want to estimate the densities. Fix $d$ and $x \in \Sigma_0(d)$,
and set
\begin{equation} \label{3.14}
\theta(k) = \theta_d^\ast(x,r_k)
= r_k^{-d} c_{d}(x,r_k) \mu(B(x,r_k)),
\end{equation}
where the second part is the definition \eqref{2.63}
of the stabilized density $\theta_d^\ast(x,r_k)$.
By Corollary \ref{t2.11}, 
\begin{equation} \label{3.15}
\Big| \log\Big({\theta(k+1) \over \theta(k)}\Big)\Big|
\leq C (\alpha_d(x,r_k) + \alpha_d(x,r_{k+1}))
\end{equation} 
with a constant $C$ that depends only on $d$ and $\CD$.
Since for $k$ large, $d(k) = d$ and 
$\alpha_d(x,r_k) = \alpha(x,r_k) \leq 2 \alpha(k)$ (by (3.3)),
\eqref{3.15} and \eqref{3.4} allow us to define
\begin{equation} \label{3.16}
\theta = \lim_{k \to \infty} \theta(k) \in (0,\infty).
\end{equation}

\ms
Let us get rid of the case when $d=0$. Suppose $d_x = 0$; then 
for $k$ large, $\alpha(k) = \alpha_0(x,r_k)$, and by \eqref{3.14}
$\theta(k) = \mu(B(x,r_k))$ (recall that when $d=0$, \eqref{2.14} readily
gives $c_d(x,r_k) = 1$ because $\nu_{V_0(x,r)}$ is a Dirac mass somewhere
in $B(x,r/2)$). 
By \eqref{3.16},  $\theta(k)$ has a positive limit,
and hence $\mu$ has an atom at $x$. Conversely, if $\mu$ has an atom at
$x$, Lemma \ref{t2.2} says that $x$ is an isolated point of $\Sigma$; then
for $r$ small, $\mu_{x,r}$ is a Dirac mass at the origin, $\alpha_0(x,r) = 0$,
we can take the index $d(x,r)$ above \eqref{2.16} equal to $0$ 
(in fact, we have to: it would be very easy to show that 
$\alpha_d(x,r) >0$ for $d \geq 1$), and hence we get that $d_x = 0$. 
Of course the set of points where $\mu$ has an atom (or where
$\Sigma$ has an isolated point) is at most countable, so we established the
description of $\Sigma_0(0)$ that was given in Part 1 of Theorem \ref{t1.5}.

\ms
We may now assume that $d_x \geq 1$.
Let us check that 
\begin{equation} \label{3.17}
\lim_{r \to 0} \theta_d(x,r) = \theta = \lim_{k \to \infty} \theta(k)
\end{equation}
(see \eqref{2.49}, \eqref{3.14}, and \eqref{3.16} for the definitions).
For each small $r > 0$, choose $k$ such that $2^{k-2} \leq r < 2^{k-1}$,
and then apply Lemma \ref{t2.12} to the ball $B(x,r_k)$, with $y=x$ and 
$a = r_k^{-1} r$; this yields 
\begin{equation} \label{3.18}
\Big|{\mu(B(x, r)) \over \mu(B(x,r_k))} - c_d(x,r_k) a^d \Big| 
\leq C \delta + 2 \delta^{-1} \alpha_d(x,r_k)
\end{equation}
for $0 < \delta \leq 10^{-2}$ such that 
$r_k^{-1}\dist(0,V_d(x,r_k)) \leq \delta$.
Let us take $\delta = C \alpha(k)^\eta$, where $C$ and 
$\eta$ are as in \eqref{3.9}; we can safely assume that $\eta \leq 1/2$
(otherwise, replace \eqref{3.9} with a worse estimate). Then \eqref{3.18}
yields 
\begin{equation} \label{3.19}
\Big|{\mu(B(x, r)) \over \mu(B(x,r_k))} - c_d(x,r_k) {r^d \over r_k^d} \Big| 
\leq C \delta + 4 \delta^{-1} \alpha(k)
\leq C' \alpha(k)^\eta 
\end{equation}
because for $k$ large, $\alpha_d(x,r_k) = \alpha(x,r_k) \leq 2\alpha(k)$
by \eqref{3.3}, and then by definition of $\delta$. 
Since $\lim_{k \to \infty} c_d(x,r_k) = 1$ by \eqref{3.9}, \eqref{2.14}, and
\eqref{1.2}, we see that $\theta_d(x,r)/\theta(x,r_k)$ tends to $1$
when $r$ tends to $0$ (see the definitions \eqref{2.49} and \eqref{3.14});
\eqref{3.17} follows at once.
Notice that Part 2 of Theorem \ref{t1.5} follows from this.

\ms
Let us now control the angles between the $W(k)$. 
We measure the angle between two vector spaces
$V$ and $W$ with the number
\begin{equation} \label{3.20}
\delta(V,W) = \| \pi_V - \pi_W \|,
\end{equation}
where $\pi_V$ is the orthogonal projection on $V$, 
$\pi_W$ is the orthogonal projection on $W$, and $\| \cdot \|$
is an operator norm on $\R^n$. Notice that $\delta(V,W)$
satisfies the triangle inequality.

Let $W^\ast(k)$ denote the vector space parallel to 
$W(k) = W_d(x,r_k)$ (for $k$ large); we claim that for $r_k$ as in \eqref{3.3}
\begin{equation} \label{3.21}
\delta(W^\ast(k), W^\ast(k+1)) \leq C (\alpha_d(x,r_k)+\alpha_d(x,r_{k+1})),
\end{equation}
with a constant $C$ that depends only on $d$ and $\CD$.

To see this, apply Lemma \ref{t2.8} to the pairs $(x,r_k)$ and 
$(x,r_{k+1})$, with $d=d_x$ and $a = 1/4$. We get that the two 
planes $W(k)$ and $W(k+1)$ are 
$C(\alpha_d(x,r_k)+\alpha_d(x,r_{k+1}))r_k$-close in 
$B(x,2r_k)$, and \eqref{3.21} follows.

Since $\sum_k \alpha(k) < \infty$ by \eqref{3.4}, 
and $\alpha_d(x,r_k) = \alpha(x,r_k) \leq 2\alpha(k)$
for $k$ large, by \eqref{3.3}, \eqref{3.12}, and because $d=d_x$,
so $\sum_k \alpha_d(x,r_k) < \infty$. We deduce from
\eqref{3.21} that the $W(k)^\ast$ converge to some vector space $W^\ast$,
and even that 
\begin{equation} \label{3.22}
\delta(W(k)^\ast, W^\ast) \leq C \sum_{l \geq k} \alpha_d(x,r_l),
\end{equation}
which tends to $0$. Denote by $W$ the $d$-plane through
$x$ parallel to $W^\ast$. Let us check that
\begin{equation} \label{3.23}
\text{$W$ is a tangent plane to $\Sigma$ at $x$.}
\end{equation}
Let $y \in \Sigma$ be given, and choose $k$ such that 
$2^{-k-3} \leq |y-x| \leq 2^{-k-1}$. Then $y\in B(x,r_k/2)$.
For each $\beta > 0$, \eqref{3.6} says that if $k$ is large enough
(i.e., if $|y-x|$ is small enough),
$x$ and $y$ both lie within $\beta 2^{-k}$ of $W(k)$.
Since $|y-x| \geq 2^{-k-3}$, this implies that 
$\dist(y-x, W^\ast(k)) \leq C \beta |y-x|$; \eqref{3.23} follows.

\ms
Our next step is to show that 
\begin{equation} \label{3.24}
\text{$\mu_{x,r}$ converges weakly to $\H^d|_{W^\ast}$ when $r$ tends to $0$.}
\end{equation}

By this we mean that for each continuous function $\psi$ with compact support,
\begin{equation} \label{3.25}
\lim_{r \to 0} \int \psi d\mu_{x,r} = \int_{W^\ast} \psi d\H^d.
\end{equation}
Let us first prove this when $\psi$ is Lipschitz. Let $R$ be such that $\psi$
is supported in $B(0,R)$, and let $r > 0$ be given. Choose $k$ such that
$2^{k-2} \leq r R \leq 2^{k-1}$, and notice that by \eqref{1.6} and two changes of variable we have
\begin{eqnarray} \label{3.26}
\int \psi d\mu_{x,r} &=& {1 \over \mu(B(x,r))} \int \psi(\frac{u-x}{r}) d\mu(u)
\nonumber \\
&=& {\mu(B(x,r_k)) \over \mu(B(x,r))} 
\int \psi\big({r_k v \over r}\big) d\mu_{x,r_k}(v).
\end{eqnarray}
 
Then 
\begin{equation}  \label{3.27}
\begin{aligned} 
\Big| \int \psi\big({r_k v \over r}\big) d\mu_{x,r_k}(v)
- \int \psi\big({r_k v \over r}\big) d\nu_{V_d(x,r_k)}(v) \Big|
&\leq {r_k \over r} \, ||\psi||_{\Lip} \W_1(\mu_{x,r_k}, \nu_{V_d(x,r_k)})
\\
\leq 4R ||\psi||_{\Lip} \alpha_d(x,r_k)
&\leq 4R ||\psi||_{\Lip} \alpha(k)
\end{aligned} 
\end{equation}
because $r_k \leq 2^{-k}$ (see above \eqref{3.3}),
by \eqref{1.5}, \eqref{2.13}, and for $k$ large.
In addition, \eqref{2.14} says that
\begin{equation}\label{3.27A}
\int \psi({r_k v \over r}) d\nu_{V_d(x,r_k)}(v)
= c_d(x,r_k) \int_{V_d(x,r_k)} \psi({r_k v \over r}) d\H^d(v).
\end{equation}
By \eqref{2.17}, \eqref{3.5}, and \eqref{3.9},
$\dist(0,V_d(x,r_k)) = 
r_k^{-1} \dist(x,W_d(x,r_k)) \leq C \alpha(k)^\eta$
tends to $0$, and by \eqref{3.22} the vector space $W(k)^\ast$
parallel to $W(k)$ and  $V_d(x,r_k)$ tends to $W^\ast$.
So $V_d(x,r_k)$ tends to $W^\ast$. Since in addition 
$c_d(x,r_k)$ tend it $1$ and $\psi$ is continuous and compactly supported, 
we get that using the fact that $R\le r_k/r\le 4R$ that 
\begin{equation}\label{3.27B}
\lim_{r\rightarrow 0} \left\{c_d(x,r_k) \int_{V_d(x,r_k)} \psi({r_k v \over r}) d\mathcal{H}^d(v)
- \int_{W^\ast} \psi({r_k v \over r}) d\H^d(v)\right\} =0
\end{equation}
Note that if $r\rightarrow 0$ then $k\rightarrow \infty$, and \eqref{3.27B} is shown by showing convergence to the same limit of every sub sequential limit.
Moreover since $W^\ast$ we have
\begin{equation}\label{3.27C}
\int_{W^\ast} \psi({r_k v \over r}) d\H^d(v)
= {r^d \over r_k^d} \int_{W^\ast} \psi(u) d\H^d(u)
\end{equation}
Hence combining \eqref{3.17},  \eqref{3.26},  \eqref{3.27},  \eqref{3.27A},  \eqref{3.27B} and  \eqref{3.27C} we have
\begin{eqnarray}
\lim_{r\rightarrow 0}\int\psi \, d\mu_{x,r} &=& 
\lim_{r\rightarrow 0}\,\frac{\mu(B(x,r_k)}{\mu(B(x,r))}c_d(x,r_k) \int_{V_d(x,r_k)} \psi({r_k v \over r}) d\mathcal{H}^d(v)\nonumber\\
&=&
\lim_{r\rightarrow 0}\,\frac{\mu(B(x,r_k)}{\mu(B(x,r))} \int_{W^\ast} \psi({r_k v \over r}) d\mathcal{H}^d(v)\nonumber\\
&=&
\lim_{r\rightarrow 0}\frac{\mu(B(x,r_k)}{r_k^d}\cdot\frac{r^d}{\mu(B(x,r))} \int_{W^\ast} \psi(u) d\mathcal{H}^d(u) \nonumber\\
&=&
 \int_{W^\ast} \psi(u) d\mathcal{H}^d(u)
\end{eqnarray}
Thus \eqref{3.25} holds for any Lipschitz function $\psi$.

Notice that $\mu_{x,r}(B(0,2^m)) = \mu(B(x,2^m r))/ \mu(B(x,r)) \leq \CD^m$
for $m \in \bN$; then for any continuous $\psi$ with compact support in some
$B(0,2^m)$, we can approximate $\psi$ uniformly by a sequence of Lipschitz
functions $\psi_j$ with support in $B(0,2^m)$, use the fact that
$\int |\psi-\psi_j| d\mu_{x,r} \leq \CD^m ||\psi-\psi_j||_\infty$
and $\int_{W^{\ast}} |\psi-\psi_j| d\H^d \leq C 2^m ||\psi-\psi_j||_\infty$,
and deduce \eqref{3.25} for $\psi$ from its analogue for the $\psi_j$.

This completes our proof of \eqref{3.24}. Note that  Part 3 of Theorem \ref{t1.5}
follows from \eqref{3.23}, \eqref{3.24} and Remma \ref{t2.1}.

\ms
We are left with Part 4 to check.
We cut $\Sigma_0(d)$ into the subsets 
\begin{equation} \label{3.28}
\Sigma_0(d,k) = \big\{ x\in \Sigma_0(d) \, ; \, 
2^k \leq \theta_d(x) < 2^{k+1} \big\}, \ k\in \bZ,
\end{equation}
as in the statement, and our first task is to show that 
\[
\H^d(\Sigma_0(d,k) \cap B(0,R)) < \infty \ \text{ for every } R>0.
\]
Set $A = \Sigma_0(d,k) \cap B(0,R)$ and let $\varepsilon \in (0,1)$ be given. 
For each $y\in A$, we can find  $r(x) > 0$ such that $r(x)< \varepsilon/5$ and 
\[
r(x)^{-d} \mu(B(x,r(x))) = \theta_d(x,r(x)) \geq 2^{-1}\theta_d(x) \geq 2^{k-1} ,
\]
where we denote by $\theta_d(x)$ the limit in \eqref{3.17} (see \eqref{1.15}).
By Vitali covering lemma 
(see Theorem 2.1 in \cite{Mattila} or the first pages of \cite{Stein}), 
we can find an a countable set $X \subset A$ such that the balls
$B(x,r(x))$, with $x\in X$, are disjoint, and the balls $B(x,5r(x))$, $x\in X$ cover $A$.
Note that since the $B(x,r(x))$ are disjoint and contained in $B(0,R+1)$ then
\[
\sum_{x\in X} (5r(x))^d \leq   \sum_{x\in X} 5^d 2^{1-k} \mu(B(x,r(x)))
\leq 5^d 2^{1-k} \mu(B(0,R+1))
\]
It follows that $\H^d(A) \leq C \mu(B(0,R+1)) < \infty$, which proves that $\mathcal{H}^d\res \Sigma_0(d,k)$ is Radon.
Since 
$\mu$ is Radon Lemma~2.13 in \cite{Mattila}, 
ensures that $2^k \H^d \leq \mu < 2^{k+1} \H^d$ on $\Sigma_0(d,k)$.
In particular these two restrictions are absolutely continuous with respect to
each other, and 
\begin{equation} \label{3.29}
\mu |_{\Sigma_0(d,k)} = \theta_d \H^d|_{\Sigma_0(d,k)},
\end{equation}
either by (2) of Theorem 2.12 of \cite{Mattila}, or by
Lemma 2.13 in \cite{Mattila}, applied to the subsets of $\Sigma_0(d,k)$ 
where $a \leq \theta_d(x) \leq b$.

Note that $\Sigma_0(d,k)$ has a unique tangent plane at each point. We claim that
\begin{equation} \label{3.30}
\Sigma_0(d,k) \text{ is rectifiable set,}
\end{equation}
The following argument is partially taken from \cite{Mattila}.

Set $E = \Sigma_0(d,k)$, and let $\varepsilon > 0$ be given.
For each $x \in E$, we can find an integer $j$ and a $d$-dimensional
vector space $V$ such that
\begin{equation} \label{3.31}
\dist(y,x+V) \leq \varepsilon |y-x| \ \text{ for } y \in E \cap B(x,2^{-j+2}).
\end{equation}
Since $G(d,n)$ is compact $V$ can be chosen to be a finite collection $\cV$
(that depends on $\varepsilon$). Set
\[
E(V,j) = \big\{ x\in E \, ; \, \eqref{3.31} \text{ holds} \big\}
\]
for each $V \in \cV$ and $j \geq 0$.
It is easy to see that the intersection of $E(V,j)$ with any ball 
of radius $2^{-j}$ is contained in a Lipschitz graph over $V$;
this shows that $E$ is contained in a countable union of Lipschitz
graphs (with small constants if needed). Since $\H^d\res E$
is locally finite, this completes our proof of Theorem \ref{t1.5}.

\section{Some examples}
\label{examples}  
The following simple examples illustrate some of the complications that may arise with measures satisfying 
the hypothesis of Theorem ref{t1.5}. We leave most of the computational details to the reader.  
Note that in case the examples live on a compact set
(like the unit cube), the doubling property \eqref{1.1} is only
satisfied for $r \leq 1$. A truly doubling measure would be easy to
construct from this one ( in fact take the constructed measure $\mu$ on the unit cube for example
and add all the translations of $\mu$ by vectors in $(4 \bZ)^n$).
Let
\begin{equation} \label{4.1}
J(x) = \int_0^1 \alpha(x,r) {dr \over r}
\end{equation}

\begin{example} \label{t4.1} {\it Limits of Dirac masses}.
Take $n=1$ and $\mu = \sum_{j \in \bZ} a_j \delta_{x_j}$,
where the $x_j$ are points of $\R$, $\delta_{x_j}$ is a Dirac mass at $x_j$,
and $a_j > 0$.
\end{example}

For the first example take $x_j = 2^{-j}$ and $a_j = 4^{-j}$.
We add $j \leq 0$ so that $\mu$ is doubling, even at the large scales.
Then $\Sigma_0(0) = \big\{ x_j \, ; \, j \in  \bZ \big\}$, and
$\Sigma = \Sigma_0(0) \cup \{ 0 \}$. In this case 
$J(x_j) = \int_{0}^1 \alpha(x_j, r) {dr \over r}$ is of the order of $j$
(all the radii $r < 2^{-j-1}$ yield $\alpha(x_j,r) = \alpha_0(x_j,r) = 0$), 
and $J(0) = \infty$.

We can let $0$ lie in a more significant part of $\Sigma$. Take the same example,
plunge $\R$ in $\R^3$ in the obvious way, and add to $\mu$ the restriction
of $\H^2$ to the plane $P$ orthogonal to $\R$; We obtain a doubling measure
$\mu_1$, essentially because $\mu(B(0,r))$ was of the order of $r^2$.
Now $\Sigma$ contains $P$, and $P \sm \{ 0 \}\subset \Sigma_0(2)$.

We may even construct $\mu$ so that $J$ is bounded 
(and in particular $J(0) < \infty$). 
Keep $x_j = 2^{-j}$ and $a_j = 4^{-j}$ for $j \leq 0$ (again, just to take care
of doubling for large scales), and for $j \geq 0$, choose the $x_j$ to be slowly
decreasing to $0$, and take $a_j = |x_j|-|x_{j+1}|$. In this case
\begin{equation} \label{4.2}
\mu(\overline B(0,x_j)) = |x_j|
\end{equation}
For instance, take $x_j = \log(j)^{-1}$ for $j$ large;
then
\[
a_j = |x_j|-|x_{j+1}| 
= {\log(j+1)^{-1}-\log(j)^{-1} \over \log(j) \log(j+1)}
\sim j^{-1} \log(j)^{-2} \sim |x_j|^2 e^{-1/|x_j|},
\]
which is much smaller than $|x_j|$. 
We add to $\mu$ the restriction to $(-\infty,0]$ of the Lebesgue measure, 
(or the image of $\mu$ by the symmetry with respect to the origin),
and we get a measure $\mu_2$, which is doubling and for which $J$ is bounded.

The main ingredient in the proof that $\mu_2$ is \eqref{4.1} 
(the measure of the atoms is essentially the same as the length of the holes, which means
that at the larger scales, $\mu$ looks a lot like the Lebesgue measure on the line). 
To show that $J$ is bounded note that since relative size of the gaps ${a_j \over |x_j|} = {|x_j|-|x_{j+1}| \over |x_j|}$ goes to $0$ 
 fast enough then the $\alpha$'s tend to 0 very fast
This provides an example where $\mu$ and $\H^d$ 
(here, with $d=0$) are mutually absolutely continuous on $\Sigma_0(d)$, 
but yet $\mu(B(0,1)) < \infty$ and $H^d(\Sigma_0(d) \cap B(0,1)) = \infty$.

\begin{example} \label{t4.2} {\it String of spheres}.
Take $n \geq 3$, pick a unit vector $e \in \R^n$
and a small $\varepsilon > 0$. Set
\begin{equation} \label{4.3}
\nu_\varepsilon(x) = \H^{n-1}|_{\d B(x,\varepsilon/10)}
\ \text{ for } x \in \R^n,
\end{equation}
and then 
\begin{equation} \label{4.4}
\mu = \varepsilon^{2-n}\sum_{j \in \bZ} \nu_{\varepsilon}(j\varepsilon e).
\end{equation}
We claim that $\mu$ is doubling and $J$ is bounded, with bounds that do not depend
on $\varepsilon$, that $\Sigma = \Sigma_0(n-1)$,
but yet the the density $\theta_{n-1}(x) = \varepsilon^{2-n}$ is arbitrarily
large, while $\mu(B(0,1))$ stays bounded.
\end{example}

\ms
We chose the coefficient $\varepsilon^{2-n}$ so that
$\mu(B(x,r)) \sim cr$ for $\varepsilon\;; r$ (where $c $
is a normalizing constant). Note that for $r > 100\varepsilon$,
$\alpha(x,r) = \alpha_1(x,r) \leq C \varepsilon/r$,
and for $r < 10^{-2}\varepsilon$, 
$\alpha(x,r) = \alpha_{n-1}(x,r) \leq r/ \varepsilon$.
We use spheres rather than balls to make sure that this happens
for every $x\in \Sigma$. Had we used balls, we would get that $J$ is small on
average, but $J(x) = \infty$ at points of the boundary.
We have to use the density $\theta_{n-1}(x)$, because $\Sigma = \Sigma_0(n-1)$.
At large scales, $\Sigma$ and $\mu$ look $1$-dimensional so the
density changes a lot.

To make the example more pathological
take a decreasing sequence $\{ x_j \}$ in $(0,\infty)$,
that converges slowly to $0$, for instance such that $x_j =\log(j)^{-1}$
for $j > 2$, set $a_j = |x_j|-|x_{j+1}|$, and consider
\begin{equation} \label{4.5}
\mu = \sum_j a_j^{2-n} [\nu_{a_j}(x_j e)+\nu_{a_j}(-x_j e)].
\end{equation}
In small balls $B(0,r)$, this looks more and more like a multiple of the 
Lebesgue measure on $\R^{n-1}$. We claim that $\mu$ is doubling, that $J$ is bounded,
and that $\Sigma = \Sigma_0(n-1) \sm \{ 0 \} $. But on the small sphere
centered on $x_j e$, the density $\theta_{n-1}(\cdot)$ is about $a_j^{2-n}$,  
which goes to $\infty$ rapidly.

\ms

The next example is similar, but we want small densities. We 
consider low-dimensional measures that look large-dimensional.

\begin{example} \label{t4.3} {\it Ocean of circles}.
Take $n \geq 2$ and $\varepsilon > 0$, and for $x\in \R^n$,
choose a circle $c_\varepsilon(x)$ centered at $x$ and with radius 
$10^{-1} \varepsilon$ and set $\nu_{\varepsilon}(x) = \H^1|_{c_\varepsilon(x)}$.
Then take
\begin{equation} \label{4.6}
\mu_\varepsilon = \sum_{x\in \varepsilon \bZ^{n}} \varepsilon^{n-1} \nu_{\varepsilon}(x).
\end{equation}
\end{example}

The measure $\mu$ is doubling, 
$J$ is bounded, 
$\Sigma = \Sigma_0(1)$, and $\mu$ is normalized so that
$\mu_\varepsilon(B(x,1))$ is roughly $1$ for all $x\in \Sigma$. 
All this holds with constants that do not depend on $\varepsilon$, and yet 
the density $\theta_1(y) = \varepsilon^{n-1}$ is as small as we want.

Here $\alpha(x,r) = \alpha_n(x,r) \leq C \varepsilon/r$ for $\varepsilon\ll r$,
$\alpha(x,r) = \alpha_1(x,r) \leq r/\varepsilon$ for $r\ll\varepsilon$.
The density $\theta_1$ changes a lot in the intermediate region where
$\varepsilon < r < 1$. A priori it is impossible to predict how wide this range is.

Of course it is easy to produce variants for which this happens all over the place,
with very small densities. For instance, pick a collection
of dyadic cubes $Q_j$, so that the $10Q_j$ are disjoint
(think about many small cubes spread out all over the place), denote by 
$r_j$ the sidelength of $Q_j$, choose very small dyadic numbers $\varepsilon_j << r_j$,
and set 
\[
\mu = \H^n|_{\R^n \sm \cup_j 2Q_j} + \sum_j \mu_{\varepsilon_j,Q_j},
\]
where 
\[
\mu_{\varepsilon_j,Q_j} = {2^n \varepsilon_j^{n-1} \over \H^1(c_1(0))} 
\sum_{x\in Q_j \cap \varepsilon_j \bZ^n}  \H^1|_{c_{\varepsilon_j}(x)},
\]
where we normalize $\mu_{\varepsilon_j,Q_j}$ so that 
$\mu_{\varepsilon_j,Q_j}(Q_j)$ is roughly $\H^n(2Q_j)$.
The additional coefficient ${2^n \over \H^1(c_1(0))}$ 
should make the verification of the doubling property a little easier,
because $\mu_{\varepsilon_j,Q_j}(2Q_j) = \H^{n}(2Q_j)$, so we
did not change the total mass, we just moved it a bit.

In this example $J$ is not bounded
because the $\H^n$ part of the measure has sharp edges.
This is not a major issue and can be solved by adding a 
1-dimensional smoothing edges with a droplet like profile.
Figure 1 depicts the intersection of the support of the modified measure
$\mu$ with the region between two vertical planes, near a cube $Q_j$
(that was made more rectangular for the sake of the picture).

\includegraphics[height=3cm]{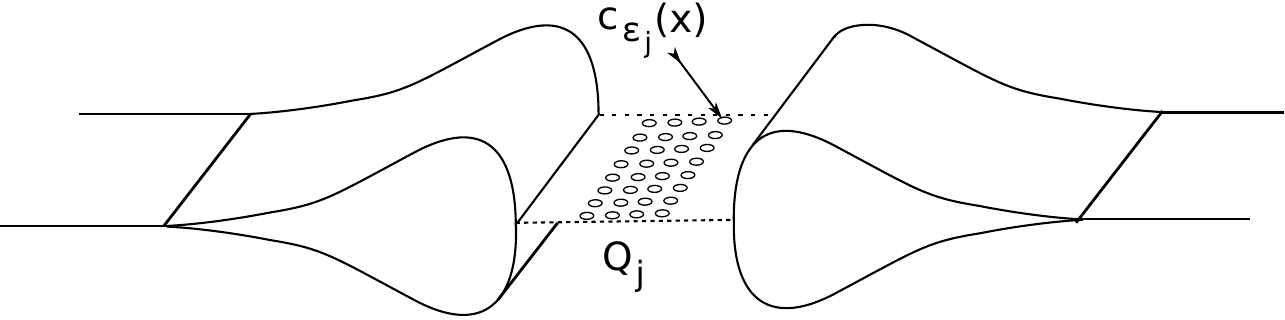}

\medskip
\centerline{{\bf Figure 1.} The support of $\mu$ between two vertical
planes and with a single $Q_j$.}
\bigskip
Notice that 
\begin{equation} \label{4.7}
\H^{1}(\Sigma_0(1))
= \sum_j \H^{1}(\Sigma_0(1) \cap 2Q_j)
= C \sum_j \varepsilon_j^{1-n} r_j^n = \infty
\end{equation}
if we choose the $\varepsilon_j$ small enough. So $\H^1$ and 
$\mu$ are mutually absolutely continuous on $\Sigma_0(1)$, 
but one is locally finite and the other one is not.

A similar construction with nested cubes $Q_j$ centered at the origin,
where one would use $\varepsilon_j$ on $Q_j \sm Q_{j+1}$, would
give an example where $\Sigma_0(1) = \Sigma \sm \{0\}$. An additional modification
could ensure that $0 \in \Sigma_0(1)$ (e.g. make $\Sigma$ thinner near a line on the $Q_j$, and
compensate by taking $\varepsilon_j$ even smaller).

The next example shows that $\Sigma \sm \Sigma_0$ may be
large, and even have a larger dimension than $\Sigma_0$.

\begin{example} \label{t4.4} {\it Snowflakes with jewelry}.
We can find $\mu$ in $\R^2$ such that $\Sigma$ is $\delta$-dimensional
for some $\delta > 1$, $J(x) < \infty$ for $\mu$-almost every
$x\in \Sigma$, and $\Sigma_0 = \Sigma_0(1)$.
\end{example}

Let $E \subset \R^2$ be a $\delta$-dimensional snowflake, and $\nu$ a
measure on $E$, such that 
\begin{equation} \label{4.8}
C^{-1} r^\delta \leq \nu(B(x,r)) \leq C r^\delta
\ \text{ for $x\in E$ and } 0 < r < 1.
\end{equation}
Such sets and measures are easy to construct, at least if $\delta > 1$ close to $1$.
Thus assume $\delta$ is close to 1.

For $j \geq 0$, choose a maximal subset $A_j$ of $E$ for which
$|x-y| \geq 100^{-j}$ for $x \neq y \in A_j$. Then, for each 
$x \in A_j$, choose $z \in B(x,100^{-j}/3)$ such that
$\dist(z_x,E) \geq 100^{-j}/10$. Denote by 
$c_j(x)$ the circle centered at $z_x$ and with radius $100^{-j}/20$
and set $\mu_{j,x}= (\H^1(c_j(x)))^{-1} \H^1|_{c_j(x)}$.
Let
\begin{equation} \label{4.9}
\mu = \sum_{j \geq 0} \sum_{x \in A_j} 100^{-j \eta} \mu_{j,z_x},
\end{equation}
where we may choose any constant $\eta > \delta$.
The geometric constants were chosen so that the $c_j(x)$
are far enough from $E$ and from each other (even for different
$j$'s). The condition $\eta > \delta$ is our way to make sure that $\mu$
is locally finite. One can check that
\begin{equation} \label{4.10}
C^{-1} r^{\eta} \leq \mu(B(x,r)) \leq C r^{\eta}
\ \text{ for $x\in E$ and } 0 < r < 1;
\end{equation}
the major contribution comes from the bounded number of
circles $c_j(x)$ for which $100^{-j}$ is roughly equal to $r$,
and then there is a convergent geometric series coming from
the contribution of larger indices $j$. The Ahlfors regular property of $\nu$
\eqref{4.8} is used to estimate the number of points of $A_j \cap B(x,2r)$.

We can deduce from \eqref{4.10} that $\mu$ is doubling.
For balls centered on $E$, we use \eqref{4.10}; for balls that
do not meet $E$, we use the fact that the $c_j(x)$ are far from each
other. For the intermediate balls, we reduce to the previous cases.

Now  $\Sigma = E \cap \Sigma_0(1)$, where $\Sigma_0(1)$
is the union of the added circles $c_j(x)$. The fact that 
$J(y) < \infty$ for $y\in c_j(x)$ is trivial, but of course we do not
get good average bounds $J(y)$. Theorem \ref{t1.5} gives
some (non uniform) control on $\Sigma_0(1)$, and nothing on
$E$. We claim that $J(x) = \infty$ on $E$.

This is another case where $\H^1(\Sigma_0(1))$ is not locally finite. This 
 seems to be needed in the construction (significantly smaller circles would not work, 
because they would be too far from each other and $\mu$ would not be doubling).

We can also add less circles so that $\H^1(\Sigma_0) < \infty$. In this case
we need to add $\nu$ to keep the measure doubling. That is,
choose for each $j$ a single point $x(j) \in A_j$, and set
\begin{equation} \label{4.11}
\mu = \nu + \sum_{j \geq 0} 100^{-j \delta} \mu_{j,x(j)}.
\end{equation}
This time we match $\mu(c_j(x))$ with $\nu(B(x,100^{-j})$, because
this way it is easy to show that $\mu$ is doubling. We are
really using $\nu$ for this. As before we claim that $J(x) = \infty$
on $E$, so $\Sigma_0 = \Sigma_0(1) = \bigcup_j c_j(x(j))$ and 
\begin{equation} \label{4.12}
\H^1(\Sigma_0) = \sum_{j \geq 0} \H^1(c_j(x(j))) < \infty.
\end{equation}
But we can still choose the $x(j)$ so that they are dense in $E$,
hence $E \cup \Sigma_0$ is also the support of $\mu|_{\Sigma_0}$.

\ms
\begin{example} \label{t4.5}
There is a measure $\mu$, which is doubling and uniformly rectifiable
of dimension $1$ (in the sense that the conclusion of Theorem \ref{t1.6} holds,
with $d=1$, for all $x\in \Sigma$ and $0 < r < \infty$), and even satisfies 
the Carleson condition of Theorem \ref{t1.8}, with a large constant $C_1$, 
but for which $\H^1(\Sigma \cap B(0,1)) = \infty$.
\end{example}

\begin{proof} 
The example requires additional notation.
We work in $\R^3$ using a square Cantor set.
Start in $\R^2$, with the four points $e_1 = (1,0)$, $e_2 = (0,1)$,
$e_3 = (-1,0)$ and $e_4 = (0,-1)$. Set $A = \{ e_1, e_2, e_3, e_4 \}$,
which we see as an alphabet with $4$ letters. Pick a number
$\rho \in (0,1/2)$, and consider the sets 
\begin{equation} \label{4.14} 
E_k = \big\{ x= \sum_{j=1}^k \rho^j \varepsilon_j \, ; \, 
(\varepsilon_1, \ldots, \varepsilon_n) \in A^k \big\}.
\end{equation}
Thus $E_k$ is composed of $4^k$ points, which are all different
because $\rho < 1/2$, and the sets $E_k$ converge to a cantor set
$E_\infty$. To each $E_k$ we associate the set
$$
F_k = E_k \times [\rho^{k+1},\rho^k] \subset \R^3,
$$
where the product is taken with intervals in the last (vertical) direction.
That is, $F_k$ is the union of $4^k$ parallel vertical segments.
We set
$$
\Sigma = \big(\{ 0 \} \times [1,\infty) \big)\cup \big(E_\infty \times \{ 0 \}\big) 
\cup \bigcup_{k \geq 1} F_k 
$$
(we added the first piece to get an infinite set, and the second one to get
a closed set). Thus $\Sigma$ looks like some sort of futuristic broom with a
long stick and many small hairs. See Figure 2 for a first approximation of $\Sigma$, 
with the three first sets $F_k$.

\centerline{
\includegraphics[height=4.5cm]{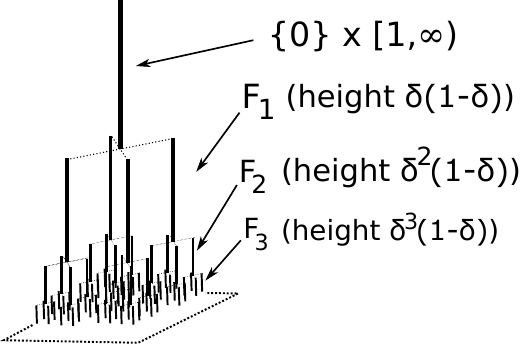}
} 

\medskip
\centerline{{\bf Figure 2.} Part of the broom $\Sigma$ (3 generations).}
\bigskip

On $\{ 0 \} \times [1,\infty)$, we put the Lebesgue measure.
On each of the $4^k$ segments of $F_k$, we put $4^{-k}$ times
the Lebesgue measure. And we put no additional mass on $E_\infty \times \{ 0 \}$.
This gives a measure $\mu$ whose support is $\Sigma$.
Notice that when we push forward $\mu$ onto the vertical axis, we get the
restriction of the Lebesgue measure to $[0,\infty)$. In particular,
\begin{equation} \label{4.15}
\mu(B(x,r)) \leq 2r \ \text{ for $x\in \Sigma$ and } r > 0.
\end{equation}

We need to evaluate the measure of a ball $B = B(x,r)$,
and we start when $x \in F_\infty = E_\infty \times \{ 0 \}$.
When restrict to $r < 1$ (otherwise, \eqref{4.15} will be enough).
Define the integer $k(r)$ by
\begin{equation} \label{4.16}
\rho^{k(r)+1} \leq r < \rho^{k(r)}.
\end{equation}
There is no contribution in $B$ from the $E_k$, $k < k(r)$, because 
$\dist(F_k,F_\infty) \geq \rho^{k+1}$.
For $k\ge k(r)$, let $N(r,k)$ be the maximal number of segments of
$F_k$ that meet a ball of radius $r$. This amounts to counting how many
points of $E_k$ lie in a ball of radius $r$, and we get at most $C4^{k-k(r)}$
(look at the expansion in \eqref{4.14}, and notice that only the first digits
are determined by $B$, up until $\rho^k \sim r$. There are no restrictions on the rest of the digits).
Here $C$ may be large if we took $\rho$ close to $1/2$, but we don't intend to
do that. We have to multiply $N(r,k)$ by the measure of each segment, which is
$4^{-k}(\rho^k-\rho^{k+1})$. We get that
\begin{equation} \label{4.17}
\mu(B(x,r)) \leq \sum_{k \geq k(r)}4^{-k}(\rho^k-\rho^{k+1}) N(r,k) 
\leq C \sum_{k \geq k(r)} 4^{-k(r)} \rho^k \leq C 4^{-k(r)} \rho^{k(r)}.
\end{equation}
Conversely, we claim that $B$ contains a full segment of $E_{k(r)+2}$,
and hence
\begin{equation} \label{4.18}
\mu(B(x,r)) \geq 4^{-k(r)-2} (\rho^{k+2}-\rho^{k+3}) 
\geq C^{-1} 4^{-k(r)} \rho^{k(r)},
\end{equation}
where this time $C$ may be large if $\rho$ is small.
Write $x \in F_\infty \simeq E_\infty$ as 
$x = \sum_{j \geq 1}\rho^j \varepsilon_j$, and set
$x_k = \sum_{j=1}^k \rho^j \varepsilon_j$ for $k \geq 1$.
Then $|x-x_k| \leq \rho^k$; applying this to $k = k(r)+2$
gives a point $x_{k(r)+2}\in E_{k(r)+2}$ such that
$|x-x_{k(r)+2}| \leq r/2$, and the claim follows.

The doubling property for balls $B(x,r)$ centered on $F_\infty$
easily follows from the estimates above. Notice that the proof also shows that
\begin{equation} \label{4.19}
\begin{aligned}
&\text{there is a line segment $L$ from some $F_k$, such that }
\\
&\hskip0.8cm
\H^1(L) \geq C^{-1}r \quad\text{ and } \quad\mu(L) \geq C^{-1} \mu(B).
\end{aligned}
\end{equation}

Now consider a ball $B(x,r)$ such that $x\in F_k$. Observe that
\begin{equation} \label{4.20}
r 4^{-k} \leq \mu(B(x,r)) \leq 2 r 4^{-k}
\ \text{ for } 0 < r \leq C_1^{-1} \rho^{k},
\end{equation}
just because in this case $B(x,r)$ does not contain anything else
than the line segment of $E_k$ that contains $x$. This gives
the doubling properties for balls of radius $r \leq (2C_1)^{-1} \rho^{k}$.
For the slightly larger radii, notice that
\begin{equation} \label{4.21}
C^{-1} \rho^k 4^{-k} \leq \mu(B(x,r)) \leq C \rho^k 4^{-k}
\ \text{ for } C_1^{-1} \rho^{k} \leq r \leq 4\rho^k,
\end{equation}
by the first part of \eqref{4.20} and \eqref{4.17}, which takes care of
$r \leq 2\rho^k$; for radii $r \geq 2\rho^k$ we just use 
\eqref{4.17} and \eqref{4.18}. So $\mu$ is doubling, and we see
that \eqref{4.19} remains valid when $x\in \Sigma \sm F_\infty$
and $0 < r \leq 1$.

Notice that \eqref{4.19} implies that $\Sigma$ intersection any ball centered in $\Sigma$
contains big line segments. In the case that $\mu$ is Ahlfors regular this implies uniform
rectifiability.
 To prove uniform rectifiability in this setting we need to produce
for very large pieces of bi-Lipschitz images, 
that is we should check that for each $\gamma > 0$,
we can find $N \geq 1$ such that for each ball $B$ centered on $\Sigma$, there
is a collection of at most $N$ segments $L_k$ of different $F_k$, such that
$\mu(B \sm \cup_k L_k) \leq \gamma \mu(B)$, and then check that
$\cup_k L_k$ is bi-Lipschitz-equivalent to a subset of $\R$. We leave
the verification to the reader, who may also use the Carleson
estimate below and Theorem \ref{t1.6} .

Now we want to check that the numbers $\alpha_1(x,r)$ satisfy the 
Carleson condition
\begin{equation} \label{4.22}
\int_{ B(x,2r)} \int_{ 0}^{2r} \alpha_1(y,t) {d\mu(y) dt \over t}
\leq C \mu(B(x,r))
\ \text{ for $x \in \Sigma$ and $r > 0$.} 
\end{equation}
for $x \in \Sigma$ and $r > 0$. This means that we should evaluate the functions
\begin{equation} \label{4.23}
J_r(y) = \int_{0}^r \alpha_1(y,t) {dt \over t}
\end{equation}
for $y \in \Sigma$ and $0 < r < 1$. 

For a single segment $L$ of $F_k$, we would get that since
$\mu$ is proportional to the Lebesgue measure on $L$,
$\alpha_1(y,t) = 0$ for $0 < r < \dist(y,\d L)$, and
$\alpha_1(y,t) \leq 2$ (trivially) for $r \geq \dist(y,\d L)$.
Although there may be other pieces of $\Sigma$ floating around,
but they do not come closer than $C^{-1} \rho^k$ from $L$,
so we get that
$$
\alpha_1(y,t) = 0 \text{ for $t < C^{-1} \dist(y,\d L)$ and
$\alpha_1(y,t) \leq 2$ otherwise.}
$$
Set $\delta(y) = \dist(y,\d L)$ when $L$ is the segment of
some $F_k$ that contains $y$; we are not interested in $y\in F_\infty$
here. We just showed that
\begin{equation} \label{4.24}
J_r(y) \leq  C \log_+\big({C r \over \delta(y)}\big)
\leq C + \log_+\big({r \over \delta(y)}\big)
\end{equation}
for $y\in \Sigma \sm F_\infty$.

Now let us fix $x \in \Sigma$ and $r >0$ and evaluate 
\begin{equation}\label{4.23A}
A(x,r) = \int_{ B(x,2r)} \int_{0}^{2r} \alpha_1(y,t) {d\mu(y) dt \over t}
= \int_{ B(x,2r)} J_r(y) d\mu(y).
\end{equation}
Cut $B(x,2r) \cap \Sigma \sm F_\infty$ into segments $L_{j,k} \subset F_k$.
Denote by $l_{j,k}$ the length of $L_{j,k}$; then by \eqref{4.24}
\begin{equation} \label{4.25}
\int_{L_{j,k}} J_r(y) d\mu(y) \leq C \mu(L_{j,k}) 
\Big(1+ \log_+\big({r \over l_{j,k}}\big)\Big)
\end{equation}
Here we have used the fact that 
$\log_+\big({r \over \delta(y)}\big) \leq \log_+\big({r \over l_{j,k}}\big)
+ \log_+\big({l_{j,k}\over \delta(y)}\big)$, and that the average of the second 
one is bounded.

Denote by $L^\ast_{j,k}$ the segment of $E_k$ that contains
$L_{j,k}$ and by $l^\ast_{j,k}$ its length. Let us check that
\begin{equation} \label{4.26}
\mu(L_{j,k}) \Big(1+ \log_+\big({r \over l_{j,k}}\big)\Big)
\leq C \mu(L^\ast_{j,k} \cap B(x,3r)) \Big(1+ \log_+\big({r \over l^\ast_{j,k}}\big)\Big).
\end{equation}
Set $l=l_{j,k}$ and  $l^\ast = l^\ast_{j,k}$. When $l< r$,  
it is enough to show that 
\begin{equation}\label{4.25A}
l \log_+({r \over l}\big) \leq C l^\ast \big(1+ \log_+({r \over l^\ast}\big)\big).
\end{equation}
Note that $l \log_+({r \over l^\ast}\big)$ is controlled by the right-hand side
and since 
$\log\big({l^\ast \over l}\big) \leq C {l^\ast \over l}$, \eqref{4.25A} holds in this case
When $l\geq r$, 
$ \log_+\big({r \over l}\big)=0$ thus \eqref{4.25A} holds trivially.
So \eqref{4.26} holds and \eqref{4.25} yields
\begin{equation} \label{4.27}
\int_{L_{j,k}} J_r(y) d\mu(y) \leq C \mu(L^\ast_{j,k} \cap B(x,3r)) 
\Big(1+ \log_+\big({r \over l^\ast_{j,k}}\big)\Big).
\end{equation}
Now we count how many indices $j$ may correspond to a given $k$.
Let $k(r)$ be as in \eqref{4.16}; for $k \leq k(r)$, there are 
at most $C$ line segments from $F_k$ that meet $B(0,2r)$, and 
\begin{eqnarray}  \label{4.28}
\sum_{k \leq k(r)} \sum_{j} \int_{L_{j,k}} J_r(y) d\mu(y)
&\leq& C \sum_{k \leq k(r)} \sum_{j} \mu(L^\ast_{j,k} \cap B(x,3r))
\nonumber \\
&\leq& C \mu(B(x,3r)) \leq C \mu(B(x,r))
\end{eqnarray}
by \eqref{4.27} and because the $L^\ast_{j,k}$ are disjoint.
For $k > k(r)$, we computed earlier that there are at most
$N(r,k) \leq C 4^{k-k(r)}$ segments $L^\ast_{j,k}$, so
by \eqref{4.27}, and since $l^\ast_{j,k} = \rho^{k} - \rho^{k+1}
\geq {1 \over 2} \rho^{k} \geq C^{-1} \rho^{k-k(r)} r$ 
(recall that $\rho^{k(r)} \sim r$) we have
\begin{align*}
\sum_{j} \int_{L_{j,k}} J_r(y) d\mu(y)
&\leq C \sum_{j} \mu(L^\ast_{j,k} \cap B(x,3r)) (1+ k - k(r))
\\&
\leq C 4^{k-k(r)} [\rho^k 4^{-k}] (1+ k - k(r))
= C (1+ k - k(r)) 4^{-k(r)} \rho^k.
\end{align*}
Here $[\rho^k 4^{-k}]$
accounts for the measure of a single $L^\ast_{j,k}$
We sum over $k$ and get that
\begin{equation} \label{4.29}
\sum_{k > k(r)} \sum_{j} \int_{L_{j,k}} J_r(y) d\mu(y)
\leq C 4^{-k(r)} \rho^{k(r)}.
\end{equation}
Recall from \eqref{4.19} that there is a single line segment $L \subset B(x,r)$ 
such that $\H^1(L) \geq C^{-1} r \geq C^{-1} \rho^{k(r)}$,
and for which $\mu(L) \geq C^{-1}\mu(B(x,r))$. Because of its length,
it comes from an $F_j$, $j \leq k(r) -C$, which means that
$$
\mu(B(x,r)) \geq \mu(L) = 4^{-j} \H^1(L) \geq C^{-1} 4^{-k(r)} \H^1(L)
\geq C^{-1} 4^{-k(r)} \rho^{k(r)}.
$$
Thus the sum in \eqref{4.29} is less than $C \mu(B(x,r))$; we add this to 
\eqref{4.28} and get that $A(x,r) \leq C \mu(B(x,r))$. This proves
the Carleson measure estimate \eqref{4.22}.

We want to check that for some choices of $\rho$,
$\H^1(\Sigma) = \infty$. The total length of $B(0,2) \cap \Sigma \sm F_\infty$ is
\begin{equation} \label{4.30}
\H^1(B(0,2) \cap \Sigma \sm F_\infty) \geq \sum_k \H^1(F_k) 
= \sum_k 4^k (\rho^k - \rho^{k+1}) = \infty
\end{equation}
as soon as $\rho \geq {1 \over 4}$. When $\rho = {1 \over 4}$, we already
get a nice additional unrectifiable limit set $F_\infty$, but with
finite length. When ${1 \over 4} < \rho < {1 \over 2}$, $F_\infty$
and then also $\Sigma$ have a dimension larger than $1$.

This completes the study of Example \ref{t4.5} and shows that Theorem \ref{t1.8} cannot be 
drastically improved.

 \end{proof}

\ms

\ms
Next we say a few words about the case when $n=d=1$,
and $\mu$ is a measure whose support is $\Sigma = \R$.
Then there is no special difficulty with the geometry of $\Sigma$,
but still we do not understand well the density properties of 
$\mu$. We only know of one interesting type of examples, 
namely the Riesz measures and their variants. The simplest Riesz products
are limits of measures $\mu_N = F_N dx$, where
\begin{equation} \label{4.13}
F_N(x) =  \prod_{k=1}^N \big(1+\alpha_k \cos(3^k x)\big),
\end{equation}
with coefficients $\alpha_k < 1$ to make sure that the finite products stay
positive. If  $\sum_k \alpha_k^2 < \infty$, the 
infinite products converges almost everywhere, and the weak limit
of the $\mu_N$ is absolutely continuous with respect to the Lebesgue measure.
When $\sum_k \alpha_k^2 = \infty$, on the contrary, the infinite product
tends to $0$ almost everywhere, and the $\mu_N$ converge weakly to a singular
measure whose support is still $\R$. [The existence of a weak limit is easy to obtain, 
because it is easy to check that $\int_0^1 d\mu_N = 2\pi$.]
See for example Section 7 of Chapter V of 
\cite{Zygmund} 

It does not seem so easy to evaluate the numbers $\alpha_1(x,r)$ for
the weak limit in question, but a first approximation suggests
that $\alpha_1(x,r)$ should behave like combination of the $\alpha_k$, for
$r$ comparable to $3^{-k}$ , so we expect that Riesz products will give
absolutely continuous measures precisely when 
$J(x) = \int_{0}^1 \alpha_1(x,r)^2 \frac{dr}{r}< \infty$ $\mu$ almost everywhere. 

The following examples are easier to compute. They can be viewed in the context of 
\cite{BR} and \cite{NRTV}.
Denote by ${\cal I}$ the collection of dyadic interval in $\R$, and for
$I \in {\cal I}$, denote by $h_I$ the (badly normalized) Haar function defined by
\[
h_I = (\1_{I_+} - \1_{I_-}),
\]
where $I_+$ and $I_-$ are the two halves of $I$, with $I_+$ on
the right. Also denote by ${\cal I}_k$ the set of intervals $I \in {\cal I}$
such that $|I| = 2^{-k}$; we restrict to $k \geq 0$.

\begin{example} \label{t4.6}
Let the coefficients $a_I$, $i \in {\cal I}$, be smaller than $1$.
Then let $d\mu_N = G_N dx$, where
\[
G_N = \prod_{k=0}^N \prod_{I \in {\cal I}_k} (1+ a_I h_I).
\]
The measures $\mu_N$ converge weakly to a measure $\mu$
\end{example}

Using the orthogonality
of the Haar functions we expand and get that $\int_I G_N = 1$
when $I \in {\cal I}_0$. If $\sum_I a_I^2 |I| < \infty$, for instance,
$\sum_I a_I h_I \in L^2$, and $G(x) = \lim_{N \to \infty} G_N(x)$
exists almost everywhere, and lies in $(0,\infty)$, because 
the series  $\sum_I a_I h_I(x)$ converges almost everywhere.
If the $a_I$ satisfy suitably normalized Carleson conditions, then
$\sum_I a_I h_I \in \BMO$, and if the Carleson norm is small enough,
$||\sum_I a_I h_I||_{BMO}$ is small too, hence $\sum_I a_I h_I$ is exponentially
integrable and $\mu$ is absolutely continuous on $\R$ and given by an
$A_\infty$ weight. See for instance \cite{GarciaCuerva} or \cite{Journe}.

Let us restrict our attention to the case when $a_I$ depends only on $|I|$, 
i.e., $a_I = a(k)$ for $I\in {I}_k$.
Then $G_N(x)$ converges almost everywhere to a nonzero limit if and only if 
$\sum_k a(k)^2 < \infty$, and then $\mu$ and Lebesgue measure are mutually
absolutely continuous, and $\mu\in A_\infty(dx)$ (one can reduce to
the case when $\sum_k a(k)^2$ is small by cutting te first terms of the product).
If $\sum_k a(k)^2 =\infty$ one can show that $\mu$ is singular with respect to the 
Lebesgue measure.

\section{Proof of the small constant theorem - Theorem \ref{t1.7}}
\label{proof2}

We prove Theorem \ref{t1.7} before Theorem \ref{t1.6}, 
because the proof is more direct as it does not involve stopping time arguments or a corona construction. 
We proceed as in Section~\ref{proof1}, except that we only use 
one value of $d \geq 1$, let the origin $x$ vary, and restrict our attention to 
a large piece of $\Sigma$ where $J(x)$ is small, where
\begin{equation} \label{5.1}
J(x) = \int_{0}^1 \alpha_d(x,r) {dr \over r}.
\end{equation}

We work with a fixed integer $d\in [1,n]$ and
some times we drop it from our notation. For the moment, we just consider any
$x\in \Sigma$ such that $J(x) < \infty$, and we review some of the results
of the previous sections.

As before, we discretize the numbers $\alpha(x,r)$ and $J(x)$.
For $x\in \Sigma$ and $k \geq 0$, we define 
\begin{equation} \label{5.2}
\alpha_k(x) = \inf \big\{ \alpha_d(x,r) \, ; \, r \in [2^{-k-1},2^{-k}] \big\},
\end{equation}
and choose $r_k = r_k(x) \in [2^{-k-1},2^{-k}]$ such that
\begin{equation} \label{5.3}
\alpha_d(x,r_k(x)) \leq 2 \alpha_k(x).
\end{equation}
Notice that for $k_0 \geq 0$,
\begin{eqnarray} \label{5.4}
\sum_{k \geq k_0} \alpha_d(x,r_k(x))
&\leq& 2 \sum_{k \geq k_0} \alpha_k(x)
\leq 2 \sum_{k \geq k_0} \fint_{[2^{-k-1},2^{-k}]} \alpha_d(x,r)\,dr
\nonumber\\
&\leq& 2 \int_{0}^{2^{-k_0}} \alpha_d(x,r) {dr \over r} \leq 2 J(x).
\end{eqnarray}
Associated to the pair $(x,r_k)$, we have $d$-planes
\begin{equation} \label{5.5}
V_k(x) = V_d(x,r_k(x)),\quad  W_k(x) = x + V_k(x),
\end{equation}
(defined as in \eqref{2.17}), and also $W_k^\ast(x)$,
which is the vector space of dimension $d$ parallel to
$V_k(x)$ and $W_k(x)$.

For $x\in\Sigma$ with 
$J(x)$ is finite \eqref{5.4} ensures that 
$\alpha_d(x,r_k)$ tends to $0$. By the proof of 
Lemma \ref{t3.1} the dimension $d(k)$ defined near \eqref{3.3}
is equal to $d$ for $k$ large; hence $x \in \Sigma_0(d)$, and the 
conclusions of Theorem \ref{t1.5} hold for such $x$.

In particular, there is a density $\theta_d(x)$, defined by \eqref{1.15},
and such that by \eqref{3.17} and \eqref{3.14}
\begin{equation} \label{5.6}
\theta_d(x) = \lim_{r \to 0} \theta_d(x,r)
= \lim_{k \to \infty} \theta_d^\ast(x,r_k)
= \lim_{k \to \infty} r_k^{-d} c_{d}(x,r_k) \mu(B(x,r_k)).
\end{equation}
 In addition, iterations of \eqref{3.15}
show that
\begin{equation} \label{5.7}
\Big| \log\Big({\theta_d^\ast(x,r_k) \over \theta_d(x)}\Big)\Big|
\leq \sum_{l \geq k} \Big| \log\Big({\theta_d^\ast(x,r_l) 
\over \theta_d^\ast(x,r_{l+1})}\Big)\Big|
\leq C \sum_{l \geq k} \alpha_d(x,r_l).
\end{equation}

Also recall from \eqref{3.23} that $\Sigma$ has a tangent plane 
$W(x) = x + W^\ast(x)$ at $x$, where $W^\ast(x)$ is the limit of the $W_k^\ast(x)$;
we even know by \eqref{3.22} that
\begin{equation} \label{5.8}
\delta(W_k^\ast(x), W^\ast(x)) \leq C \sum_{l \geq k} \alpha_d(x,r_k);
\end{equation}
see the definition \eqref{3.20}.

\ms
We are now ready to start the proof of Theorem \ref{t1.7}.
By translation and dilation invariance, we may assume that the
ball $B(x,r)$ in the statement is $B = B(0,1/2)$. Then the main assumption
(namely, \eqref{1.19}) is that
\begin{equation} \label{5.9}
\int_{B(0,1)} \int_{0}^1 \alpha_d(x,r) {dr d\mu(x)\over r} 
= \int_{ B(0,1)} J(x) d\mu(x) \leq C_1 \mu(B),
\end{equation}
where exceptionally $C_1$ is a very small constant that we can choose at the
end of the argument, in particular in terms of the small constant $\gamma$
for which we want \eqref{1.20} and \eqref{1.23} to hold.

Note that the Borel set $A \subset B$ defined by
\begin{equation} \label{5.10}
A = \big\{ x\in \Sigma \cap B \, ; \, J(x) \leq \gamma^{-1} C_1 \big\}
\end{equation}
satisfies \eqref{1.20}. In fact by \eqref{5.9} we have
\begin{equation} \label{5.11}
\mu(B \sm A) \leq {\gamma \over C} \int_{x\in B} J(x) d\mu(x)
\leq \gamma \mu(B).
\end{equation}
 It remains to see that $A$ is
contained in $\gamma$-Lipschitz graph, and that \eqref{1.23} holds.

We may as well assume that $\gamma < \CD^{-4}$, and then
\eqref{5.11} also implies that there is 
$x_0 \in A \cap B(0,2^{-5})$. Then we choose an initial radius
$r_0 \in ({9 \over 10},{95\over 100})$ such that
\begin{equation} \label{5.12}
\alpha_d(x_0,r_0) \leq 10 \int_{90 \over 100}^{95\over 100} \alpha_d(x_0,r) dr
\leq 20 J(x_0) \leq 20 \gamma^{-1} C_1
\end{equation} 
Notice that for $x\in B$
$|x-x_0| \leq 1/2 + |x_0| \leq \frac{2}{3} r_0$ because $r_0 \geq 9/10$
and $x_0 \in B(0,2^{-5})$, thus $B\subset B(x_0,\frac{2}{3}r_0)$.

Set $P = W_0^\ast(x_0)$ (the vector $d$-plane parallel to $W_0(x_0)$;
see near \eqref{5.5}), and denote by $\pi$ the orthogonal projection
onto $P$. Also set $\pi^\perp = I - \pi$, the orthogonal projection
onto the orthogonal complement $P^\perp$.
We want to show that $A$ is contained in a $\gamma$-Lipschitz
graph over $P$, and for this it is enough to show that
\begin{equation} \label{5.13}
|\pi^\perp(x) - \pi^\perp(y)| \leq {\gamma \over 2} |x-y|
\ \text{ for } x,y \in A.
\end{equation}

Let $x, y \in A$ be given, and let us first assume that
$|x-y| \geq 2^{-10}$.
In this case, Lemma \ref{t2.6}, applied to the ball $B(x_0,r_0)$, says that
$x$ and $y$ both lie within $C \alpha_d(x_0,r_0)^\eta$ from
$W_d(x_0,r_0) = W_0(x_0)$. Recall that $|x-x_0|\le \frac{2}{3} r_0$ and  $|y-x_0|\le \frac{2}{3} r_0$,
This yields
\[
|\pi^\perp(x) - \pi^\perp(y)| \leq C \alpha_d(x_0,r_0)^\eta
\leq C (\gamma^{-1} C_1)^{\eta},
\]
and \eqref{5.13} follows if $C_1$ is small enough, depending on $\gamma$.
 Now suppose that $|x-y| \leq 2^{-10}$, and let $k$ be 
such that $2^{-k-3} \leq |x-y| < 2^{-k-2}$; then $k \geq 7$.
By Lemma \ref{t2.6}, applied to the ball $B(x,r_k(x))$, 
\begin{eqnarray} \label{5.14}
&&\dist(x,W_d(x,r_k(x))) + \dist(y,W_d(x,r_k(x)))
\leq C r_k(x) \alpha_d(x,r_k(x))^\eta 
\nonumber\\
&&\hskip7.5cm\leq C 2^{-k} (\gamma^{-1} C_1)^{\eta}.
\end{eqnarray}
With our new notation, $W_d(x,r_k(x))$ is the same as $W_k(x)$, and
the vector space parallel to $W_k(x)$ is $W_k^\ast(x)$.
Denote by $\pi_k^\perp$ the orthogonal projection
on $W_k^\ast(x)^\perp$; then \eqref{5.14} says that
\begin{equation} \label{5.15}
|\pi_k^\perp(x) - \pi_k^\perp(y)| \leq C 2^{-k} (\gamma^{-1} C_1)^{\eta}.
\end{equation}
By \eqref{5.8} (or more directly \eqref{3.22}),
\begin{equation} \label{5.16}
\delta(W_k^\ast(x), W^\ast_7(x)) \leq C \sum_{l \geq 7} \alpha_d(x,r_k)
\leq C J(x) \leq C \gamma^{-2} C_1.
\end{equation}
Finally, applying Lemma \ref{t2.8} to the ball
$B(x_0,r_0)$ and with $(x,r_7(x))$ in the role of $(y,t)$, we have
that
\begin{equation} \label{5.17}
\delta(W^\ast_7(x), W_0^\ast) \leq C \alpha_d(x_0,r_0) + C \alpha_d(x,7)
\leq C \gamma^{-2} C_1
\end{equation}
(the assumption \eqref{2.34} holds because by our choices of $x_0$, $r_0$, $r_7$ and because for $x\in A \subset B = B(0,1/2)$). 
Then \eqref{5.15} combined with \eqref{5.16}, \eqref{5.17}, and the definition \eqref{3.20} implies that
\begin{eqnarray} \label{5.18}
|\pi^\perp(x) - \pi^\perp(y)| &\leq&
|\pi_k^\perp(x) - \pi_k^\perp(y)| + |x-y| \, \|\pi^\perp - \pi_k^\perp\|
\nonumber\\
&\leq&
|\pi_k^\perp(x) - \pi_k^\perp(y)| + |x-y| \delta(W_k^\ast(x), W_0^\ast) 
\nonumber\\
 &\leq& C 2^{-k} (\gamma^{-1} C_1)^{\eta} + C \gamma^{-2} C_1 |x-y|.
\end{eqnarray}
The desired
estimate \eqref{5.13}  follows from this, because $2^{-k-3} \leq |x-y|$
provided $C_1$ is small enough depending on $C_\delta$, $d$ and $\gamma$ So $A$ is contained in a $\gamma$-Lipschitz
graph.

To prove \eqref{1.23} we estimate densities.
We start from \eqref{5.6} and \eqref{5.7}, which we apply with $k=5$.
This yields by \eqref{5.4} and \eqref{5.10} that
\begin{equation} \label{5.19}
\Big| \log\Big({\theta_d(x) \over \theta_d^\ast(x,r_5(x))}\Big)\Big|
\leq C \sum_{l \geq 5} \alpha_d(x,r_l(x)) \leq C J(x) \leq C \gamma^{-1} C_1.
\end{equation}
 Then we apply 
Lemma \ref{t2.12} to the ball $B(x_0,r_0)$, the point $y=x$,
and the radius $ar_0 = r_5(x)$. Lemma \ref{t2.6} says that 
$\dist(x,W_d(x_0,r_0)) \leq C r_0 \alpha_d(x_0,r_0)^\eta$
(for some $C \geq 0$ and $\eta \in (0,1/2)$ that depend on $\CD$),
so we can take $\wt\delta = C \alpha_d(x_0,r_0)^\eta$ to apply 
Lemma \ref{t2.12}. By \eqref{2.68},
\[
\Big|{\mu(B(x, r_5(x))) \over \mu(B(x_0,r_0))} - c_d(x_0,r_0) 
\big({r_5(x) \over r_0}\big)^d \Big| 
\leq C \wt\delta + 2 \wt\delta^{-1} \alpha_d(x_0,r_0)
\leq C \alpha_d(x_0,r_0)^\eta.
\]
Since $c_d(x_0,r_0) \geq 1$ and $r_0 \leq 2^6 r_5(x)$,
we may divide, take a logarithm, observe that
$\log a \sim a-1$ for $a$ close to $1$, and get that
\[
\Big|\log\Big({\mu(B(x, r_5(x))) r_0^d \over c_d(x_0,r_0) \mu(B(x_0,r_0))r_5(x)^d } 
\Big)\Big| 
\leq C \alpha_d(x_0,r_0)^\eta.
\]
The same computation, performed with $B = B(0,1/2)$, yields
\[
\Big|{\mu(B) \over \mu(B(x_0,r_0))} - c_d(x_0,r_0) 
\big({1 \over 2 r_0}\big)^d \Big| 
\leq C \alpha_d(x_0,r_0)^\eta
\]
and then
\[
\Big|\log\Big({\mu(B) 2^d r_0^d \over c_d(x_0,r_0) \mu(B(x_0,r_0))} 
\Big)\Big| 
\leq C \alpha_d(x_0,r_0)^\eta.
\]
[To be fair, Lemma \ref{t2.12} was only stated when $a < 1/3$; 
here we take $a$ a little larger than $1/2$, but a center very close to $x_0$, 
and the same proof applies.] We combine the previous estimates and get that
\begin{equation} \label{5.20}
\Big|\log\Big({\mu(B(x, r_5(x))) \over \mu(B) 2^d r_5(x)^d } 
\Big)\Big|  \leq C \alpha_d(x_0,r_0)^\eta \leq C (\gamma^{-1} C_1)^{\eta}.
\end{equation}
Recall from \eqref{5.6} that $\theta_d^\ast(x,r_5(x)) 
= r_5(x)^{-d} c_d(x,r_5(x)) \mu(B(x,r_5(x)))$; then by \eqref{5.19} 
and \eqref{5.20}
\begin{equation} \label{5.21}
\Big|\log\Big({\theta_d(x) \over c_d(x,r_5(x)) \mu(B) 2^d} \Big)\Big|  
\leq C (\gamma^{-1} C_1)^{\eta}.
\end{equation}
Finally notice that 
$\dist(x,W_d(x,r_5(x)) \leq C r_5(x) \alpha_d(x,r_5(x))^\eta
\leq C r_5(x) (\gamma^{-1} C_1)^{\eta}$ 
by Lemma~\ref{t2.6} and \eqref{5.19} again, so
\eqref{2.14} yields $|\log c_d(x,r_5(x))| \leq C (\gamma^{-1} C_1)^{\eta}$.
Thus 
\begin{equation} \label{5.22}
\Big|\log\Big({\theta_d(x) \over 2^d\mu(B)} \Big)\Big| \leq C (\gamma^{-1} C_1)^{\eta}.
\end{equation}
This holds for every $x\in A$. With the notation of Theorem \ref{t1.5},
this implies that $A$ is contained in one, or at most two, of the sets $\Sigma_0(d,k)$,
and in particular $\H^d$ and $\mu$ are equivalent on $A$. In addition,
$\mu = \theta_d \H^d\res A$, and \eqref{1.23} follows easily (recall that
$r = 1/2$ in \eqref{1.23}).

This completes our proof of Theorem \ref{t1.7}.

\begin{remark} \label{t5.1}
At the scale of $B$, $\Sigma$ is flat, and does not have big holes.
This comes from Lemma \ref{t2.7}, for instance. In fact, for $x\in A$,
$\Sigma$ never has hole  of size larger than
$C\alpha_d(x,r_k(x))^{1\over d+1}$ in $B(x,r_k(x))$ 
In the case when $A=\Sigma \cap B$, 
then $A$ is Reifenberg-flat, 
and equal to a Lipschitz graph inside a slightly 
smaller ball.

The fact that $\Sigma$ does not have big holes in $B$ would also
help us do the necessary patching argument if we replaced $2r$ in
\eqref{1.19} with $\lambda r$, for some other constant $\lambda > 1$.

\end{remark}

\section{Proof of the main theorem - Theorem \ref{t1.6}}
\label{proof3}

\subsection{Dyadic cubes}
\label{cubes}
The proof of Theorem \ref{t1.6} relies on the same ingredients as the proof
of Theorem \ref{t1.7}, but since the function $J$ is only bounded
(and not necessarily small) on the set of interest, we cannot expect to control the angles
and densities as well as before, and there will be a few scales, depending on $x$,
where we lose information. This requires a stopping time argument, and
we use the standard machinery that is a 
``corona construction''. Since the statement involves very big pieces of
bi-Lipschitz images, we will need
to cut out thin slices away from $\Sigma$ to separate different pieces.
Dyadic pseudo-cubes with small boundaries will be good for that also.

The main focus of the proof is to control the geometry. The density estimates
required to prove \eqref{1.21} are straightforward.

Let $\mu$ be a general doubling measure with no
no atoms. We first review a construction from \cite{Christ-T(b)}, which is
precisely adapted to doubling measures; we change the notation a little to suit our needs. 

In what follows, $C$ denotes a constant that depends only on the doubling
constant $\CD$ from \eqref{1.1}. 

There exists a collection of set sets $\Delta_j$, $j\in \bZ$,
with the following properties. Each set $\Delta_j$ is a set of
Borel subsets of $\Sigma$, and 
\begin{equation} \label{6.1}
\Sigma  \text{ is the disjoint union of the sets } Q, Q\in \Delta_j. 
\end{equation}
We call these sets ``cubes'', by analogy with the usual dyadic
cubes, but they may not be as smooth, and they are subsets of $\Sigma$.
Each cube $Q \in \Delta_j$ has a ``center'' $c_Q \in Q$, and
\begin{equation} \label{6.2}
\Sigma \cap B(c_Q, C^{-1} 2^{-j}) \subset Q \subset \Sigma \cap B(c_Q,2^j).
\end{equation}
Moreover the cubes are nested, i.e.,
\begin{equation} \label{6.3}
\text{For $Q\in \Delta_j$ and $k \leq j$, there is a unique
cube $R \in \Delta_k$ that meets $R$.}
\end{equation}
Then $Q \subset R$, because $\Sigma$ is the disjoint union of the
cubes $R \in \Delta_k$. When $j=k+1$, we often call $R$ a parent of $Q$,
and $Q$ a child of $R$. In general, $R$ is an ancestor of $Q$ and $Q$
is a descendant of $R$. Note that a cube $R$ may have exactly one child, in this case 
we see $R$ as belonging to two different generations. One could avoid this by skipping generations
but is is not worth the trouble. For additional details see \cite{David88}.

 For our needs, it would be enough to define the sets
$\Delta_j$ only for $j \geq j_0$, where $j_0$ depends on the size of the
ball $B(0,r)$ that we consider.
We set $\Delta = \cup_j \Delta_j$ (our set of cubes), and to each
cube $Q$ we associate the approximate diameter $d(Q)$, where
\begin{equation} \label{6.4}
d(Q) = 2^{-j} \ \text{ when } Q \in \Delta_j. 
\end{equation}
We also define the enlarged set
\begin{equation} \label{6.5}
\lambda Q = \big\{ x\in \Sigma \, ; \, \dist(x,Q) \leq (\lambda-1)d(Q) \big\}
\ \text{ for } \lambda > 1
\end{equation}
and the reduced set
\begin{equation} \label{6.6}
\lambda Q = \big\{ x\in Q \, ; \, \dist(x,\Sigma \sm Q) \geq (1-\lambda)d(Q) \big\}
\ \text{ for } \lambda < 1.
\end{equation}
The important feature of our cubes is that they have ``small boundaries'', in the following sense. 
There is a small constant $\kappa > 0$, which also depends only on $\CD$, 
such that
\begin{equation} \label{6.7}
\mu(Q \sm \lambda Q) \leq C (1-\lambda)^\kappa \mu(Q)
\end{equation}
when $Q\in \Delta$ and $0 < \lambda < 1$ (but this is only meaningful
when $\lambda$ is close to $1$). Since by \eqref{6.1} and \eqref{6.2},
$2Q$ can only intersect $C$ cubes of the same generation as $Q$,
and by \eqref{1.1} all these cubes have roughly the same size,
\eqref{6.7} also implies the following control on the exterior part:
\begin{equation} \label{6.8}
\mu(\lambda Q \sm Q) \leq C (\lambda-1)^\kappa \mu(Q)
\end{equation}
for $Q\in \Delta$ and $1 < \lambda < 2$.

See \cite{Christ-T(b)} for details on the existence of $\Delta$.
The main issue is to mix the small boundary condition and the 
hierarchical structure. The fact that the constants above do not depend
on $n$, but just on $\CD$ comes from the proof, and is not surprising.

\ms
Our proof will use the notion of \textit{semi-adjacent cubes}.

\begin{definition} \label{t6.1}
Let $\lambda > 1$ be a large constant to be chosen later,
depending on $\CD$ through the constants in \eqref{6.2}-\eqref{6.7}.
We say that two cubes $Q, R \in \Delta$ are semi-adjacent,
and we write $Q \sim R$, when
\begin{equation} \label{6.9}
Q \subset \lambda R  \ \text{ and } \  R \subset \lambda Q. 
\end{equation}
\end{definition}

\ms
Notice that this implies that $d(Q)$ and $d(R)$ are comparable,
and also (by \eqref{1.1}) that $\mu(Q)$ and $\mu(R)$ are comparable, 
with constants that depend on $\lambda$.  We do not always mention that 
dependence, because $\lambda$ will be chosen first (so that some geometrical 
constraints are satisfied). We choose $\lambda$ large enough so that
$Q\sim R$ when $Q$ is a child of $R$ and when $Q$ and $R$ are of the 
same generation and $\dist(Q,R) \leq d(Q)$.
We also need a large number $\lambda^\ast > 0$ (larger than $\lambda$,
to be chosen after $\lambda$),
which we use to associate an $\alpha$-number to each cube $Q\in \Delta$. 
That is, we set
\begin{equation} \label{6.10}
\alpha(Q) = \inf\big\{ \alpha_d(x,r) \, ; \, x\in Q \text{ and } 
\lambda^\ast d(Q) \leq r \leq 2 \lambda^\ast d(Q) \big\},
\end{equation}
then select a pair $(x_Q,r_Q)$, with $x_Q \in Q$, $\lambda^\ast d(Q) \leq r_Q 
\leq 2\lambda^\ast d(Q)$, and
\begin{equation} \label{6.11}
\alpha_d(x_Q,r_Q) \leq 2 \alpha(Q),
\end{equation}
then set
\begin{equation} \label{6.12}
W(Q) = W_d(x_Q,r_Q),
\end{equation}
and denote by $W^\ast(Q)$ the $d$- vector space parallel to $W(Q)$.

\begin{lemma} \label{t6.2}
Using the notation above, if $\lambda^\ast \geq 10\lambda$, then
\begin{equation} \label{6.13}
\dist(y,W(Q)) \leq  C \lambda^\ast \alpha(Q)^\eta d(Q)
\ \text{ for $Q \in \Delta$ and } y\in \lambda Q.
\end{equation}
Moreover when $Q$ and $R \in \Delta$ are semi-adjacent cubes,
\begin{equation} \label{6.14}
\dist(z,W(Q)) \leq C \lambda^\ast (\alpha(Q)+\alpha(R)) (d(Q)+d(R))
\end{equation}
for $z\in W(R) \cap B(x_Q, \lambda^\ast(d(Q)+d(R)))$, and
\begin{equation} \label{6.15}
\dist(z,W(R)) \leq C \lambda^\ast (\alpha(Q)+\alpha(R)) (d(Q)+d(R))
\end{equation} 
for $z\in W(Q) \cap B(x_Q, \lambda^\ast(d(Q)+d(R)))$.
\end{lemma}
\ms\begin{proof}  
By Lemma \ref{t2.6} for $y\in \Sigma \cap B(x_Q, \frac23 r_Q)$
\begin{equation} \label{6.16}
\dist(y,W(Q)) \leq C \alpha_d(x_Q,r_Q)^\eta r_Q
\leq C \lambda^\ast \alpha(Q)^\eta d(Q),
\end{equation}
and $\lambda Q \subset \Sigma \cap B(x_Q, \frac23 r_Q)$ because
$\lambda^\ast \geq 10 \lambda$, and \eqref{6.13} follows. 
Similarly, suppose for instance that $r_Q \geq r_R$; notice that
$\dist(x_R,Q) \leq \lambda d(Q)$ because $R \subset\lambda Q$,
then $|x_Q-x_R| \leq (\lambda+2) d(Q)$, then
Lemma \ref{t2.8} applies because $|x_Q-x_R| + {r_R \over 2} < r_Q$,
and we get that\begin{equation} \label{6.17}
\dist(z,W(Q)) \leq C (\alpha(Q)+\alpha(R)) r_Q
\end{equation}
for $z\in W(R) \cap B(x_Q, 2r_Q)$, and 
\begin{equation} \label{6.18}
\dist(z,W(R)) \leq C (\alpha(Q)+\alpha(R)) r_Q
\end{equation}
for $z\in W(Q) \cap B(x_Q, 2r_Q)$. By elementary geometry,
these estimates also hold with $B(x_Q, 2r_Q)$ replaced by the possibly slightly 
larger ball $B(x_Q, \lambda^\ast(d(Q)+d(R)))$, or 
$B(x_Q, \lambda^\ast(d(Q)+d(R)))$ (if we need to exchange $Q$ and $R$);
so \eqref{6.14} and \eqref{6.15} follow.
\end{proof}

\subsection{Stopping time regions and Lipschitz graphs}
\label{stop}

For this part of the proof, we rely on definitions and a construction from
\cite{DS} or \cite{of-and-on}.

\begin{definition}\label{t6.3}
A {\it stopping-time region} is a set $S\subseteq \Delta$ of cubes that satisfies
the following properties. First, $S$ contains a maximal cube $Q(S)$, i.e., a
cube $Q(S) \in S$ such that
\begin{equation} \label{6.19}
Q \subset Q(S) \ \text{ for each } Q \in S.
\end{equation}
But also, $S$ is {\it coherent}. That is, if $Q\in S$ is properly contained in $Q(S)$, 
then $R\in S$ for all $Q\subseteq R\subseteq Q(S)$, and all the siblings of 
$Q$ (i.e., the children of the parent of $Q$) are also in $S$.

We say that $S$ is a {\it stopping-time region with constant $\alpha$}
when in addition 
\begin{equation} \label{6.20}
\sum_{Q\subseteq R\subseteq Q(S)}\alpha(R)<\alpha
\ \text{ for } Q \in S.
\end{equation}
\end{definition}

We define the set of minimal cubes of the stopping time region $S$ by
\begin{equation} \label{6.21}
M(S)=\{Q\in S: Q\mbox{ has no children in }S\}.
\end{equation}

\ms
Next we fix a stopping-time region $S$ with constant $\alpha$, where
$\alpha > 0$ is a small constant that will be chosen later, and use the $W(Q)$,
$Q \in S$ to build a Lipschitz graph over the $d$-plane
\begin{equation} \label{6.22}
W_S = W(Q(S)).
\end{equation}

For $Q \in S$, we denote by $\pi_Q$ the orthogonal projection onto $W(Q)$
and by $\pi_Q^\perp$ the orthogonal projection onto the orthogonal complement of 
$W(Q)^\ast$ (the vector space parallel to $W(Q)$).
In the special case of $Q = Q(S)$, simply write $\pi$ (instead of
$\pi_{Q(S)}$) and $\pi^\perp$. Set
\begin{equation} \label{6.23}
d_{S}(x)=\inf_{Q\in S}(\dist(x,Q)+d(Q)),
\end{equation}
for $x\in \bR^{n}$,
\begin{equation} \label{6.24}
D_{S}(p)=\inf\big\{d_S(x) \, ; \, x\in \pi^{-1}(p)\big\}
\end{equation}
for $p\in W_S$, and
\begin{equation} \label{6.25}
Z(S)=\{x\in\Sigma: d_{S}(x)=0\}.
\end{equation}

For the next lemma, and the ensuing construction, we assume that
$\lambda$ and $\lambda^\ast$ are large enough, and $C$ denotes a constant
that may depend on $d$, $\CD$, $\lambda$, and $\lambda^\ast$.

\begin{lemma} \label{t6.4}
We have that
\begin{equation} \label{6.26}
|\pi^\perp(x)-\pi^\perp(y)| \leq C \alpha^{\eta} |x-y|
\end{equation}
for $x,y \in 100Q(S)$ such that 
\begin{equation} \label{6.27}
|x-y| \geq  \frac{1}{10}\min\{d_{S}(x),d_{S}(y)\}.
\end{equation}
\end{lemma}

\begin{proof}  
Here $\eta$ is still as in Lemma \ref{t2.6}.
Let $x, y \in 100Q(S)$ satisfy \eqref{6.27}, and assume without loss of generality 
that $d_S(x) \leq d_S(y)$.
Thus $d_S(x) \leq 10 |x-y|$ by \eqref{6.27}, and by \eqref{6.23}
we can find $Q\in S$ such that $\dist(x,Q)+d(Q) \le 11|x-y|$. 
Let $R$ be the largest cube of $S$ that contains $Q$ and such that 
$d(R) \leq 11 |x-y|$.
First assume that $R = Q(S)$. By \eqref{6.13},the assumption that
$x,y \in 100Q(S)$, recalling that $\pi$ is the projection on $W_S = W(Q(S))$ and using 
\eqref{6.20} and the fact that in this case $R = Q(S)$ we have
\begin{eqnarray}  \label{6.28}
|\pi^\perp(x)-\pi^\perp(y)| &\leq& \dist(x,W_S) + \dist(y,W_S)
\leq C \lambda^\ast \alpha(Q(S))^\eta d(Q(S))
\nonumber \\
&\leq& C \lambda^\ast \alpha^\eta d(Q(S))
\leq 11 C \lambda^\ast \alpha^\eta |x-y|.
\end{eqnarray}
So \eqref{6.26} holds in this first case.

If $R \neq Q(S)$, the parent $R^\ast$ of $R$ does not satisfy
the size constraint (because the coherence condition says that $R^\ast \in S$).
That is, $11 |x-y|  \leq d(R^\ast) = 2d(R)$. 
Since $\dist(x,R) \leq \dist(x,Q) \leq 11|x-y|$ by definition of $Q$,
we see that $x$ and $y$ lie in $\lambda R$, and we can apply \eqref{6.13}
to the cube $R$ and the points $x$ and $y$. This yields
\begin{eqnarray}  \label{6.29}
|\pi^\perp_R(x)-\pi^\perp_R(y)| &\leq& \dist(x,W(R)) + \dist(y,W(R))
\leq C \lambda^\ast \alpha(R)^\eta d(R)
\nonumber \\
&\leq& C \lambda^\ast \alpha^\eta d(R)
\leq C \lambda^\ast \alpha^\eta |x-y|.
\end{eqnarray}
It will be good to know that
\begin{equation} \label{6.30}
\delta(W(R)^\ast,W(Q(S))^\ast )\leq C \alpha
\ \text{ for every cube } R \in S 
\end{equation}
(and where $C$ may depend on $\lambda^\ast$ too).
Observe that \eqref{6.17} and \eqref{6.18} imply that 
\begin{equation} \label{6.31}
\delta(W(Q_1)^\ast,W(Q_2)^\ast) 
= \|\pi^\perp_{Q_1}-\pi^\perp_{Q_2}\|
\leq C \lambda^\ast \alpha(Q_1) + \alpha(Q_2)
\end{equation}
whenever $Q_1$ and $Q_2$ are semiadjacent cubes 
(also see the definition \eqref{3.20} of $\delta$). 
In particular, this holds when $Q_2$ is the parent of $Q_1$. 
We apply this to $R$ and all its ancestors up to $Q(S)$
(excluded), use the triangle inequality to sum the estimates, and get using  \eqref{6.20} that
\[
\delta(W(R)^\ast,W(Q(S))^\ast) 
\leq C \lambda^\ast \sum_{R \subset T \subset Q(S)} \alpha(T)
\leq C \alpha,
\]
and so \eqref{6.30} holds.
With our specific choice $R$, $x$ and $y$ as above,
\[
|\pi^\perp_R(x-y)-\pi^\perp(x-y)| \leq \|\pi^\perp_R-\pi^\perp\| \, |x-y|
\leq C \alpha |x-y|,
\]
and \eqref{6.26} follows from \eqref{6.29}. This proves Lemma \ref{t6.4}.
\end{proof}

The next step is to use partitions of unity to construct a Lipschitz
mapping $A : W_S \to W_S^\perp$. To simplify the notation
let $W = W_S$, which we identify with $\R^d$, denote by $W^\perp$
the orthogonal vector space, and set 
\begin{equation} \label{6.32} 
U = W \cap B(\pi(x_{Q(S)}), 10d(Q(S))).
\end{equation}
 Also set $Z = Z(S)$
and $D = D_S$. We keep $d_S$ as it is, to avoid confusion with
other functions $d$. 

We may already define $A$ on $\pi(Z)$. Indeed, notice that by \eqref{6.26} and because $d_S(x)=d_S(y)=0$ for $x,\, y\in Z$ then 
\begin{equation} \label{6.33}
|\pi^\perp(x)-\pi^\perp(y)| \leq C \alpha^{\eta} |x-y|
\ \text{ for } x, y \in Z.
\end{equation}
Thus we
define $A(p)$, for $p\in \pi(Z)$, by the fact that
\begin{equation} \label{6.34}
A(p) \in W^\perp \ \text{ and } p + A(p) \in Z.
\end{equation}

To define $A$ on $U \sm \pi(Z)$, we use partitions of unity
and a (standard!) dyadic grid on $W$. We typically call $R$ a dyadic cube on $W$, 
and $Q$ one of our pseudo-cubes in $\Sigma$. 

Notice that $d_S$ is $1$-Lipschitz, and $D$ is $1$-Lipschitz too, 
by \eqref{6.23}. Thus $Z$ is compact, $\pi(Z)$ is closed, 
$U \sm \pi(Z)$ is open, and $D(p) > 0$ on that set.

Denote by $\cD$ the set of dyadic cubes $R$ of $W$ that meet $U \sm \pi(Z)$, 
such that 
\begin{equation} \label{6.35}
\diam(R) \leq {1 \over 20} D(p) \ \text{ for } p \in R, 
\end{equation}
which are maximal, in the sense that their parent does not satisfy \eqref{6.35}.
Notice that the small dyadic cubes that contain a given point 
$p\in U \sm \pi(Z)$ satisfy \eqref{6.35}, so the cubes 
of $\cD$ cover $U \sm \pi(Z)$. By maximality they are also almost disjoint.
The following properties are easy consequences of the definitions, and the fact
that $d_S$ and $D$ are Lipschitz, and their proof can be found in 
\cite[Section 8]{DS}:
\begin{equation} \label{6.36}
\diam(R) \leq D(p) \leq 60 \diam(R) \ \text{ when $p \in 10R$ and } R \in \cD,
\end{equation}
\begin{eqnarray} \label{6.37}
&36^{-1} \diam(R') \leq \diam(R) \leq 36 \diam(R') \nonumber\\
&\hbox{ for } R, R' \in \cD \hbox{ such that }
10R \,\cap \, 10R' \neq \emptyset,
\end{eqnarray}
and
\begin{equation} \label{6.38}
\sum_{R \in \cD} \1_{10R} \leq C 
\end{equation}
where $C$ depends on $d$. This is obtained by counting how many dyadic cubes of roughly the same size can meet
$10 Q$.

To each $R \in \cD$, we associate a pseudo-cube $Q_R \in S$, as follow. 
Pick $p\in R \cap U$, then use \eqref{6.24} to pick $x\in \R^n$ such that
$d_S(x) \leq 2 D(p)$, then find $Q' \in S$ such that
$\dist(x,Q') + d(Q') \leq 3 D(p)$, and replace $Q'$ with the largest
ancestor $Q_R$ of $Q'$ that lies in $S$ and for which $d(Q_R) \leq 3 D(p)$.
This way we get that
\begin{equation} \label{6.39}
d(Q_R) + \dist(Q_R, \pi^{-1}(R)) \leq d(Q_R) + \dist(x, Q_R) \leq 6 D(p)
\leq 360 \diam(R).
\end{equation}
Let us check that in addition
\begin{equation} \label{6.40}
d(Q_R) \geq {20 \over 11}\diam(R).
\end{equation}
If $Q_R \neq Q(S)$, then the parent $Q'_R$ of $Q_R$ is too large, 
and we get that by \eqref{6.35} that
\[
d(Q_R) = {1\over 2}d(Q'_R) \geq {3 \over 2} D(p) \geq 30 \diam(R).
\]
If $Q_R = Q(S)$, set $p_0 = \pi(x_{Q(S)})$ and notice that by \eqref{6.35}, \eqref{6.25} and \eqref{6.24}, since
 $D$ is Lipschitz, $x_{Q(S)} \in Q(S)$, and because $p \in U$ then we have
\[
\begin{aligned}
20 \diam(R) &\leq D(p) \leq D(p_0)+ |p-p_0| \leq d_S(x_{Q(S)}) + |p-p_0|
\\
&\leq d(Q(S)) + |p-p_0| \leq 11 d(Q(S)) = 11d(Q_R).
\end{aligned}
\]
So \eqref{6.40} holds in this second case too.

Now we define a partition of unity on $U \sm \pi(Z)$. We start with smooth functions
$\wt \varphi_R$, $R \in \cD$, such that 
$\1_{2R} \leq \wt \varphi_R \leq \1_{3R}$ and
$|\nabla\wt \varphi_R| \leq C \diam(R)^{-1}$. Then we set
\[
\varphi_R = {\wt \varphi_R \over \sum_{R' \in \cD} \wt \varphi_{R'}}
\]
on $U$, observe that for $p\in U$,
\begin{equation} \label{6.41}
|\nabla \varphi_R(p)| \leq C \diam(R)^{-1} 
\end{equation}
and
\begin{equation} \label{6.42}
\sum_{R \in \cD} \varphi_R(p) = \1_{U \sm \pi(Z)}(p).
\end{equation}
Denote by $A_R : W \to W^\perp$ the affine map whose graph is $W(Q_R)$, and set
\begin{equation} \label{6.43}
A(p) = \sum_{R \in \cD} \varphi_R(p) A_R(p) 
\ \text{ for } p\in U \sm \pi(Z)
\end{equation}
(we already defined $A$ on $\pi(Z)$). This is the function whose graph
will approximate $\Sigma$ at the scale of the smallest cubes of $S$.

\begin{lemma} \label{t6.5}
The function $A$ is $C\alpha^\eta$-Lipschitz on $U$, and, if 
$$
\Gamma_S = \big\{ p+A(p) \, ; \, p\in U \big\}
$$ 
denotes the graph of $A$ over $U$, 
\begin{equation} \label{6.44}
\dist(x, \Gamma_S) \leq  
|\pi^\perp(x) - A(\pi(x))| \leq C \alpha^\eta d_S(x)
\ \text{ for } x\in 8Q(S).
\end{equation}
As usual, $C$ and $\eta$ depend only on $d$, $\CD$, $\lambda$, 
and $\lambda^\ast$.
\end{lemma}

\begin{proof}  
First we check that for $R, \ H \in \cD$ such that $10R \cap 10H \neq \emptyset$,
\begin{equation} \label{6.45}
\dist(Q_R,Q_H) \leq 10^{5} \diam(R).
\end{equation}
Pick $x\in Q_R$ and $y\in Q_S$. Notice that by \eqref{6.39}
\[
d_S(x) \leq \dist(x,Q_R) + d(Q_R) = d(Q_R) \leq 360\diam(R).
\]
 If $|x-y| \leq 10 d_S(x)$, \eqref{6.45} holds. 
Otherwise, Lemma~\ref{t6.4} says that 
$|\pi^\perp(x)-\pi^\perp(y)| \leq C \alpha^{\eta} |x-y| \leq |x-y|/10$
(if $\alpha$ is small enough), hence $|x-y| \leq 2 |\pi(x)-\pi(y)|$. 
But by \eqref{6.39}
\[
\begin{aligned}
\dist(\pi(x), R) &= \dist(x,\pi^{-1}(R)) 
\leq \diam(Q_R) + \dist(Q_R,\pi^{-1}(R))
\\
&\leq 2d(Q_R) + \dist(Q_R,\pi^{-1}(R)) \leq 720 \diam(R),
\end{aligned}
\]
and similarly $\dist(\pi(y), H) \leq 720 \diam(H)$; thus
\[
\begin{aligned}
|x-y| \leq 2 |\pi(x)-\pi(y)| &\leq 2\dist(\pi(x), R) + 2\dist(\pi(y), H)
+ 2\diam(R \cup H) 
\\ &
\leq 720 \diam(R) + 720 \diam(H) + 20 (\diam(R) +\diam(H))
\\&
\leq 10^5 \diam(R)
\end{aligned}
\]
because $10R \cap 10H \neq \emptyset$ and by \eqref{6.37}. 
This proves \eqref{6.45}.

If $\lambda$ is large enough \eqref{6.45} says that $Q_R$ and
$Q_H$ are semi-adjacent. Thus by
\eqref{6.14} and \eqref{6.15}, $W(Q_R)$ and $W(Q_H)$
are quite close to each other in a large area near $Q_R$ and $Q_H$. In particular, for $p \in 10R$
\begin{equation} \label{6.46}
|A_R(p)-A_H(p)| \leq C (\alpha(Q_R)+\alpha(Q_H))(d(Q_R)+d(Q_H))
\leq C \alpha \diam(R). 
\end{equation}

Next we check that $A$ is Lipschitz. We start on a cube $2R$, $R \in \cD$, 
where we can just differentiate \eqref{6.43}, since $|DA_H| \leq C \alpha$ by \eqref{6.30}, and because we only sum over $H$ such that $3H$ meets $2R$ we have
\begin{equation} \label{6.47}
\begin{aligned}
|DA(p)| &= \big|\sum_{H \in \cD}  \varphi_H(p) DA_H(p) + D\varphi_H(p) A_H(p) \big|
\\ &
\leq C \alpha \sum_{H \in \cD} \varphi_H(p)
+ \big|\sum_{H \in \cD}  D\varphi_H(p) A_H(p) \big|
\\ &
\leq C \alpha + \big|\sum_{H \in \cD}  D\varphi_H(p) [A_H(p)-A_R(p)]\big|
\\ &
\leq C \alpha \diam(R) \sum_{H \in \cD} |D\varphi_H(p)| \leq C \alpha,
\end{aligned}
\end{equation}
where we have also used 
 \eqref{6.45}, \eqref{6.41}, and \eqref{6.37}. So
$A$ is locally $C \alpha$-Lipschitz in $U \sm \pi(Z)$.

Notice that $A$ is $C \alpha^\eta$-Lipschitz in $\pi(Z)$, by
\eqref{6.33} and the definition \eqref{6.34}. The next step is to show
that for $R \in \cD$,
\begin{equation} \label{6.48}
|A(p)-A(q)| \leq C \alpha^\eta |p-q| 
\ \text{ when $p\in R$ and $q\in \pi(Z)$.} 
\end{equation}
Let $y\in Z$ be such that $\pi(y) = q$, and pick any $x\in Q_R$.
Since $d_S(y) = 0$, Lemma~\ref{t6.4} says that 
\begin{equation} \label{6.49}
|\pi^\perp(x)-\pi^\perp(y)| \leq C \alpha^{\eta} |x-y|.
\end{equation}
Also, since $x\in Q_R$; \eqref{6.13} says that
\begin{equation} \label{6.50}
\dist(x,W(Q_R)) \leq C \alpha(Q_R)^\eta d(Q_R) \leq C \alpha^\eta \diam(R)
\end{equation}
because $R \in S$ (a stopping time region with constant $\alpha$), and by \eqref{6.39}.
Since $W(Q_R)$ is the graph of $A_R$ and $A_R$ is $C \alpha$-Lipschitz, this yields
$|\pi^\perp(x) - A_R(\pi(x))| \leq C \alpha^\eta \diam(R)$, hence
\begin{equation} \label{6.51}
|\pi^\perp(y)-A_R(\pi(x))| \leq C \alpha^{\eta} (|x-y|+\diam(R)).
\end{equation}
Now we evaluate horizontal distances. Notice that, just because
$x\in Q_R$ and $p\in R$,
\begin{eqnarray} \label{6.52}
|\pi(x)-p| &\leq& \dist(\pi(x),R) + \diam(R) = \dist(x,\pi^{-1}(R)) + \diam(R)
\nonumber \\
&\leq& \diam(Q_R) + \dist(Q_R,\pi^{-1}(R)) + \diam(R)
\nonumber \\
&\leq& 2d(Q_R) + + \diam(R) \leq 721 \diam(R)
\end{eqnarray}
by \eqref{6.39}. Then if $\alpha$ is small enough, by \eqref{6.49}, \eqref{6.52} and \eqref{6.35}, and using the facts that $\pi(y)=q$,
and that $D$ is $1$-Lipschitz and vanishes at $q$ we have
\begin{eqnarray} \label{6.53}
|x-y| &\leq& 2 |\pi(x)-\pi(y)| \leq |\pi(x)-p| + |p-q|
\\
&\leq& 721 \diam(R) + |p-q| \leq 40 D(p) + |p-q|
\leq 41|p-q|.
\nonumber
\end{eqnarray}
Notice that
$$
|A_R(\pi(x))-A_R(p)| \leq C \alpha |\pi(x)-p| \leq C \alpha \diam(R)
$$
because $A_R$ is $C\alpha$-Lipschitz and by \eqref{6.52}, and that
\begin{equation} \label{6.54}
|A_R(p)-A(p)| = \sum_{H} \varphi_H(p) [A_R(p)-A_H(p)]
\leq C \alpha \diam(R)
\end{equation}
because we only sum over $H$ such that $3H$ contains $p$, and by
\eqref{6.46}. We compare these estimates with \eqref{6.51} and get that
\begin{equation} \label{6.55}
|\pi^\perp(y)-A(p)| \leq C \alpha^{\eta} (|x-y|+\diam(R))
\leq C \alpha^{\eta} |p-q|
\end{equation}
by \eqref{6.53}. This proves \eqref{6.48} because $y = q + A(q)$ by definition
of $A$ on $\pi(Z)$.

We may now prove that $A$ is $C\alpha^\eta$-Lipschitz on $U$.
We just need to check that $|A(p)-A(q)| \leq C\alpha^\eta |p-q|$
for $p,q \in U \sm \pi(Z)$. If the segment $[p,q]$ does not meet
$\pi(Z)$, we just integrate $DA$ on $[p,q]$ and use \eqref{6.47}.
Otherwise, we pass through a point of $[p,q] \cap \pi(Z)$ and
use \eqref{6.48} twice.

To check \eqref{6.44}, and for this it would be useful to know that
\begin{equation} \label{6.56}
{d_S(x) \over 2} \leq D(\pi(x)) \leq d_S(x)
\ \text{ for } x\in 50Q(S).
\end{equation}
The second inequality comes directly from the definition \eqref{6.24}. 
For the first one, we need to check that if $x\in 50Q(S)$, then
\begin{equation} \label{6.57}
d_S(x) \leq 2 d_S(y) \ \text{ for every } y \in \pi^{-1}(x).
\end{equation}
Let such $y$ be given, and assume that $y \neq x$. 
Notice that $d_S(x) \leq \dist(x,Q(S))+ d(Q(S)) \leq 50d(Q(S))$, 
by \eqref{6.23} and because $Q(S) \in S$. If $y \notin 100 Q(S)$,
$d_S(y) \geq \dist(y,Q(S)) > d_S(x)$, and \eqref{6.57} holds.
Otherwise, $y\in 100 Q(S)$, we can apply Lemma \ref{t6.4}, and since 
\eqref{6.26} fails (if $\alpha$ is small enough and because $\pi(x)=\pi(y)$),
we get that $|x-y| \leq {1 \over 10} d_S(y)$. Then
$d_S(x) \leq d_S(y) + |x-y| < 2 d_S(y)$, as needed.
So \eqref{6.56} holds.

Now we check \eqref{6.44}.
Let $x\in 8Q(S)$ be given. If $x\in Z$, then $x = \pi(x) + A(\pi(x))$ by
\eqref{6.34}, and \eqref{6.44} holds. Let us assume that
$d_S(x) > 0$. Set $p = \pi(x)$ and notice that $p \in U$,
by \eqref{6.32} and because $\diam(Q(S)) \leq 2 d(Q(S))$ by \eqref{6.2}.
By \eqref{6.56}, $D(p) \geq d_S(x)/2 > 0$, so $p$ lies in some cube
$R \in \cD$. Let $y$ be any point of $Q_R$. 
If $|x-y| \leq {1 \over 10} d_S(x)$, then 
$$
\dist(x,Q_R) \leq {1 \over 10} d_S(x) \leq {2 \over 10} D(p) \leq 12 \diam(R)
\leq  9 d(Q_R)
$$
by \eqref{6.36} and \eqref{6.40}, so $x\in 10Q_R$. Otherwise,
we can apply Lemma \ref{t6.4} and get that \eqref{6.26} holds, hence
$$
|x-y| \leq 2 |\pi(x)-\pi(y)| = 2|p-\pi(y)| \leq 712\diam(R)
$$
by \eqref{6.52}. In this case, 
$$
\dist(x,Q_R) \leq |x-y| \leq 712\diam(R) \leq 500 d(Q_R).
$$

In both cases, $x\in \lambda Q_R$, so \eqref{6.13} says that
\[
\dist(x,W(Q_R)) \leq C \alpha(Q_R)^\eta d(Q_R) \leq C \alpha^\eta \diam(R)
\]
as in \eqref{6.50}, and 
\begin{equation} \label{6.58}
|\pi^\perp(x)-A_R(\pi(x))| \leq C \alpha^{\eta} \diam(R)
\end{equation}
because $W(Q_R)$ is the graph of $A_R$, which is $C\alpha$-Lipschitz
(recall that $\alpha$ is small).
Since $\pi(x) = p \in R$, we get that
\begin{eqnarray} \label{6.59}
|A_R(\pi(x))-A(\pi(x))| &=&
|A_R(p)-A(p)| = \sum_{H} \varphi_H(p) [A_R(p)-A_H(p)] 
\nonumber
\\
&\leq& C \alpha \diam(R) 
\end{eqnarray}
as in \eqref{6.54}. The second part of \eqref{6.44} follows from this and 
\eqref{6.58}, because $20\diam(R) \leq D(p) \leq d_S(x)$ by \eqref{6.35}
and \eqref{6.56}. The first part follows from the second
one because $x= p + \pi^\perp(x)$ and $p+A(p) \in \Gamma_S$, since $p\in U$.
This completes our proof of Lemma~\ref{t6.5}.
\end{proof}

At this time the purely geometric construction is complete, and we may choose $\lambda$ and $\lambda^\ast$, so that they satisfy the constraints
above.

\subsection{A corona decomposition}
\label{corona}

In the previous subsection we proved some estimates relative
to a stopping time region with small constant $\alpha > 0$.
We now need to construct such regions, estimate how many
there are and (in the next subsection) use them to find bi-Lipschitz images in $\Sigma$.
We proceed as in \cite{DS} and \cite{of-and-on}.

Let us first explain how to construct the stopping regions.
Recall that we want to prove Theorem \ref{t1.6},
so we are given a ball $B = B(x,r)$ centered on $\Sigma$.
Without loss of generality, we may assume that $B = B(0,1)$.

We have assumed for the previous sections that $\mu$ has no atoms,
and in fact it is enough to assume that $\mu$ has no atoms in $2B$. 
(If there were atoms the dyadic cubes of Subsection \ref{cubes} would still exist, but we would
have to be more careful in the construction above.) Note that if $\mu$ has an atom at
$x_0$, the integral $\int_0^1 \alpha_d(x_0,r) { dr \over r }$ diverges.
Then the existence of an atom in $2B$ would contradict \eqref{1.19}.

We do not need all the cubes of $\Delta$, and we restrict to
the set
\begin{equation} \label{6.60}
\Delta_B = \big\{ Q \in \Delta \, ; \, 4 \lambda^\ast d(Q) \leq 1 \text{ and } 
Q \cap B \neq \emptyset \big\}.
\end{equation}
We first define a bad set of cubes
\begin{equation} \label{6.61}
\cB = \big\{ Q \in \Delta_B \, ; \, \alpha(Q) \geq \alpha \big\},
\end{equation}
where $\alpha > 0$ is the same very small constant as in the previous section,
and $\alpha(Q)$ is defined by \eqref{6.10}. 

Set $\cG = \Delta_B \sm \cB$. We want to decompose $\cG$ into
stopping time regions with constant $\alpha$. For $Q \in \Delta$, set 
\begin{equation} \label{6.62}
\Delta(Q) = \big\{ R \in \Delta \, ; \,  R \subset Q \text{ and } d(R) \leq d(Q) \big\};
\end{equation}
we added the strange condition $d(R) \leq d(Q)$ because a same set could
correspond to cubes of different generations, and $d(Q)$ determines
the generation of $Q$. If $R \in \Delta(Q)$, we also define the chain of cubes 
between $Q$ and $R$ as
\begin{equation} \label{6.63}
\Delta(Q;R) = \big\{ H\in \Delta \, ; \,  R \subset H \subset Q \text{ and } 
d(R) \leq d(H) \leq d(Q) \big\}.
\end{equation}

\begin{definition} \label{t6.6}
For each $Q_0 \in \cG$, the \it{stopping time region below} $Q_0$,
denoted by $S(Q_0)$, is the set of subcubes $Q \in \Delta(Q_0)$ such that 
\begin{equation} \label{6.64}
\sum_{R \in \Delta(Q_0;Q)} \alpha(R) \leq \alpha,
\end{equation}
and the relation \eqref{6.64} also holds for all the siblings of
$Q$ (i.e., all the children of the parent of $Q$).
\end{definition}

Notice that if $Q \in S(Q_0)$, all the cubes of $\Delta(Q_0;Q)$ lie in 
the good set $\cG$. It could be that $S(Q_0)$ is reduced to the single
cube $Q_0$. Note that $S(Q_0)$ is a stopping time region with constant
$\alpha$, and its top cube is $Q_0$.
Although \eqref{6.20} and \eqref{6.64} are similar, there is a difference in that
we specify a coherence condition in Definition~\ref{t6.3}.
 
We cover $\cG$ inductively. We start with a cube
$Q_1 \in \cG$ of maximal size $d(Q_1)$, construct the region $S(Q_1)$,
remove it from $\cG$, look for a cube $Q_2 \in \cG \sm S(Q_1)$ of maximal size,
remove $S(Q_2)$ from $\cG \sm S(Q_1)$, and so on. This gives a partition
of $\cG$ as
\begin{equation} \label{6.65}
\cG = \bigcup_{j\geq 1} S(Q_j) = \bigcup_{S \in \cF} S,
\end{equation}
where we call $\cF$ the collection of stopping time regions $S(Q_j)$.

Our goal in this subsection is to construct a 
Lipschitz graph $\Gamma_S$ for each $S\in\cF$ and to 
control the size of
$\cB$ and $\cF$.

\begin{proposition} \label{t6.7}
There is a constant $C \geq 0$, that depends on $d$, $\CD$,
$\lambda$, and $\lambda^\ast$, with the following properties. 
\begin{equation} \label{6.66}
\sum_{Q \in \Delta_B} \alpha(Q) \mu(Q) \leq C C_1 \mu(B),
\end{equation}
where $C_1$ is the constant in \eqref{1.19}. Moreover if
\begin{equation} \label{6.67}
\cB' = \big\{ Q \in \Delta_B \, ; \, \alpha(Q) \geq \alpha/2 \big\};
\end{equation}
then
\begin{equation} \label{6.68}
\sum_{Q \in \cB'} \mu(Q) \leq C \alpha^{-1} C_1 \mu(B).
\end{equation}
Finally,
\begin{equation} \label{6.69}
\sum_{S \in \cF} \mu(Q(S)) \leq C (1 + \alpha^{-1} C_1) \mu(B).
\end{equation}
\end{proposition}

\ms\begin{proof}  
Let us first observe that by \eqref{6.10} and Chebyshev,
\begin{eqnarray} \label{6.70}
\alpha(Q) &\leq& \fint_{ Q} \fint_{\{\lambda^\ast d(Q)
\leq r \leq 2\lambda^\ast d(Q)\}} \alpha_d(x,r) d\mu(x) dr
\nonumber\\
&\leq& 2\mu(Q)^{-1} \int_{ Q} \int_{\{\lambda^\ast d(Q)\leq r 
\leq 2\lambda^\ast d(Q)\}} \alpha_d(x,r) d\mu(x) {dr \over r}
\end{eqnarray}
Denote by $\cR(Q)$ the region of integration, i.e., set
$\cR(Q) = Q \times [\lambda^\ast d(Q),  2\lambda^\ast d(Q)]$.
Notice that $\cR(Q) \subset 2B \times (0,2)$, in particular because 
$4\lambda^\ast d(Q) \leq 1$ when $Q\in \Delta_B$ (see \eqref{6.60}).
Also notice that for a given $(x,r)$, if $\1_{\cR(Q)}(x,r) =1$,
then $d(Q)$ is a power of $2$
such that $\lambda^\ast d(Q) \leq r \leq 2\lambda^\ast d(Q)$,
so it can only take one or two values, and for each one there is a unique cube
$Q$ that contains $x$. Thus
\begin{equation} \label{6.71}
\sum_{Q \in \Delta_B} \1_{\cR(Q)} \leq 2,
\end{equation}
and now by \eqref{6.70}, \eqref{6.71}, and \eqref{1.19}
\begin{eqnarray} \label{6.72}
\sum_{Q \in \Delta_B} \alpha(Q) \mu(Q) 
&\leq& 2 \sum_{Q \in \Delta_B}
\int_{\cR(Q)} \alpha_d(x,r) d\mu(x) {dr \over r}
\\
&\leq&  2 \int_{\Sigma \cap B(0,2)}\int_0^{2} 
\alpha_d(x,r) \Big(\sum_{Q \in \Delta_B} \1_{\cR(Q)}(x)\Big) d\mu(x) {dr \over r} 
\nonumber\\
&\leq& 4  \int_{\Sigma \cap B(0,2)}\int_0^{2} \alpha_d(x,r) d\mu(x) {dr \over r}
\leq 4 C_1 \mu(B).
\nonumber
\end{eqnarray}
This proves \eqref{6.66}.

Notice that \eqref{6.68} follows from \eqref{6.66} and Chebyshev;
hence we are left with \eqref{6.69} to check.
We distinguish between different types of stopping time regions $S$,
based on the behavior of the set $M(S)$ of minimal cubes of $S$
(see \eqref{6.21}).

First observe that if $S \in \cF$ and $Q \in M(S)$, there is a child $H$ of $Q$
for which the condition \eqref{6.64} fails (because otherwise we would have added
all the children of $Q$ to $S$). Select such a child, and call it $H(Q)$. Then set
\begin{equation} \label{6.73}
M_1(S) = \big\{ Q \in M(S) \, ; \, H(Q) \in \cB' \big\}
\ \text{ and } \ M_2(S) = M(S) \sm M_1(S).
\end{equation}
Note that by \eqref{6.68} the cubes $Q$ such that $Q \in M_1(S)$ are rare.
If $Q \in M_2(S)$, we use the fact that since \eqref{6.64} 
fails for $H(Q)$,
\begin{equation} \label{6.74}
\sum_{R \in \Delta(Q(S);Q)} \alpha(R) 
= \sum_{R \in \Delta(Q(S);H(Q))} \alpha(R) - \alpha(H(Q)) \geq {\alpha \over 2}
\end{equation}
(because $\alpha(H(Q)) \leq \alpha/2$ by definition of $M_2(S)$ and $\cB'$).
To prove \eqref{6.69} we first control the set $\cF_1$ of regions $S\in \cF$ such that
\begin{equation} \label{6.75}
\sum_{Q \in M_1(S)} \mu(Q) \geq {\mu(Q(S)) \over 3}.
\end{equation}
Notice that by \eqref{6.75}, \eqref{6.2} and \eqref{1.1}, and \eqref{6.68}
\begin{eqnarray}  \label{6.76}
\sum_{S \in \cF_{1}} \mu(Q(S)) 
&\leq& 3 \sum_{S \in \cF_{1}} \sum_{Q \in M_1(S)} \mu(Q)
\leq C\sum_{S \in \cF_{1}} \sum_{Q \in M_1(S)} \mu(H(Q))
\nonumber \\
&\leq& C \sum_{H \in \cB'} \mu(H) \leq C \alpha^{-1} C_1 \mu(B).
\end{eqnarray}

Next consider the set $\cF_2$ of regions $S\in \cF$ such that
\begin{equation} \label{6.77}
\sum_{Q \in M_2(S)} \mu(Q) \geq {\mu(Q(S)) \over 3}.
\end{equation}
Let $S \in \cF_{2}$ be given. Observe that by \eqref{6.74} by changing the order of summation we have
\begin{eqnarray}\label{6.77A}
\mu(Q(S)) &\leq& 3 \sum_{Q \in M_2(S)} \mu(Q) 
\leq 6\alpha^{-1} \sum_{Q \in M_2(S)} \sum_{R \in \Delta(Q(S);Q)} \alpha(R) \mu(Q)\nonumber\\
&\leq& 6\alpha^{-1} \sum_{R  \in \Delta(Q(S))}\sum_{Q \in \Delta(R),\, Q\in M_2(S)}  \alpha(R) \mu(Q)\nonumber\\
&\leq& 6\alpha^{-1} \sum_{R  \in \Delta(Q(S))} \alpha(R) \mu(R),
\end{eqnarray}
where we have used the fact that
subcubes $Q \in \Delta(R)$ that lie in $M_2(S)$ are disjoint (by minimality, \eqref{6.1}, and the 
nesting property \eqref{6.3}).
Using \eqref{6.77A} we have
\begin{eqnarray}  \label{6.78}
\sum_{S \in \cF_{2}} \mu(Q(S)) 
&\leq& 6\alpha^{-1} \sum_{S \in \cF_{2}} \sum_{R \in \Delta(Q(S))} \alpha(R) \mu(R)
\nonumber \\
&\leq&  6\alpha^{-1}  \sum_{R \in \Delta_B} \alpha(R) \mu(R)
\leq C \alpha^{-1} C_1 \mu(B)
\end{eqnarray}
because the stopping time regions $S$ are disjoint and contained in
$\Delta_B$, and by \eqref{6.66}. We are thus left with the set $\cF_3$ of regions 
$S\in \cF$ such that
\begin{equation} \label{6.79}
\sum_{Q \in M(S)} \mu(Q) \leq {2\mu(Q(S)) \over 3}
\end{equation}
(the set $\cF_3$ may intersect the previous ones, but this is all right).
Set for $S\in \cF$
\begin{equation} \label{6.80}
Z_0(S) = Q(S) \sm \Big(\cup_{Q \in M(S)} Q \Big).
\end{equation}
We claim that for $x\in Z_0(S)$, and for $k$ large, 
the unique cube $Q \in \Delta_k$ that contains $x$ 
(see \eqref{6.1}) lies in $S$. 
Indeed, let $Q \in \Delta$ contain $x$, and suppose
that $d(Q) \leq d(Q(S))$. Then $Q$ is a subcube of $Q(S)$.
Suppose that $Q \notin S$, and let $Q_1$ be the smallest ancestor
of $Q$ that lies in $S$. Thus the child of $Q_1$ that contains $Q$
does not lie in $S$. By the coherence rule, none of the children of
$Q_1$ lie in $S$. Then $Q_1 \in M(S)$, a contradiction since $x\in Z_0(S)\cap Q_1$.
Notice also that 
\begin{equation} \label{6.81}
Z_0(S) \cap Z_0(S') \ \text{ for } S \neq S' \in \cF,
\end{equation}
simply because $S$ and $S'$ are disjoint. In fact the small cubes
that contain $x\in Z_0(S)$ are contained in $S$ and therefore 
cannot be contained in $S'$.

Since by \eqref{6.79},  $\mu(Z_0(S)) \geq \mu(Q(S))/3$ for $S \in \cF_3$,
we get that
$$
\sum_{S \in \cF_{3}} \mu(Q(S)) 
\leq 3 \sum_{S \in \cF_{3}} \mu(Z_0(S))
\leq  3 \mu(\cup_{S \in \cF} Z_0(S)) \leq 3 \mu(2B) \leq 3 \CD \mu(B)
$$
by \eqref{6.81} and because $Z_0(S) \subset Q(S) \subset 2B$ for $S \in \cF$.
This last estimate completes the proof of \eqref{6.69}. Proposition \ref{t6.7} follows.
\end{proof}

\subsection{Big pieces of bi-Lipschitz images}
\label{s:big-piece}

In this subsection we use the Lipschitz graphs $\Gamma_S$
associated to the stopping time regions of Section \ref{corona} 
to construct a large set $A \subset \Sigma$, and a bi-Lipschitz mapping
$f : A \to f(A)$, with values in $\R^d$ (as in the statement of Theorem \ref{t1.6}).

Our proof follows Section 16 in \cite{DS}, which we just need to modify slightly
because $\mu$ is a doubling measure which is not necessarily Ahlfors regular.
The small boundary condition \eqref{6.7} is useful here, because we need
to separate cubes from each other without removing to much mass.
For this purpose, we use a new constant $\rho \in (0,1)$,
very close to $1$. We replace many cubes $Q$ with
the slightly smaller $\rho Q$ defined in \eqref{6.6}.

Since $B$ is not one of our pseudo-cubes, we add a top layer to our construction,
and set
$$
Q_0 = \bigcup_{Q \in \Delta_B} Q,
$$
which we see as a common ancestor of all cubes. If $\lambda^\ast$
is taken large enough, we can be sure that the cubes of $\Delta_B$
are all strictly contained in $Q_0$ (see \eqref{6.60}). 
Otherwise, we consider $Q_0$ as a cube of the previous
generation even if it coincides with a different cube as a subset of $\Sigma$. 
Let us also set $d(Q_0)=1$.
The set $A$ is obtained by removing from $\Sigma \cap B$
a certain number of small sets. 

Let us define an exceptional set of cubes by
\begin{equation} \label{6.82}
T'= \{ Q_0 \} \cup \cB\cup \{Q(S):S\in \cF\}\cup\{M(S): S\in \cF\}
\end{equation}
and, for $Q \in \Delta_B$, denote by 
$j(Q)$ the number of cubes $R \in T'$ such that
$Q \subset R$ and $d(Q) < d(R)$. Thus $j(Q_0) = 0$,
and $j(Q) \geq 1$ for all the other cubes.
Let $N$ be a large number, to be chosen soon, and set
\begin{equation} \label{6.83}
T=\big\{ Q\in T' \, ; \, j(Q) \leq N \big\}.
\end{equation}
Consider the sets 
\begin{equation} \label{6.84}
F_1 = \bigcup_{Q\in T'\sm T} Q  \ \text{ and } \  
F_2 = \bigcup_{Q\in T'}  (Q\sm \rho Q).
\end{equation}
We show that if $\rho$ and $N$ are chosen correctly,
\begin{equation} \label{6.85}
\mu(F_1) + \mu(F_2) < {\gamma \over 2}\, \mu(B),
\end{equation}
where $\gamma$ is the small constant from Theorem \ref{t1.6}.
First notice that $T'$ is not too large. In fact using Proposition~\ref{t6.7}
as well as the fact that the cubes of $M(S)$ are disjoint and contained in
$Q(S)$ we have
\begin{align*}
\sum_{Q\in T'}\mu(Q)
& \leq \sum_{Q\in \cB}\mu(Q) + \sum_{S\in \cF}\mu(Q(S))
+\sum_{S\in \cF}\sum_{Q\in M(S)}\mu(Q) 
\\
& \leq \sum_{Q\in \cB}\mu(Q)+2\sum_{S\in \cF}\mu(Q(S))
\leq C(1+\alpha^{-1} C_1) \mu(B).
\end{align*}
Next observe that since
each cube of $T'\sm T$ is contained in at least $N$ cubes of $T'$,
$\sum_{Q \in T'} \1_{Q} \geq N$ on $F_1$; then
\begin{align}\label{6.85A}
\mu(F_1) &= \int_{F_1} d\mu \leq N^{-1} \int \sum_{Q \in T'} \1_{Q} d\mu
= N^{-1} \sum_{Q \in T'} \mu(Q) \nonumber
\\
&\leq C N^{-1} (1+\alpha^{-1} C_1) \mu(B)
\leq {\gamma \over 4}\, \mu(B) 
\end{align}
if $N$ is chosen large enough, depending also on $\alpha$, $C_1$, and $\gamma$.
Moreover, by \eqref{6.7}
\begin{align}\label{6.85B}
\mu(F_2) &\leq \sum_{Q\in T'} \mu(Q\sm \rho Q)
\leq C (1-\rho)^\kappa \sum_{Q\in T'} \mu(Q)\nonumber
\\
&\leq C (1-\rho)^\kappa (1+\alpha^{-1} C_1) \mu(B)
\leq {\gamma \over 4}\, \mu(B) 
\end{align}
if $\rho$ is chosen close enough to $1$, depending on $\alpha$, $C_1$, and $\gamma$.
We choose $N$ and $\rho$ so that \eqref{6.85A} and \eqref{6.85B} are satisfied, and get \eqref{6.85}.
We take
\begin{equation} \label{6.86}
A = (\Sigma \cap B) \sm (F_1 \cup F_2),
\end{equation}
and then \eqref{1.20} follows from \eqref{6.85}. We even have room to
remove a tiny piece from $A$ when we deal with density.

\ms
We still need to define a bi-Lipschitz mapping $f$ on $A$.
Let us first check that
\begin{equation} \label{6.87}
A \subset \bigcup_{S\in \cF \atop j(Q(S))\leq N} Z_0(S),
\end{equation}
where $Z_0(S)$ is as in \eqref{6.80}. Let $x\in A$ be given,
and denote by $Q_k(x)$ the unique cube of $\Delta_k$ that contains
$x$ (see \eqref{6.1}). At most $N$ of these cubes $Q_k(x)$ lie in $T'$,
because they are all nested, and if there were more than $N$,
one of them would lie in $T \sm T'$; this is impossible because $x\notin F_1$.
The largest of these $Q_k(x)$ that lies in $\Delta_B$ is a cube of $T'$,
because it is either a bad cube or the top cube of a stopping time region.
Then let $Q$ be the smallest $Q_k(x)$ that lies $T'$. The $Q$ cannot be
a bad cube or a minimal cube of some region, because its child that contains $x$
is either bad or the top cube of some new region, hence lies in $T'$
(a contradiction with the minimality of $Q$). So $Q = Q(S)$ for some
$S\in \cF$. And $j(Q(S))\leq N$ because otherwise $Q \in T\sm T'$
(which is impossible because $x\in A \cap Q$). So \eqref{6.87} holds.

For each $S \in \cF$, Lemma \ref{t6.5} gives a Lipschitz graph $\Gamma_S$.
Let us change notation slightly to avoid confusion. Denote by 
$\pi_S$ (instead of $\pi$) the orthogonal projection on the set
$W_S = W(Q(S))$, and call $A_S$ (instead of $A$) the Lipschitz function 
of Section \ref{corona}. Thus $A_S$ is defined in
$U_S = W_S \cap B(p_S,10d(Q(S)))$, where $p_S = \pi_S(x_{Q(S)})$
is just some point of $\pi_S(Q(S))$ (see \eqref{6.32}), and $\Gamma_S$
is the graph of $A_S$ over $U_S \subset W_S$.

We also have a set $Z(S)$, defined by \eqref{6.25}, which contains
$Z_0(S)$. In fact note that points of $Z_0(S)$ lie in arbitrarily small cubes of $S$,
and therefore by definition \eqref{6.23}) lie in $Z(S)$, which is contained in $\Gamma_S$
(by \eqref{6.44} and because $d_S(x) = 0$ forces $x\in Q(S)$).

\begin{lemma} \label{t6.8}
Using the notation above we have that if $\alpha$ (as defined in \eqref{6.64}) is small enough, then for every $S\in \cF$, the projection 
$\pi_S$ is $2$-bi-Lipschitz on the set 
\begin{equation} \label{6.88}
E(S) = Z(S) \cup \big\{ c_Q \, ; \, Q \in M(S) \big\} \subset Q(S).
\end{equation}
\end{lemma}

Recall that $c_Q$ is the center of $Q$ provided by \eqref{6.2}.
As usual, small enough depends on 
$\CD$, $\lambda$, and $\lambda^\ast$.

\begin{proof}  
We know that $E(S) \subset Q(S)$ (because $d_S(x) = 0$ forces $x\in Q(S)$), and
that $Z(S) \subset \Gamma_S$,. Now we show that for $Q\in M(S)$, $c_Q$ is not far. from $\Gamma_S$. Set $p_Q = \pi_S(c_Q)$ and
notice that $d_S(c_Q) \leq d(Q)$ by \eqref{6.23}; so \eqref{6.44} says that
\begin{equation} \label{6.89}
|\pi^\perp(c_Q) - A_S(p_Q)| \leq C \alpha^\eta d(Q),
\end{equation}
where we still denote by $\pi^\perp$ the projection on the vector space $W_S^\perp$.

Set $a = (2C)^{-1}$, where $C$ is as in \eqref{6.2}; then \eqref{6.2} says
that 
\begin{equation} \label{6.90}
\Sigma \cap B(c_Q,2ad(Q)) \subset Q
\ \text{ for } Q \in \Delta.
\end{equation}
We claim that
$$
d_S(z) \geq a d(Q) \ \text{ for } z\in \Sigma \cap B(c_Q, ad(Q)).
$$
Indeed, if $R$ is a cube of $S$, either $R \supset Q$ and
then $\dist(z,R) + d(R) \geq d(R) \geq d(Q)$, or else
$R$ does not meet $Q$ and then $\dist(z,R) + d(R) \geq \dist(z,\Sigma \sm Q)
\geq a d(Q)$. This proves the claim (see \eqref{6.23}), which itself implies that
\begin{equation} \label{6.91}
B(c_Q, ad(Q)) \cap Z(S) = \emptyset.
\end{equation}
Next we claim that 
\begin{equation} \label{6.92}
\pi^\perp \text{ is $C \alpha^\eta$-Lipschitz on } E(S), 
\end{equation}
where now $C$ also depends on $A$. Let $x, z \in E(S)$ be given.
If $x = c_Q$ and $y = c_R$ for different cubes $Q, R \in M(S)$, 
\begin{align*}
|\pi^\perp(x) - \pi^\perp(y)|
&= |\pi^\perp(c_Q) - \pi^\perp(c_R)| 
\\& \leq |A_S(p_Q)-A_S(p_R)| + 
|\pi^\perp(c_Q) - A_S(p_Q)| + |\pi^\perp(c_R) - A_S(p_R)|
\\& \leq C \alpha^\eta |p_Q-p_R| + C \alpha^\eta (d(Q)+d(R))
\leq C \alpha^\eta |x-y|
\end{align*}
by \eqref{6.89}, because $|p_Q-p_R| \leq |c_Q-c_R|=|x-y|$, and more importantly
because $|x-y| \geq \dist(x,R) \geq 2a d(Q)$ and similarly
$|x-y| \geq 2a d(R)$. If $x = c_Q$ and $y \in Z(S)$,
\begin{align*}
|\pi^\perp(x) - \pi^\perp(y)|
&= |\pi^\perp(c_Q) - A_S(\pi_S(y))| 
\\& \leq |A_S(p_Q)-A_S(\pi_S(y))| + |\pi^\perp(c_Q) - A_S(p_Q)| 
\\& \leq C \alpha^\eta |p_Q-\pi_S(y)| + C \alpha^\eta d(Q)
\leq C \alpha^\eta |x-y|
\end{align*}
because $y \in \Gamma_S$ and $|x-y| = |c_Q-y| \geq \dist(c_Q,Z(S)) \geq ad(Q)$. 
The two other cases are similar, and \eqref{6.92} follows.
If $\alpha$ is small enough, \eqref{6.92} implies that for $x, y \in E(S)$,
$$
|x-y| \geq |\pi_S(x) - \pi_S(y)| \geq |x-y|- |\pi_S^\perp(x) - \pi_S^\perp(y)|
\geq (1-C \alpha^\eta) |x-y| \geq {1 \over2} |x-y|;
$$
Lemma \ref{t6.8} follows.
\end{proof}

\ms
Next we arrange the various $\pi_S$ to form a single map.
Denote by $T^\ast$ the set of cubes $Q \in T$ such that $Q = Q(S)$ 
for some $S \in \cF$ which is not reduced to $Q$. [When 
$S$ is just composed of $Q$, we find it more convenient
to see $Q$ as a minimal cube.]

When $Q \in T^\ast$, set $E_Q = E(S)$ and define a mapping 
$g_Q : E_Q \to \R^d$ by 
\begin{equation} \label{6.93}
g_Q(x) = \varphi \circ \pi_S(x),
\end{equation}
where $\varphi$ is an isometry from $W_S$ to $\R^d$ such that
$\varphi(\pi_S(c_Q)) = 0$ (to normalize). Notice that
$g_Q$ is still $2$-bi-Lipschitz on $E_Q = E(S)$.

For the other cubes $Q \in T\sm T^\ast$, we also define a set $E_Q$
and a mapping $g_Q : E_Q \to \R^d$. Denote by ${\rm ch}(Q)$
the set of children of $Q$, and set for $Q\in T$ not a top cube
\begin{equation} \label{6.94}
E_Q = \big\{ c_R \, ; \, R \in {\rm ch}(Q) \big\}.
\end{equation}
Notice that 
$E_Q  \subset B(c_Q, d(Q))$, by \eqref{6.2}, and 
\begin{equation} \label{6.95}
|x-y| \geq a d(Q) \ \text{ for }  x \neq y \in E_Q,
\end{equation}
where $a$ is still as in \eqref{6.90}
(notice that $c_R \notin T$ when $R \neq T \in {\rm ch}(Q)$,
and apply \eqref{6.90} to $T$). Also, $E_Q$ has at most $C_1$ elements, 
where $C_1$ depends on $\CD$ (just observe that for $R \in {\rm ch}(Q)$,
$Q \subset B(c_R, 2d(Q)) \subset C B(c_R, 2ad(R))$ and that 
$B(c_R, 2ad(R)) \subset R$). This also works in the special case of $Q = Q_0$,
except that we have to take $C_1$ even larger, to account for the jump of
size between $d(Q_0)$ and the next ones.

Pick a set $X_Q \subset \R^d \cap B(0,d(Q))$ with the same number of elements as 
$E_Q$. We can do this so that 
\begin{equation} \label{6.96}
|u-v| \geq c d(Q) \ \text{ for } u \neq v \in X_Q,
\end{equation}
where $c$ also depends on $\CD$ through $C_1$. Then
let $g_Q = E_Q \to X_Q \subset \R^d$ be any bijection.
Notice that $g_Q$ is bi-Lipschitz, because 
\begin{equation} \label{6.97}
{c \over 2}  |x-y| \leq  cd(Q) \leq |g_Q(x) -g_Q(y)| \leq 2d(Q) \leq {2 \over a} |x-y| 
\end{equation}
for $x \neq y \in E_Q$, by \eqref{6.95} and \eqref{6.96}.
Let us record the fact that for $Q \in T$,
\begin{equation} \label{6.98}
g_Q(E_Q) \subset \R^d \cap B(0,d(Q)).
\end{equation}
When $Q \in T^\ast$, this comes from \eqref{6.88},
\eqref{6.2}, the fact that $g_Q$ is $1$-Lipschitz, and our normalization 
$g_Q(c_Q) = 0$. Otherwise, this is because $X_Q \subset B(0,d(Q))$.

\ms
The functions $g_Q$ for $Q \in T$, are basic building blocks that need
to be glued together. We now focus on the  different levels in our
implicit stopping time construction.
Decompose $T$ into the $N+1$ disjoint families
\begin{equation} \label{6.99}
T_j = \big\{ Q\in T \, ; \, j(Q) = j \big\}, \, 0 \leq j \leq N.
\end{equation}
For $Q \in T_j$ and for each $0 \leq i \leq j$, denote by $Q^i$ the cube of $T_i$
that contains $Q$. These are the predecessors of $Q$ in the iterated stopping time,
with $Q^j = Q$ and $Q^0 = Q_0$. Set $Q^\ast = Q^{j-1}$ when
$j \geq 1$ (the previous stopping cube).

We need transition maps that link $Q$ to the $Q^i$s.
For $Q \in T$ such that $j(Q) \geq 1$, define $h_Q$ by
\begin{equation}\label{6.100}
h_{Q}(x)= 10^{-1}ac x  +g_{Q^\ast}(c_Q)
\end{equation}
(where $a < 1$ is as in \eqref{6.90} and $c > 0$ is as in \eqref{6.96}).
This map is defined everywhere, but we use it when $x \in \R^d$,
and then $h_Q(x) \in \R^d$ because $g_{Q^\ast}(c_Q) \in \R^d$.
Typically, we have points of $\R^d$ coming from $g_Q$ or
previous constructions, and $h_Q$ sends them to their right
place in the next construction, near $g_{Q^\ast}(c_Q)$.
The fact that $h_Q$ is a contraction will help separate the pieces.
Notice that
\begin{equation} \label{6.101}
h_Q(B(0,2d(Q)) \subset B\big(g_{Q^\ast}(c_{Q}), 5^{-1}ac d(Q)\big).
\end{equation}

To construct the bi-Lipschitz function we need we first  compose the $h_Q$s. For $Q \in T_j$, $j \geq 1$, and 
$1 \leq i \leq j$, set
\begin{equation} \label{6.102}
h_Q^i = h_{Q^i} \circ h_{Q^{i+1}} \circ \ldots \circ h_{Q^j}
= h_{Q^i} \circ h_{Q^{i+1}} \circ \ldots \circ h_{Q},
\end{equation}
and then define $f_Q : E_Q \to \R^d$ by
\begin{equation} \label{6.103}
 f_Q = h_Q^1 \circ g_Q.
\end{equation}
 
Recall that for different stopping regions $S$, the sets $Z_0(Q(S))$ are disjoint, by \eqref{6.81} so we may set
\begin{equation} \label{6.104}
f(x) = f_{Q(S)}(x) \ \text{ for } x\in Z_0(S)
\end{equation}
for all $S \in \cF$. In fact, we may restrict our attention
to the regions $S$ such that $Q(S) \in T^\ast$, because 
otherwise $S$ is just composed of its top cube $Q(S)$
(by definition of $T^\ast$, see above \eqref{6.93}), then
$Q(S) \in M(S)$ (see \eqref{6.21}), and $Z_0(S)$ is empty (see \eqref{6.80}).

Recall from \eqref{6.87} and \eqref{6.83} that $A \subset \cup_S Z_0(S)$, 
where the union is over the regions $S$ such that $Q(S) \in T$;
thus \eqref{6.104} gives a definition of $f$ on a set that contains $A$. 
Naturally we want to show that $f$ is bi-lipschitz on $A$.

\begin{lemma} \label{t6.9}
There is a constant $L$ such that the function $f$ constructed in \eqref{6.104} satisfies
\begin{equation} \label{6.105}
L^{-1} |x-y| \leq |f(x)-f(y)| \leq L |x-y| \ \text{ for } x, y \in A.
\end{equation}
\end{lemma}

\ms
Here $L$ depends on the various constants of
the construction, including $\alpha$ and the recently chosen $N$ and $\rho$.

\begin{proof} 
We first look at how $f(x)$ behaves when $x\in A$. If $x\in A$ then
 $x\in Z_0(S)$ for some $S \in \cF$ (the only one
for which $Q(S)$ contains $x$). Set $Q = Q(S)$ and $j = j(Q) \in [1,N]$
(we know that $Q \neq Q_0$ because $Q$ is a top cube and $Q_0$
has a special status; see the comments below the definition of $Q_0$, 
above \eqref{6.82}).
We also know that $Q \in T^\ast$ (because otherwise $Z_0(S)$ is empty),
and $f(x) = f_Q(x)$ is given by \eqref{6.103}.

We define a sequence of points $x_i$, $0 \leq i \leq j$, so that  
\begin{equation} \label{6.106}
x_j = g_Q(x) \in B(0,d(Q)),
\end{equation}
where $g_Q$ is given by \eqref{6.93} and the last inclusion 
comes from \eqref{6.98} (recall that $x \in Z_0(S) \subset E(S) = E_Q$;
see near Lemma \ref{t6.8} and \eqref{6.93}). 
Then we apply $h_Q = h_{Q^j}$ to $x_j$ and obtain
$$
x_{j-1} = h_Q(x_j) \in B\big(g_{Q^\ast}(c_{Q}), 5^{-1}acd(Q)\big)
= B\big(g_{Q^{j-1}}(c_{Q}), 5^{-1}acd(Q)\big),
$$
by \eqref{6.101}. If $j= 1$, we stop. Otherwise, we continue, and define
$x_{j-2} = h_{Q^{j-1}}(x_{j-1}) = h_Q^{j-1}(x_j)$. We iterate up until
we define the last point $x_0 = f_Q(x) = f(x)$. We claim that for
$0 \leq i < j$, 
\begin{equation} \label{6.107}
x_i = h_{Q^{i+1}}(x_{i+1}) = h_Q^{i+1}(x_j)
\in B\big(g_{Q^{i}}(c_{Q^{i+1}}), 5^{-1}acd(Q^{i+1})\big)
\subset B\big(0, 2d(Q^{i})\big).
\end{equation}
The first identity is a definition, and the second one comes from \eqref{6.102}.
The inclusion above holds for $i = j-1$. We prove the general statement by 
a "backward" induction argument.

To check that the inclusion in \eqref{6.107} holds for $i=j-2$ it is enough 
to show that $c_{Q^{i+1}} \in E_{Q^i}$, because then \eqref{6.98}
will say that $g_{Q^{i}}(c_{Q^{i+1}}) \in B(0,d(Q^{i}))$.
If $Q^{i} \in T^\ast$ and $S'$ is the stopping time region such that $Q^{i}= Q(S')$,
then $E_{Q^i} = E(S')$ is given by \eqref{6.88}. In addition, $Q^{i+1}$,
which lies in the generation just after $Q^i$, is one of the minimal cubes of
$S'$; then $c_{Q^{i+1}} \in E_{Q^i}$. If $Q^{i} \in T \sm T^\ast$,
$E_{Q^i}$ is given by \eqref{6.94}, and also $Q^{i+1}$ is one of the children
of $Q^i$. Then $c_{Q^{i+1}} \in E_{Q^i}$, and inclusion for $j-2$ follows from the inclusion from $j-1$

Finally, the inclusion for $i$ follows from the one for $i+1$,
because \eqref{6.101} says that $h_{Q^{i+1}}(B(0,2d(Q^{i+1}))) 
\subset B\big(g_{Q^{i}}(c_{Q^{i+1}}), 5^{-1}acd(Q^{i+1})\big)$.
This proves our claim \eqref{6.107}.

Incidentally, all our points $x_i$, $0 \leq i \leq j$, lie in $\R^d$, because 
$x_j = g_Q(x) \in \R^d$, and then all the mappings $h_R$ preserve $\R^d$. 

\ms
To check \eqref{6.105}, let $x, y \in A$.
Let $S$, $Q = Q(S)$, $j$, and the $x_i$ be as before. In particular $x\in Z_0(S)$.
Similarly, let $S'$ be such that $y \in Z_0(S')$, set $R = Q(S')$, 
$j' = j(R)$, and define the $y_i$, $0 \leq i \leq j'$ as we did for the $x_i$.

Consider the ancestors $Q^i$ of $Q$, $0 \leq i \leq j$, and 
the ancestors $R^i$ of $R$, $0 \leq i \leq j'$.
Then let $i$ be the largest index such that $Q^i = R^i$. 
This is the smallest common ancestor of $Q$ and $R$ in the 
stopping time construction; notice that $Q^i$ and $R^i$ are conveniently
indexed by their generation. Of course it could be that $i=0$,
if $Q$ and $R$ do not have a common ancestor in $\Delta_B$.

Since $i \leq \min(j,j')$, we have defined the points $x_i$ and $y_i$.
We claim that to prove Lemma \ref{t6.9} it is enough to show that 
\begin{equation} \label{6.108}
C^{-1}|x-y| \leq |x_i-y_i| \leq C |x-y|.
\end{equation}

In fact if $i=0$ since $x_0=f(x)$ and $y_0=f(y)$, \eqref{6.108} is equivalent to \eqref{6.105}. Otherwise, with the notation above, we have that
\begin{equation} \label{6.109}
f(x) = x_0 = h_{Q^1} \circ \ldots \circ h_{Q^i}(x_i)
\end{equation}
(see \eqref{6.107} and \eqref{6.102}) and similarly
\begin{equation} \label{6.110}
f(y) = y_0 = h_{R^1} \circ \ldots \circ h_{R^i}(y_i).
\end{equation}
Since$Q^i = R^i$ (by definition of $i$), then
$Q^l = R^l$ for $1 \leq l \leq i$. Thus the mappings above are the same.
In this case estimate \eqref{6.105} follows from \eqref{6.108}, 
because all the $h_{Q^l}$ are bi-Lipschitz with uniform bounds, and there are at most
at most $N$ of them. 

To show that \eqref{6.108} holds first assume that $Q = R$. This means that $S = S'$
(recall that $Q = Q(S)$ and $R = Q(S')$), and $x, y$ are both
points of $Z_0(Q)$ (and hence $Q \in T^\ast$). 
In this case $i = j = j'$, $x_i = g_Q(x)$, and 
$y_i = g_Q(y)$ (see \eqref{6.106}), and \eqref{6.108} holds
because $g_Q$ is given by \eqref{6.93}, then Lemma \ref{t6.8} 
says that $g_Q$ is $2$-bi-Lipschitz on $E(S)$, and
$x,y \in Z_0(S) \subset E(S)$ (see the remark above Lemma \ref{t6.8}).

We now assume that $Q \neq R$. If $Q \subset R$ then $R = Q^i$ for some $i \leq j$
(because $R \in T$), $i < j$ because $R \neq Q$, and 
$i$ is as above (i.e., $R$ is the smallest common 
ancestor of $Q$ and $R$).
Then $y_i = g_R(y)$ (by \eqref{6.106}), while $x_i$ is given by
 \eqref{6.107} (because $i<j$).

Set $H = Q^{i+1}$. With this notation, \eqref{6.107} for $x_i$ says that
\begin{equation} \label{6.111}
|x_i - g_R(c_{H})| \leq 5^{-1} acd(H).
\end{equation}
Also recall that $R \in T^\ast$. That is, $R = Q(S')$ and
the stopping time region $S'$ is not reduced to $R$.
A first consequence of this is that $H$, which is of the
next generation in $T$, is one of the minimal cubes of $S'$.
Also $g_R$ is given by a formula analogous to \eqref{6.93}, i.e., $g_R$
is equivalent to $\pi_{S'}$. Notice that $y \in Z_0(S')$ lies in
$E(S')$ (see the remark above Lemma \ref{t6.8}),
and $c_H \in E(S')$ too (directly by \eqref{6.88}, since $H \in M(S')$).
Lemma  \ref{t6.8} and \eqref{6.93} yield that
\begin{equation} \label{6.112}
{1 \over 2}|y - c_{H}| \leq |\pi_{S'}(y) - \pi_{S'}(c_{H})|  = |g_R(y) - g_R(c_{H})| 
\leq |y - c_{H}|.
\end{equation}
Recall that $y_i = g_R(y)$; so \eqref{6.111} and \eqref{6.112} yield
\begin{equation} \label{6.113}
{1 \over 2}|y - c_{H}| - 5^{-1} acd(H)
\leq|y_i -x_i | \leq |y - c_{H}|+ 5^{-1} acd(H).
\end{equation}
Next we estimate $|y - c_{H}|$ and $|x-y|$. First observe that
$y \in Z_0(S') \subset \Sigma \sm H$, by \eqref{6.80} because $H \in M(S')$, 
so \eqref{6.90} says that
\begin{equation} \label{6.114}
|y - c_{H}| \geq 2ad(H).
\end{equation}
Also, $Q \subset H$ (because $H = Q^{i+1}$ and $i < j$),
so $x\in H$ and, by \eqref{6.2},
\begin{equation} \label{6.115}
|x - c_{H}| \leq d(H).
\end{equation}
But $x\in A$, hence $x \notin F_2$ by \eqref{6.86}
and $x\in \rho H$ by \eqref{6.84}. Then
\begin{equation} \label{6.116}
|x-y| \geq \dist(x,\Sigma \sm H) \geq (1-\rho) d(H)
\end{equation}
(see the definition \eqref{6.6}). 
Note that by \eqref{6.113}, \eqref{6.115} and \eqref{6.116} we have
$$
|y_i -x_i |  \leq |y - c_{H}| + 5^{-1} ad(H) \leq |y - x|+2 d(H)
\leq (1+{2 \over 1-\rho})\, |y - x|
$$
By \eqref{6.113} and \eqref{6.114} we have
\begin{align*}
|y_i -x_i |  &\geq {1 \over 2}|y - c_{H}| - 5^{-1} acd(H)
\geq {1 \over 4}|y - c_{H}|+{a \over 4} d(H)
\\
& \geq {a \over 4} (|y - c_{H}|+ d(H)) \geq {a \over 4}|y - x|
\end{align*}

This proves \eqref{6.108} and \eqref{6.105} in the case $ Q\neq R$ and $Q\subset R$. The case 
 when $R \subset Q$ is similar. Thus
we are left with the case where $Q$ and $R$ do not meet.
Then the index $i$ defined above (the generation of the smallest
common ancestor of $Q$ and $R$ in $T$) is smaller than $j$ and $j'$.
This means that both $x_i$ and $y_j$ are as in \eqref{6.107}.

Set $H = Q^{i+1}$ and $H' = R^{i+1}$; these are disjoint sub-cubes
of $Q^i$. By \eqref{6.107} we have that
\begin{equation} \label{6.117}
|x_i - g_{Q^i}(c_{H})| \leq 5^{-1} acd(H)
\ \text{ and } \ 
|y_i - g_{Q^i}(c_{H'})| \leq 5^{-1} acd(H').
\end{equation}
Both points $c_{H}$ and $c_{H'}$ lie in the next generation
of centers (relative to $Q^i$), so they lie in the set $E_{Q^i}$
(either by \eqref{6.94} and because $H$ and $H'$ are children of $Q^i$,
or by \eqref{6.88} and because $H$ and $H'$ are minimal cubes of $S''$, where
$Q^i = Q(S'')$). Since $g_{Q^i}$ is bi-Lipschitz on that set, we get that
\begin{equation} \label{6.118}
{c \over 2} |c_{H'}-c_{H}| \leq |g_{Q^i}(c_{H'}) - g_{Q^i}(c_{H})| 
\leq {2 \over a} |c_{H'}-c_{H}|
\end{equation}
(see \eqref{6.97}). In addition by \eqref{6.90}
\begin{equation} \label{6.119}
 |c_{H'}-c_{H}| 
 \geq {1 \over 2} \big(\dist(c_H,\Sigma \sm H)+ \dist(c_{H'},\Sigma \sm H')\big)
\geq a(d(H)+d(H')).
\end{equation}
Also $x \in Q \subset H$ and $y \in R \subset H'$,
so
\begin{equation} \label{6.120}
|x - c_{H}| + |y - c_{H'}| \leq d(H)+d(H')
\end{equation}
as in \eqref{6.115}, and the fact that $x, y \in A \subset \R^n \sm F_2$
implies that
\begin{equation} \label{6.121}
|x-y| \geq {1 \over 2} \big(\dist(x,\Sigma \sm H)+ \dist(y,\Sigma \sm H')\big)
\geq {1-\rho \over 2}(d(H)+d(H'))
\end{equation}
because $H$ and $H'$ are disjoint, and as in \eqref{6.116}.
Then by \eqref{6.118}, \eqref{6.117}, \eqref{6.119}, \eqref{6.120} and because ${1 \over 2} - {1 \over 10} \geq {1 \over 5}+{1 \over 10}$ we have
\begin{align*}
|y_i -x_i |  
&\geq |g_{Q^i}(c_{H'}) - g_{Q^i}(c_{H})| - |x_i - g_{Q^i}(c_{H})| -|y_i - g_{Q^i}(c_{H'})|
\\
& \geq {c \over 2} |c_{H'}-c_{H}| - 5^{-1} ac(d(H)+d(H'))
\\
& \geq {c \over 10} \big(|c_{H'}-c_{H}| + d(H)+d(H')\big)
\geq {c \over 10} |x-y|.
\end{align*}
Futhermore by \eqref{6.118}, \eqref{6.117}, \eqref{6.120}, and \eqref{6.121}
we have
\begin{align*}
|y_i -x_i |  & \leq |g_{Q^i}(c_{H'}) - g_{Q^i}(c_{H})| 
+ |x_i - g_{Q^i}(c_{H})| + |y_i - g_{Q^i}(c_{H'})|
\\
&\leq {2 \over a} |c_{H'}-c_{H}| + 5^{-1} ac(d(H)+d(H'))
\\
&\leq {2 \over a} \big(|x-y| + 2(d(H)+d(H'))\big)
\leq {2 \over a}\big(1 + {4 \over 1-\rho} \big) \, |x-y|.
\end{align*}
This proves \eqref{6.108} in the only remaining case, \eqref{6.105} follows,
and so does Lemma \ref{t6.9}.
 \end{proof}

\subsection{Density control}
\label{density}

At this point, we have a set $A$ that satisfies 
the geometric conditions in Theorem \ref{t1.6}.
We still need to show that 
$\mu|_{f(A)}$ and $\H^d|_{f(A)}$ are mutually absolutely continuous,
and satisfy \eqref{1.21}.
In order to accomplish this we remove a small piece of $A$. Set
\begin{equation} \label{6.122}
A^\sharp = \Big\{ x\in A \, ; \,  
\int_{0}^{2r} \alpha_{d}(x,t)\frac{dt}{t} 
\leq {2C_1 \over \gamma}
\Big\}.
\end{equation}
It follows at once from \eqref{1.19} and Chebyshev that
$$
\mu(A \sm A^\sharp) \leq {\gamma \mu(B) \over 2}.
$$
Then $A^\sharp$ still satisfies \eqref{1.20}, by \eqref{6.86} and \eqref{6.85}, and it inherits
the geometric properties of $A$. We check to \eqref{1.21}.

Notice that every point $x\in A^\sharp$ satisfies the initial condition $J(x)<\infty$
of Section~\ref{proof2}, and so $x \in \Sigma_0(d)$, the set of Theorem \ref{t1.5}(
see the lines below \eqref{5.5}).
Also, \eqref{5.6} and \eqref{5.7} hold, which imply that the density
$\theta_d(x)$ defined by \eqref{1.15} satisfies
\begin{equation} \label{6.123}
\Big| \log\Big({\theta_d^\ast(x,r_1) \over \theta_d(x)}\Big)\Big|
\leq C \sum_{l \geq 1} \alpha_d(x,r_l),
\end{equation}
where the radii $r_l = r_l(x)$ are defined in \eqref{5.2} and \eqref{5.3}
and 
$$
\theta_d^\ast(x,r_1) = r_1^{-d} c_{d}(x,r_1) \mu(B(x,r_1))
$$
is defined in \eqref{3.14}. Since 
$r_1 \in [1/4,1/2]$, $x\in \Sigma \cap B$, $\mu$ is doubling,
and $1 \leq c_{d}(x,r_1) \leq 2^d$ by \eqref{2.14} or \eqref{1.8}
$$
C^{-1} \mu(B) \leq \theta_d^\ast(x,r_1) \leq C \mu(B).
$$
By \eqref{5.4}, \eqref{5.1}, and the definition of $A^\sharp$ we have
$$
\sum_{l \geq 1} \alpha_d(x,r_l) \leq 2 J(x) = 2\int_{0}^{1} \alpha_{d}(x,t)\frac{dt}{t} 
\leq {4C_1 \over \gamma}.
$$
Thus
$$
C^{-1} \mu(B) \leq \theta_d(x) \leq C \mu(B)
\ \text{ for } x\in A^\sharp,
$$
for some large constant $C$ that depends on $C_1$. This proves that
$A^\sharp$ in contained in a finite union of sets $\Sigma_0(d,k)$
as in \eqref{1.16}, and Theorem \ref{t1.5} says that $\mu$ and 
$\H^d$ are mutually absolutely continuous on $A^\sharp$,
$\mu = \theta_d \H^d$ there. The estimate \eqref{1.21} follows at once
(recall that $r=1$ here). This completes the proof of Theorem \ref{t1.6}.

\section{Proof of the uniform rectifiability result - Theorem \ref{t1.8}}
\label{UR}

In this section we prove Theorem \ref{t1.8}, the uniform version of
Theorem \ref{t1.7}. Let $\mu$ and $\Sigma$ be as in the statement of Theorem \ref{t1.8}, and choose $C_1$ so small
that we can apply Theorem \ref{t1.7}, with a constant $\gamma > 0$ that will
be chosen soon. 

We saw at the beginning of Section \ref{proof2} that 
$\mu$-almost every point $x$ of $\Sigma$ lies in the good set $\Sigma_0(d)$
of Theorem \ref{t1.5}, and hence the density 
$\theta_h(x) = \lim_{r \to 0} r^{-d} \mu(B(x,r))$ exists (and is finite) (
see argument above \eqref{5.6}).

\begin{lemma} \label{t7.1}
There is a constant $C \geq 0$, that depends only $d$ and $\CD$, such that
$\log\, \theta_d \in \BMO(\mu)$, with
\begin{equation} \label{7.1}
|| \log\,\theta_d ||_{BMO} \leq C C_1^\eta.
\end{equation}
\end{lemma}

\begin{proof}  
Here $C_1$ still denotes the constant from \eqref{1.19}, 
and $\eta > 0$ is the same constant as in Lemma \ref{t2.6}. Thus we can make $C C_1^\eta$ as small as we want.
We present two slightly different proofs: one exploits and adapts the computations performed in Section \ref{proof2}, the other
uses a variant of John-Nirenberg's arguments.

To prove \eqref{7.1} we need to show that for any $x\in \Sigma$ and any $r > 0$,
\begin{equation} \label{7.2}
\fint_{B(x,r)} \big |\log\,\theta_d(y) - m(x,r) \big| d\mu(y) \leq CC_1^\eta,
\end{equation}
where $m(x,r) = \fint_{B(x,r)}  \log\, \theta_d$.

By rotation and dilation invariance, we may assume that
$B(x,r) = B(0,1/2)$ (as we did in Section \ref{proof2}).
Recall from \eqref{5.19} that for $\mu$-almost every $y\in B(0,1/2)$,
\begin{equation} \label{7.3}
\Big| \log\Big({\theta_d(y) \over \theta_d^\ast(y,r_5(y))}\Big)\Big|
\leq C \sum_{l \geq 5} \alpha_d(y,r_l(y)) \leq C J(y),
\end{equation}
where $J(y) = \int_0^1 \alpha_d(y,t) {dt \over t}$ as in \eqref{5.1} 
and $\theta_d^\ast$ comes from \eqref{3.14}. 
Note that \eqref{5.20} and \eqref{5.21} only use the fact that $J(y) < \infty$.
Therefore combining \eqref{5.20} , \eqref{5.21} and \eqref{7.3} we obtain that 
for $\mu$almost every $y\in B(0,1/2)$
\begin{equation} \label{7.4}
\Big|\log\Big({\theta_d(y) \over 2^d\mu(B)} \Big)\Big| \leq 
C (\gamma^{-1} C_1)^{\eta}+ C J(y).
\end{equation} 
Let 
$a = \log(2^d\mu(B))$. Note that using \eqref{7.4} and hypothesis \eqref{1.19} we obtain
\begin{eqnarray} \label{7.5}
\fint_ {B(0,1/2)} \big |\log\,\theta_d(y) - a \big| d\mu(y) 
&\leq& C   (\gamma^{-1} C_1)^{\eta}+  C\fint_{B(0,1/2)} J(y) dy
\nonumber\\
&&\hskip-2cm
\leq C (\gamma^{-1} C_1)^{\eta} + C C_1 \leq C (\gamma^{-1} C_1)^{\eta}.
\end{eqnarray}
The fact that we can replace $a$ by $m(x,r)$ in \eqref{7.5} to get \eqref{7.2} uses the triangle inequality 
and
is a standard trick she working in $\BMO$.

An alternative to this proof notices that \eqref{5.22} ensures that
the function $f = \log\, \theta_d$ is $C (\gamma^{-1} C_1)^{\eta}$-close to a constant
on the set $A \subset \Sigma \cap B$. Since by \eqref{1.20}
$\mu(B(x,r) \sm A) \leq \gamma \mu(B(x,r)$ for some constant
$\gamma$ which is as small as we want, the lemma follows by a known variant
by J.-O. Str\"omberg of the proof of John and Nirenberg's theorem on $\BMO$. 
The main remark is that the
proof of John - Nirenberg's theorem that can be found in \cite{Journe}, page 32, 
with the standard cubes in $\R^n$ can be carried out in this setting 
with the pseudo-cubes of Subsection \ref{cubes}. It yields John - Nirenberg's theorem
for doubling measures.
We claim that the
proof of John and Nirenberg's theorem that can be found in \cite{Journe}, page 32, 
with the standard cubes in $\R^n$ replaced
with the pseudo-cubes of Subsection \ref{cubes}, works for this and
also gives a proof of John and Nirenberg's result for doubling measures that will be used soon.
To implement Stromberg's argument, we first truncate the function (in case it was not locally integrable),
replace $m_Q f$ and $m_{Q^\ast}f$ (the averages of $f$ in $Q$ and $Q^\ast$) by the constant associated to a cube and its 
parent (the constant $f$ is close to in a large set of $Q$ and $Q^\ast$). Notice that the difference is small because $\gamma$ is small.
This completes the second proof of Lemma \ref{t7.1}.
\end{proof}

\ms
A standard fact about $\BMO$ functions is that if $f \in \BMO$
and $\|f\|_{BMO}$ is small enough, then $e^{\pm f}$ is locally integrable
(John and Nirenberg's). In fact $f$ is an $A_p$-weight
 for any $p > 1$. This  implies that $e^f\in A_\infty$. The fact that this is valid for doubling measures
can be seen by noting that the proofs in either \cite{Journe} or
\cite{GarciaCuerva} can be implemented in this setting using pseudo-cubes
instead of regular cubes. Thus
\begin{equation} \label{7.6}
\theta_d^{-1} \in A_{\infty}(d\mu).
\end{equation}

More explicitly, Theorem \ref{t1.5} gives sets $\Sigma_0(d,k)$
that cover $\mu$-almost all $\Sigma$, and on which $\H^d$ and $\mu$
are mutually absolutely continuous with density bounded above and below.
On $\Sigma_0 = \cup_k \Sigma_0(d,k)$, 
$\H^d$ and $\mu$ are still mutually absolutely continuous, with
\begin{equation} \label{7.7}
\mu|_{\Sigma_0} = \theta_d \H^d|_{\Sigma_0}
\ \text{ and } \ 
\H^d|_{\Sigma_0} = \theta_d^{-1} \mu|_{\Sigma_0},
\end{equation}
where we use the fact that
$\theta_d^{-1} \in L^1_{loc}(\mu)$ (see \eqref{7.6}). The next step is to show that
\begin{equation} \label{7.8}
\H^d(\Sigma \sm \Sigma_0) = 0.
\end{equation}
We do so using a density argument.
 Recall $\alpha_d(x,r)$ is a Borel function of $(x,r) \in \Sigma \times (0,\infty)$ (see Remark \ref{tt-r2}). Therefore since
$\Sigma_0 = \big\{ x\in \Sigma \, ; \,
\int_0^1 \alpha(x,r) {dr \over r}< \infty \big\}$ (see\eqref{3.1}) $\Sigma_0$ is Borel measurable.
Thus applying Part (2) of Theorem 6.2 in \cite{Mattila} to the set $S = \Sigma_0 \cap B(0,R)$
for any $R > 0$; and using \eqref{7.7} and the fact that $\theta_d^{-1} \in L^1_{loc}(\mu)$ we have
$$
\H^d(S) = \int_S \theta_d^{-1} d\mu \leq \int_{B(0,R)} \theta_d^{-1} d\mu <\infty.
$$
Theorem 6.2 in \cite{Mattila}
guarantees that 
\begin{equation} \label{7.10}
\limsup_{r \to 0} r^{-d} \H^d(S \cap B(x,r)) = 0
\end{equation}
for $\H^d$-almost every point of $\R^n \sm S$.
Thus if \eqref{7.8} fails, we can find $R > 0$ and $x\in \Sigma \cap B(0,R)$
such that \eqref{7.10} holds.
But \eqref{7.10} fails for every $x\in \Sigma \cap B(0,R)$, because 
as soon as $B(x,r) \subset B(0,R)$, the proof of Theorem \ref{t1.7} gives a big piece of
Lipschitz graph $A \subset \Sigma_0 \cap B(x,r)$, 
with $\H^d(A) \geq C^{-1} r^d$ (by \eqref{1.20} and \eqref{1.21}). 
This contradiction shows that \eqref{7.8} holds.

Once we have \eqref{7.8}, \eqref{7.7} says that that the two measures 
$\H^d|_{\Sigma}$ and $\mu$ are mutually absolutely continuous with respect
to each other, with $A_{\infty}$ densities, and the rest of Theorem \ref{t1.8} follows easily. 
Let us first show that $\Sigma$ is Ahlfors regular.

Set $f = \log\,\theta_d$ and, for $x\in \Sigma$ and $r > 0$,
set $m(x,r) = \fint_{B(x,r)} f$. Since $f\in \BMO$ with a small
norm, John and Nirenberg's theorem guarantees that
\begin{equation} \label{7.11}
\fint_{B(x,r)} e^{|f-m(x,r)|} d\mu \leq C.
\end{equation}
We evaluate $m(x,r)$, using either \eqref{1.21} or more directly \eqref{5.22} which ensures 
that for $y\in A\subset B(x,r)$
$\big|f(y) - \log\big(r^{-d}\mu(B(x,r))\big)\big| \leq 1$. This proves that
$\big|m(x,r) - \log\big(r^{-d}\mu(B(x,r))\big)\big| \leq 1$, that is
\begin{equation} \label{7.12}
C^{-1} r^{-d}\mu(B(x,r)) \leq e^{m(x,r)} \leq C r^{-d}\mu(B(x,r)).
\end{equation}
By \eqref{7.8}, \eqref{7.7}, \eqref{7.11}, and \eqref{7.12} we have
\begin{eqnarray} \label{7.13}
\H^d(\Sigma \cap B(x,r)) &=& \H^d(\Sigma_0 \cap B(x,r))
= \int_{\Sigma_0 \cap B(x,r)} \theta_d^{-1} d\mu 
\nonumber\\
&=& \int_{B(x,r)} e^{-f} d\mu  = e^{-m(x,r)} \int_{B(x,r)} e^{[f-m(x,r)]} d\mu
\nonumber\\
&\leq &  C e^{-m(x,r)} \mu(B(x,r)) \leq C r^d.
\end{eqnarray}
 Similarly,
\begin{eqnarray} \label{7.14}
\H^d(\Sigma \cap B(x,r)) &=& \int_{B(x,r)} e^{-f} d\mu
= \mu(B(x,r)) \fint_{B(x,r)} e^{-f} d\mu
\nonumber\\
&\geq& e^{-m(x,r)} \mu(B(x,r)) \geq C^{-1} r^d,
\end{eqnarray}
where this time the main step involves Jensen's inequality.

Hence $\Sigma$ is  $d$-Ahlfors regular and Theorem \ref{t1.7} ensures that it contains 
big pieces of Lipschitz graphs. 
This completes our proof of Theorem \ref{t1.8}.
\qed

\bibliographystyle{amsplain}

\providecommand{\bysame}{\leavevmode\hbox to3em{\hrulefill}\thinspace}
\providecommand{\MR}{\relax\ifhmode\unskip\space\fi MR }
\providecommand{\MRhref}[2]{
\href{http://www.ams.org/mathscinet-getitem?mr=#1}{#2}
}
\providecommand{\href}[2]{#2}

\noindent {Jonas Azzam: Universitat Aut\`{o}noma de Barcelona,
Departament de Matem\`{a}tiques,
08193 Bellaterra (Barcelona).
Email: JonasAziz.Azzam@uab.cat}\\

\noindent {Guy David: Universit\'e Paris-Sud, Laboratoire de Math\'{e}matiques, 
UMR 8658 Orsay, F-91405
CNRS, Orsay, F-91405. Email: guy.david@math.u-psud.fr}\\

\noindent {Tatiana Toro: University of Washington, 
Department of Mathematics,
Seattle, WA 98195-4350. Email: 
toro@math.washington.edu}

\end{document}